\documentclass[11pt,oneside]{amsart}

\usepackage{graphicx}
\usepackage{amsfonts}
\usepackage{epsf}
\usepackage{amssymb}
\usepackage{amsmath}
\usepackage{amscd}
\usepackage[usenames,dvipsnames]{color}
\usepackage{tikz}
\usepackage{pdfpages}
\usepackage{fancyhdr}
\usepackage{setspace}
\usepackage{hyperref}
\usepackage[all]{xy}
\usetikzlibrary{matrix}
\usepackage{verbatim}
\usepackage{amsthm}

\theoremstyle{theorem}
\newtheorem{theorem}{Theorem}[section]
\newtheorem{proposition}[theorem]{Proposition}
\newtheorem{lemma}[theorem]{Lemma}
\newtheorem{definition}[theorem]{Definition}
\newtheorem{question}[theorem]{Question}
\newtheorem{corollary}[theorem]{Corollary}

\newtheorem{example}[theorem]{Example}

\newtheorem{lettertheorem}{Theorem}

\makeatletter
\newtheorem{letterconjecture}[lettertheorem]{Conjecture}

\makeatletter
\newtheorem{lettercorollary}[lettertheorem]{Corollary}

\makeatletter

\newtheorem*{rep@theorem}{\rep@title}
\newcommand{\newreptheorem}[2]{%
\newenvironment{rep#1}[1]{%
 \def\rep@title{#2 \ref{##1}}%
 \begin{rep@theorem}}%
 {\end{rep@theorem}}}
\makeatother

\newreptheorem{theorem}{Theorem}
\newreptheorem{lemma}{Lemma}
\newreptheorem{question}{Question}

\numberwithin{figure}{section}

\newcommand{\F}{\mathbb{F}}

\newcommand{\Z}{\mathbb{Z}}

\newcommand{\N}{\mathbb{N}}
\newcommand{\Q}{\mathbb{Q}}

\newcommand{\spinc}{\text{Spin}^c}
\newcommand{\HFhat}{\widehat{HF}}
\newcommand{\CFhat}{\widehat{CF}}
\newcommand{\im}{\text{Im}}

\title{Distinguishing topologically and smoothly doubly slice knots}

\author{Jeffrey Meier} 
\address{Department of Mathematics\\Indiana University\\
         Bloomington, IN 47408\\USA}
\email{jlmeier@indiana.edu}
\urladdr{http://pages.iu.edu/~jlmeier/home.html}

\begin{document}

\begin{abstract}

We construct an infinite family of smoothly slice knots that we prove are topologically doubly slice.  Using the correction terms coming from Heegaard Floer homology, we show that none of these knots is smoothly doubly slice.  We use these knots to show that the subgroup of the double concordance group consisting of smoothly slice, topologically doubly slice knots is infinitely generated.  As a corollary, we produce an infinite collection of rational homology 3--spheres that embed in $S^4$ topologically, but not smoothly.
\end{abstract}

\maketitle

\rhead{\thepage}
\lhead{\author}
\thispagestyle{empty}

\raggedbottom

\pagenumbering{arabic}

\setcounter{section}{0}


\section{introduction}\label{section:introduction}

A knot $K$ in $S^3$ is called \emph{smoothly doubly slice} if there exists a smoothly embedded, unknotted 2--sphere $\kappa$ in $S^4$ such that $\kappa\cap S^3=K$.  Analogously, $K$ is called \emph{topologically doubly slice} if $\kappa$ is topologically locally flat.  The question of which slice knots are doubly slice was first posed by Fox in 1961 \cite{fox:problems}, and Zeeman showed that $K\#(-K)$ is always doubly slice \cite{zeeman}.  Work of Sumners encapsulates what was known up to about 1970 \cite{sumners:invertible}.  In particular, he gave necessary algebraic conditions for a knot to be doubly slice and proved that $9_{46}$ is the only doubly slice knot up to 9 crossings.  Although his proof that $9_{46}$ is doubly slice is (necessarily) geometric in nature, his obstruction methods are actually purely algebraic.  He showed that $9_{46}$ is the only knot up to 9 crossings that is algebraically doubly slice.  A knot $K$ is called \emph{algebraically doubly slice} if there exists an invertible $\Z$--valued matrix $P$ such that
$$PA_KP^\tau = \begin{bmatrix}
 0 & B_1 \\
 B_2 & 0 
 \end{bmatrix},$$
where $A_K$ is a Seifert matrix for $K$, and $B_1$ and $B_2$ are square matrices of equal dimension.  Matrices of this form are often called \emph{hyperbolic}, and have been studied by Levine \cite{levine:hyperbolic}.  
We remark that all these concepts  generalize to higher dimensions (see, for example \cite{sumners:invertible}), but we will restrict our attention to the classical dimension.

Since the work of Sumners, there have been three major geometric developments in the theory, all in the topologically locally flat category. In what follows, we will take `slice' and `doubly slice' to mean `topologically slice' and `topologically doubly slice' and clarify the category when necessary or helpful.

First, in 1983, Gilmer-Livingston showed, using Casson-Gordon invariants, that there exist slice knots that are algebraically doubly slice, but not doubly slice \cite{gilmer-livingston:embedding}.  Second, about 10 years ago, Kim \cite{kim:new} extended the bi-filtration technology introduced by Cochran-Orr-Teichner in \cite{COT} to the class of topologically doubly slice knots.  At the same time, Friedl \cite{friedl:eta} showed that certain $\eta$--invariants coming from metabelian representations $\pi_1(M_K)\to U(k)$, where $M_K$ denotes 0--surgery on $K$, can be used to obstruct double sliceness.

In this paper, the invariants used are the correction terms coming from Heegaard Floer homology (see \cite{oz-sz:absolute}).  These are smooth manifold invariants, so they are well suited to distinguish the smooth and topologically locally flat categories. A second property these invariants enjoy is the fact that, while they can be used to obstruct smooth sliceness, they do not completely vanish for smoothly slice knots, as do invariants such as the signature, $\tau$--invariant, or $s$--invariant.  In other words, they encode enough information to distinguish smooth double sliceness and smooth sliceness.  The main result of the present paper is the following.

\begin{lettertheorem}\label{thm:main}
There exists an infinite family of smoothly slice knots that are topologically doubly slice, but not smoothly doubly slice.
\end{lettertheorem}





Recall that two knots $K_0$ and $K_1$ are said to be \emph{concordant} if $K_1\#(-K_2)$ is slice (where $-K$ denotes the mirror reverse of $K$) or, equivalently, if there exists a properly embedded cylinder $C\subset S^3\times I$ such that $C\cap S^3\times\{i\}=K_i$ for $i=0,1$.  If $K_0$ and $K_1$ are concordant, we write $K_0\sim K_1$.  Concordance can be studied in either the smooth or the topologically locally flat categories and induces (different) equivalence relations therein.  Let $\mathcal C$ denote the set of knots in $S^3$ up to smooth concordance.  Under connected sum, $\mathcal C$ inherits an abelian group structure and is called the smooth \emph{concordance group}.  Similarly, one can define the topological concordance group $\mathcal C^{top}$ and the algebraic concordance group $\mathcal G$. There exist surjective homomorphisms
$$\mathcal C\stackrel{\psi}{\longrightarrow} \mathcal C^{top}\stackrel{\phi}{\longrightarrow}\mathcal G.$$

These groups have received a large amount of attention, and many interesting theorems and examples have expanded our understanding of their nature; however, there remain many open problems.  For example, it is still not known whether or not $\mathcal C$ and $\mathcal C^{top}$ contain elements of finite order greater than two. On the other hand, Levine \cite{levine:invariants,levine:groups} proved that
$$\mathcal G\cong \Z^\infty\oplus\Z_2^\infty\oplus\Z^\infty_4.$$
For an excellent survey, see \cite{livingston:survey}.

It would be natural to define $K_0$ and $K_1$ to be doubly concordant if $K_0\#(-K_1)$ is doubly slice.  However, it is not known whether this gives an equivalence relation.  The issue is the following unsolved problem.

\begin{question}\label{question:cancelation}
Suppose that $K$ is doubly slice and that $J\#K$ is doubly slice.  Then, must $J$ be doubly slice?
\end{question}

Without an affirmative answer to Question \ref{question:cancelation}, one cannot prove transitivity of the desired equivalence relation.  
Following \cite{stoltzfus:double}, we say that $J$ is \emph{stably doubly slice} if $J\#K$ is doubly slice for some doubly slice knot $K$.  Then, Question \ref{question:cancelation} is simply asking whether or not there exist stably doubly slice knots that are not doubly slice.
Because of these difficulties, we must adopt a different definition of doubly concordant.

Recall that two knots $K_0$ and $K_1$ are concordant if and only if there exist two slice knots $J_0$ and $J_1$ such that $K_0\#J_0=K_1\#J_1$.  This follows from the more common definition of concordant by realizing that the analogue of Question \ref{question:cancelation} for slice knots is true:  If $K$ is slice and $J\#K$ is slice, then $J$ is slice.  With this in mind, we adopt the following definition.

\begin{definition}
Two knots $K_0$ and $K_1$ are smoothly \emph{doubly concordant} if there exist smoothly doubly slice knots $J_0$ and $J_1$ such that $K_0\#J_0=K_1\#J_1$.  We write $K_0\stackrel{\mathcal D}{\sim}K_1$.
\end{definition}

It is straightforward to verify that $\stackrel{\mathcal D}{\sim}$ is an equivalence relation.  We let $\mathcal C_\mathcal D$ denote the set of knots in $S^3$ modulo this relation, which inherits an abelian group structure under connected sum and  is called the smooth \emph{double concordance group}.  Analogously, we can define the topological double concordance group $\mathcal C_\mathcal D^{top}$ and the algebraic double concordance group $\mathcal G_\mathcal D$, and we have surjective homomorphisms
$$\mathcal C_\mathcal D\stackrel{\psi_\mathcal D}{\longrightarrow} \mathcal C^{top}_\mathcal D\stackrel{\phi_\mathcal D}{\longrightarrow}\mathcal G_\mathcal D.$$

The study of these structures is complicated by Question \ref{question:cancelation}.  In Subsection \ref{subsec:question}, we show that, under certain conditions, if $K$ is smoothly stably doubly slice, then the correction terms of $\Sigma_2(K)$ must vanish in the same way as when $K$ is smoothly doubly slice.   In this light, one consequence of Theorem \ref{thm:main} is that $\mathcal T_\mathcal D\not=0$, where $\mathcal T_\mathcal D = \ker(\psi_\mathcal D)$.

In \cite{GRS}, Grigsby, Ruberman, and Strle (building on work of Jabuka and Naik \cite{jabuka-naik:order}) defined invariants that can be used to obstruct a knot from having finite order in $\mathcal C$.  After a slight modification, we show that similar invariants can be applied to $\mathcal C_\mathcal D$.  After restricting our attention to a certain subfamily of the knots from Theorem \ref{thm:main}, we are able to show the following.

\begin{lettertheorem}\label{thm:main2}
There is an infinitely generated subgroup $\mathcal S$ inside $\mathcal T_\mathcal D$, generated by smoothly slice knots whose order in $\mathcal C_\mathcal D$ is at least three.
\end{lettertheorem}

One would like to say that the knots in $\mathcal S$ have infinite order in $\mathcal C_\mathcal D$.  Unfortunately, due to Question \ref{question:cancelation}, we can only obstruct order one and order two.

\begin{letterconjecture}\label{conj:infty}
The subgroup $\mathcal S\subset \mathcal T_\mathcal D$ is isomorphic to $\Z^\infty$.
\end{letterconjecture}

We have the following corollary to Theorem \ref{thm:main}.  

\begin{lettercorollary}\label{coro:embed}
There exists an infinite family of rational homology 3--spheres that embed in $S^4$ topologically, but not smoothly.
\end{lettercorollary}

Note that these manifolds are not integral homology spheres. An affirmative answer to Question \ref{question:cancelation} would imply Conjecture \ref{conj:infty}.  If the Conjecture \ref{conj:infty} is false, then there are knots in $\mathcal S$ whose branched double covers do not smoothly embed in $S^4$, but do stably embed smoothly in $S^4$.  See \cite{budney-burton:embeddings} for a survey concerning 3--manifold embeddings in $S^4$.



\subsection*{Organization}
In Section \ref{section:background}, we give a brief outline of the proofs of Theorems \ref{thm:main} and \ref{thm:main2} and give a background overview of the relevant theories.  In Section \ref{section:geometry}, we give the construction of the pertinent family of knots and prove that they are topologically doubly slice.  We also introduce and discuss the 3--manifolds and 4--dimensional cobordisms that are used in the proof of Theorem \ref{thm:main}, discuss the sub-family of knots used to prove Theorem \ref{thm:main2}, and address the subtlety of Question \ref{question:cancelation}.  In Section \ref{section:HF}, we recall the pertinent aspects of Heegaard Floer theory.  In Section \ref{section:calculations}, we perform the calculations necessary to prove that the knots are not smoothly doubly slice.  In Section \ref{section:infinite_order}, we use invariants introduced by Grigsby, Ruberman, and Strle to prove Theorem \ref{thm:main2}.  The proofs of the main theorems rely on calculations of  the knot Floer complexes for certain torus knots and the positive, untwisted Whitehead double of the right-handed trefoil. These facts, some of which are found in \cite{HKL},  are presented in Appendix \ref{appendix}.

\subsection*{Acknowledgements}
The author owes a great deal of gratitude to \c{C}a\u{g}ri Karakurt and Tye Lidman, who generously shared their insight and knowledge of Heegaard Floer homology on numerous occasions and whose comments and ideas throughout this project were invaluable.  The author would also like to thank his advisor, Cameron McA. Gordon, for his continued support and guidance and for freely sharing his expertise and comprehensive knowledge of all things knot theoretical.

\begin{figure}
\centering
\includegraphics[scale = .75]{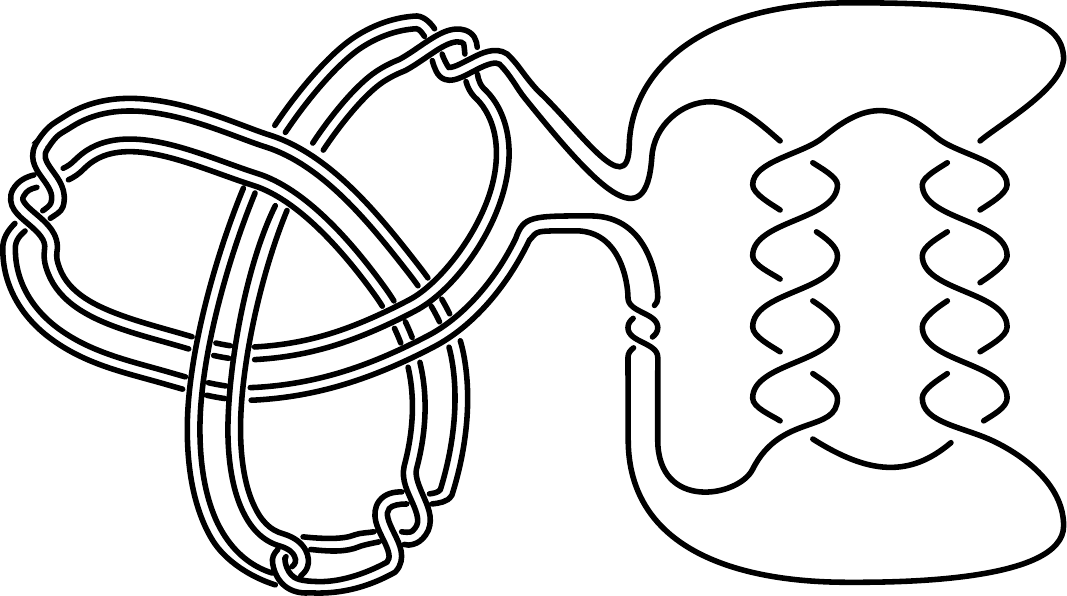}
\caption{One member of the family $\mathcal K_p$; here, $p=5$.}
\label{fig:InfectedKnotSmall}
\end{figure}


\section{Background and outline of proof}\label{section:background}

In Section \ref{section:geometry}, we construct the knots $\mathcal K_p$ for odd primes $p$, and prove that they are topologically doubly slice. 
(See Figure \ref{fig:InfectedKnotSmall} for an example.)

The most difficult task of this paper is showing that the $\mathcal K_p$ are not smoothly doubly slice.  This is accomplished by studying the double covers of $S^3$ branched along these knots.  If $K$ is a smoothly doubly slice knot, then it is the intersection of a smoothly unknotted 2--sphere $\kappa\subset S^4$ with the standard $S^3\subset S^4$.  
So we have $(S^3,K)\subset (S^4,\kappa)$, where the first pair sits as the equator of the second.  Taking the branched double cover, we get $(\Sigma_2(K),K)\subset (S^4,\kappa)$.  This gives a smooth embedding of the branched double cover $\Sigma_2(K)$ of $K$ into $S^4$.  We have proved the following proposition, which first appeared in \cite{gilmer-livingston:embedding}.

\begin{proposition}
If $K$ is a smoothly doubly slice knot, then $\Sigma_2(K)$ embeds smoothly into $S^4$.
\end{proposition}

Thus, we can prove that a knot is not smoothly doubly slice by showing that its branched double cover does not embed smoothly in $S^4$.  To do this, we make use of the correction terms coming from Heegaard Floer homology.  For more details, see Section \ref{section:HF}.  For now, let $M$ denote a closed 3--manifold, and let $\frak s\in\spinc(M)$.  Let $d(M,\frak s)$ denote the correction term associated to the pair $(M,\frak s)$.  The main tool in this paper is the following theorem, which also appears in \cite{donald:embedding} and \cite{gilmer-livingston:embedding} in one form or another.

\begin{theorem}\label{thm:hyp_corr_terms}
Let $M$ be a rational homology 3--sphere that embeds smoothly in $S^4$.  Then $H_1(M)= G_1\oplus G_2$ with $G_1\cong G_2$.  Furthermore, there is an identification $\spinc(M)\cong H^2(M;\Z)\cong H_1(M)$ such that $$d(M,\frak s)=0\hspace{.25in} \forall \frak s\in G_1\cup G_2.$$
In other words, if $|H_1(M)|=n^2$, then at least $2n-1$ of the $n^2$ correction terms associated to $M$ must vanish.
\end{theorem}

\begin{proof}
Since $M$ embeds smoothly in $S^4$, we get a decomposition $S^4 = U_1\cup_MU_2$, where $U_i$ is a rational homology 4--ball for $i=1,2$.  Let $G_i = H_1(U_i)$ for $i=1,2$.  By analyzing the Mayer-Vietoris sequence induced by this decomposition, we see that $H_1(M)\cong H_1(U_1)\oplus H_1(U_2)=G_1\oplus G_2$.  The proof that $G_1\cong G_2$ is due to Hantzsche \cite{hantzche}, and is as follows.  By analyzing the relative sequence for $(S^4,U_1)$, we see that $H_1(U_1)\cong H_2(S^4,U_1)$.  By excision, $H_2(S^4,U_1)\cong H_2(U_2,M)$, and by Lefschetz duality, $H_2(U_2,M)\cong H^2(U_2)$.  Finally, by the universal coefficients theorem, $H^2(U_2)\cong H_1(U_2)$ (since $H_1(U_2)$ and $H_2(U_2)$ are both torsion).

Now consider the dual isomorphism $G_1\oplus G_2\cong H^2(M)$, whose restrictions to $G_i$ are induced by the inclusion $M\hookrightarrow U_i$ for $i=1,2$.  Elements in $H^2(M)$ that are in the image of this inclusion from $G_i$ correspond to $\spinc$ structures on $M$ that extend to $\spinc$ structures over $U_i$ for $i=1,2$.  However, for any 3--manifold $Y$ and $\frak s\in\spinc(Y)$, we have that $d(Y,\frak s) = 0$ whenever $(Y,\frak s) = \partial (W,\frak t)$, where $W$ is a rational homology 4--ball and $\frak t$ extends $\frak s$ (see \cite{oz-sz:absolute}).

If follows that $d(M,\frak s)=0$ for any $\frak s\in G_1\cup G_2$, which is a set of cardinality $2n-1$.
\end{proof}

Let $\mathcal Z_p$ denote the double cover of $S^3$ branched along the knot $\mathcal K_p$.  In Section \ref{section:calculations}, we make use of the surgery exact triangle to relate the Heegaard Floer homology of $\mathcal Z_p$ to that of simpler manifolds (manifolds obtained as surgery on knots in $S^3$, to be precise).  Using this set-up, we show in Corollary \ref{coro:terms} that only $2p-3$ of the $p^2$ correction terms associated to $\mathcal Z_p$ vanish.  By Theorem \ref{thm:hyp_corr_terms}, this implies Theorem \ref{thm:main}, as well as Corollary \ref{coro:embed}.

Of course, the statement that at least $2n-1$ of the $n^2$ correction terms must vanish does not use the full strength of Theorem \ref{thm:hyp_corr_terms}, since it makes no use of the group structure of the correction terms.  Jabuka and Naik \cite{jabuka-naik:order} used this group structure to prove that many low crossing knots (whose concordance order was unknown) are not order 4 in $\mathcal C$.  Grigsby, Ruberman, and Strle investigated this concept further in \cite{GRS}, and introduced knot invariants that can be used to obstruct finite concordance order among knots.  We refine one set of these invariants so that they can be used to obstruct order one and order two in the double concordance group, and use them to prove that a family related to the $\mathcal K_p$ generates an infinite rank subgroup in $\mathcal C_\mathcal D$ (see Section \ref{section:infinite_order}).  This proves Theorem \ref{thm:main2}.


\section{Geometric considerations}\label{section:geometry}

In this section, we use the method of infection to construct the knots $\mathcal K_p$ and $\mathcal K_{p,k}$.  We then describe a sufficient condition for a knot to be doubly slice and use it to prove that these knots are topologically doubly slice.  Next, we introduce the 3--manifolds triad that will be used in Section \ref{section:calculations}, and describe the  4--dimensional cobordisms relating them.  Finally, we address Question \ref{question:cancelation}.

\subsection{Infection and the knots $\mathcal K_p$}\label{subsec:infection}\  

Let $\vec\eta = (\eta_1, \ldots, \eta_n)$ be an $n$--component unlink in $S^3$, and choose an open tubular neighborhood $N_i$ of each $\eta_i$ such that $\overline N_i\cap \overline N_j=\emptyset$ for $i\not=j$.  Let $E = S^3 - \cup_{i=1}^n N_i$.  Next, consider a collection of knots $\vec J=(J_1,\ldots, J_n)$, and let $E_{J_i}$ denote the exterior of $J_i$.  Let $M$ be the manifold obtained by gluing $E_{J_i}$ to $E$ along $\partial N_i$ such that the meridian and longitude of $\eta_i$ are identified with the longitude and meridian, respectively, of $J_i$.  This choice of gluing ensures that $M$ is diffeomorphic to $S^3$.

Let $K\subset E$, and let $f:E\to M$ be the natural inclusion.  Then the knot $K_{\vec\eta}(\vec J) = f(K)$ is the result of \emph{infection} on $K$ by $\vec J$ along $\vec \eta$.  In the case when $\vec\eta$ is a knot, we simply write $I_\eta(J)$.  See Figure \ref{fig:InfectedKnots}.  The construction, as given, dates back at least as far as \cite{gilmer:slice}.



\begin{example}\ 
\begin{enumerate}
\item If $n=1$, we recover the satellite construction.  In particular, if $\eta$ is chosen to be a meridian of $K$, then infection of $K$ by $J$ along $\eta$ is simply $K\#J$.
\item If $K\cup\eta$ is the positive Whitehead link (see Figure \ref{fig:WhiteheadAndTrefoil} (b)), then infection of $K$ by $J$ along $\eta$ is the positive, untwisted Whitehead double of $J$, which we denote by $Wh^+(J,0)$.  For example, if $J$ is the right-handed trefoil, then $Wh^+(J,0)$ is shown in Figure \ref{fig:WhiteheadAndTrefoil} (c).
\end{enumerate}
\end{example}

\begin{figure}
\centering
\includegraphics[scale = .75]{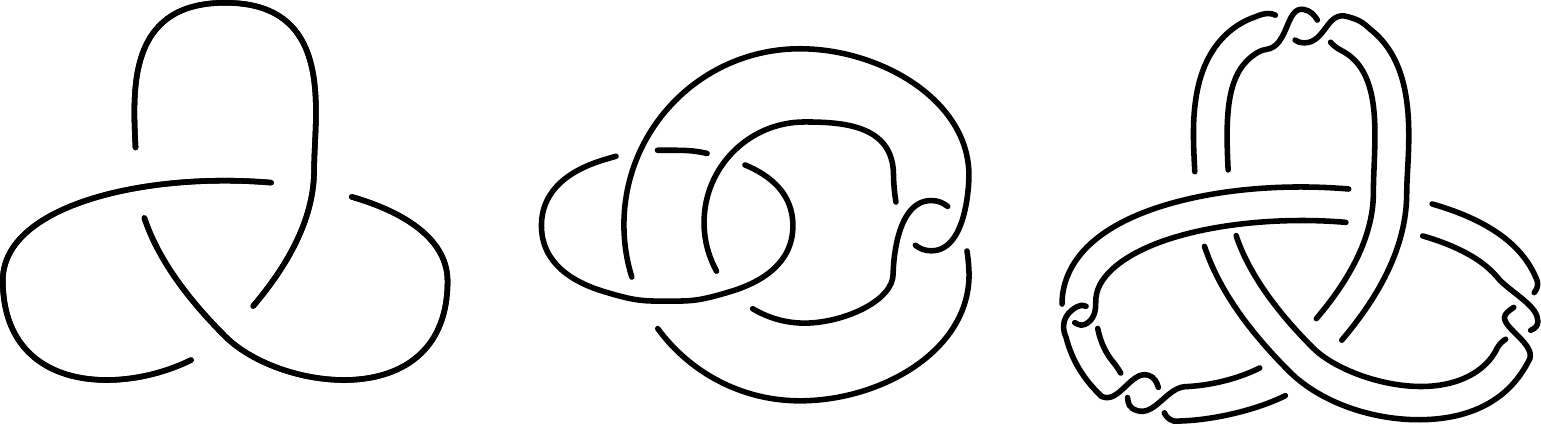}
\put(-293,-15){(a)}
\put(-167,-15){(b)}
\put(-58,-15){(c)}
\caption{(a) The right-handed trefoil, (b) the positive Whitehead link, and (c) the positive, untwisted Whitehead double of the right-handed trefoil.}
\label{fig:WhiteheadAndTrefoil}
\end{figure}

Throughout, we will denote the $(p,q)$--torus knot by $T_{p,q}$ for $2\leq p<|q|$ (see Figure \ref{fig:TorusKnot}).

Let $I_{J,p}$ denote the knot obtained by infecting $T_{2,p}\#(T_{2,-p})$ with $J$ along $\eta$ (see Figure \ref{fig:InfectedKnots}).  
\begin{figure}
\centering
\includegraphics[scale = .75]{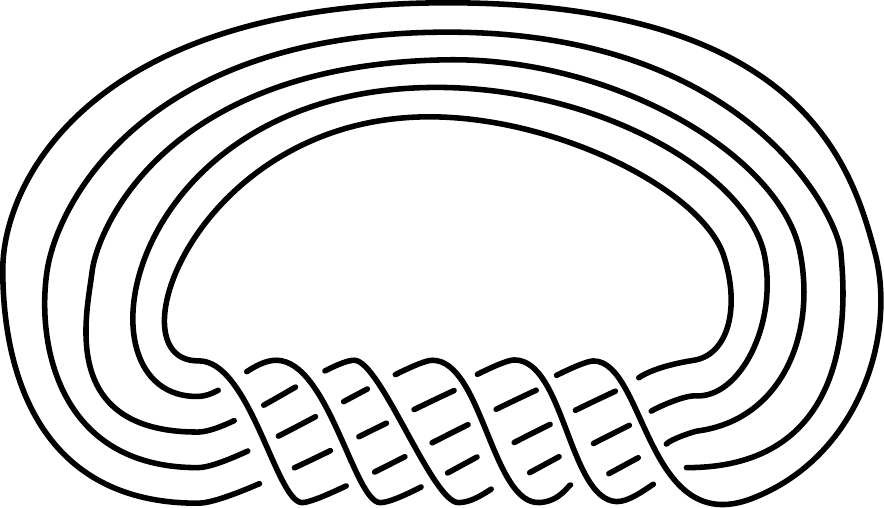}
\caption{An example of the torus knot $T_{p,p+1}$; here $p=5$.}
\label{fig:TorusKnot}
\end{figure}
Let $D$ be the positive, untwisted Whitehead double of the right handed trefoil, and let $\mathcal K_p = I_{D,p}$ for $p$ an odd prime (see Figures \ref{fig:InfectedKnotSmall} and \ref{fig:InfectedKnots}(b)).  Let $\mathcal K_{p,k} =I_{\#_kD,p}$, and note that $\mathcal K_{p,1}=\mathcal K_p$.  The rest of the paper will be devoted to proving that these knots are topologically doubly slice, but not smoothly doubly slice. 

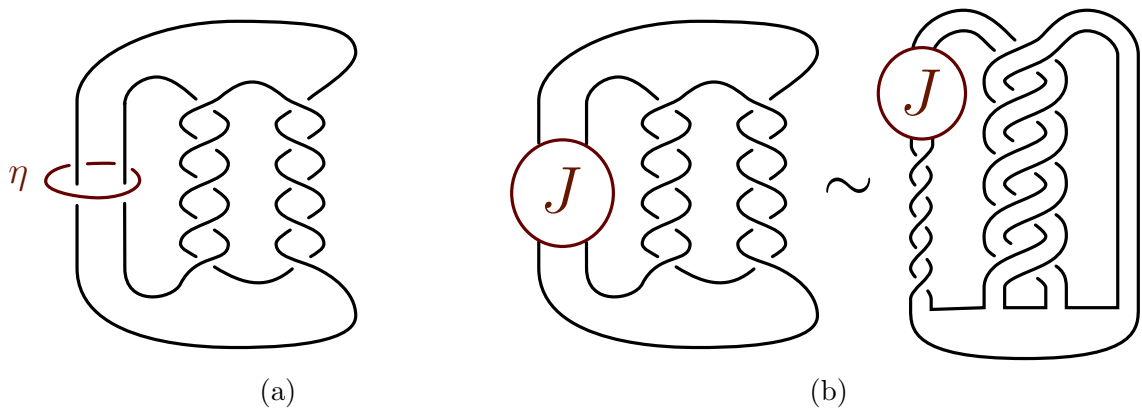
\begin{figure}
\centering
\includegraphics[scale = .75]{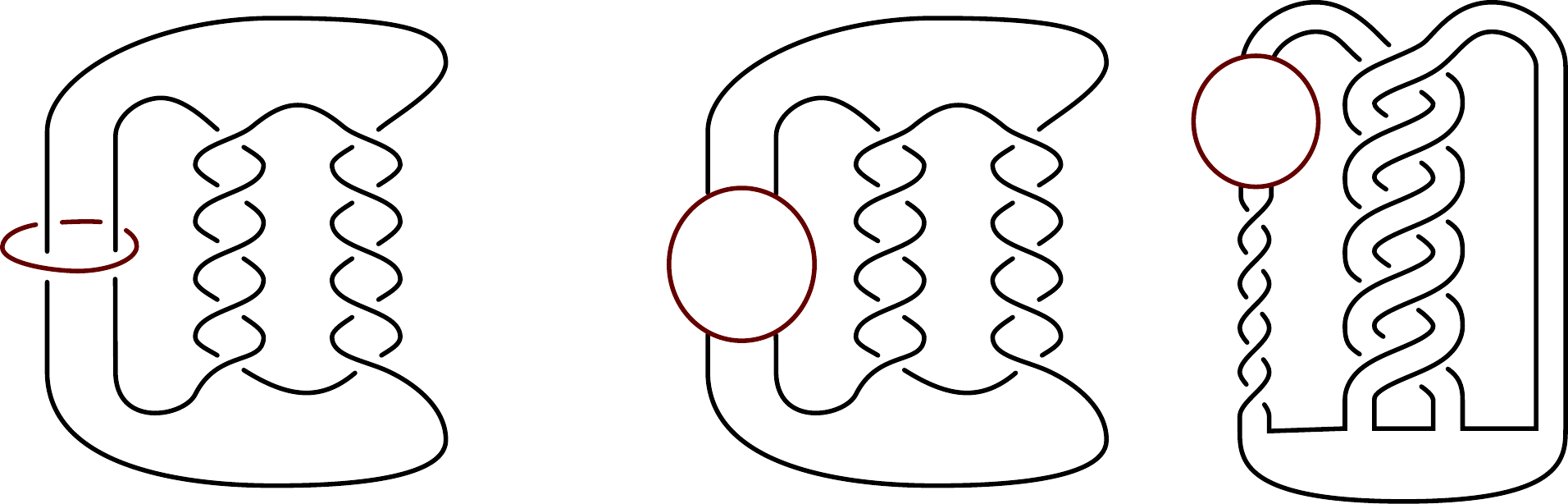}
\put(-333,-15){(a)}
\put(-428,67){\textcolor{Sepia}{\Large$\eta$}}
\put(-125,-15){(b)}
\put(-120,60){\Huge$\sim$}
\put(-227,56){\Huge\textcolor{Sepia}{$J$}}
\put(-91,93){\Huge\textcolor{Sepia}{$J$}}
\caption{(a) The knot $T_{2,p}\#T_{2,-p}$ along with the infection curve $\eta$.  (b) Two descriptions of the result of infecting $T_{2,p}\#T_{2,-p}$ with some knot $J$ along $\eta$.  Here, $p=5$.}
\label{fig:InfectedKnots}
\end{figure}

\subsection{A sufficient condition for double sliceness}\label{subset:first}\ 

In this subsection we will present a sufficient condition for a knot $K$ to be doubly slice that applies when $K$ is obtained by a certain type of infection.  We remark that Donald \cite{donald:embedding} gives a different sufficient condition: one which involves systems of ribbon bands for $K$.

Our criterion will make use of some well-known facts about topologically locally flat surfaces in 4--manifolds that result from the work of Freedman and Quinn \cite{freedman:4manifolds,freedman-quinn}.

\begin{theorem}\label{thm:freedman}\ 
\begin{enumerate}
\item Let $K$ be a knot in $S^3$ with Alexander polynomial $\Delta_K=1$.   Then, there exists a topologically locally flat disk $D$ properly embedded in $B^4$ with $\partial D=K$ and $\pi_1(B^4-D)\cong\Z$.

\item Let $\kappa$ be a topologically locally flat 2--knot in $S^4$ with $\pi_1(S^4-\kappa)\cong\Z$.  Then, there exists an embedded 3--ball $B\subset S^4$ with $\partial B=\kappa$.

\end{enumerate}
\end{theorem}

There is a simple corollary to this theorem that will be useful below (cf. \cite{kirby-melvin,gordon-sumners}).

\begin{corollary}\label{coro:freedman}
Let $K$ be a knot in $S^3$ with $\Delta_K=1$.  Then, $K$ is topologically doubly slice.
\end{corollary}

\begin{proof}
By Theorem \ref{thm:freedman}, we know that $K$ bounds a topological disk $D\subset B^4$ whose complement has fundamental group $\Z$.  Moreover, we have $\pi_1(S^3-K)\to\pi_1(B^4-D)\cong \Z$ is surjective.  If we double the pair $(B^4, D)$ along the boundary $(S^3,K)$, then we get $(S^4, \kappa)$, where $\kappa$ is a topological 2--knot.  It follows that $\pi_1(S^4-\kappa)\cong \Z$ by van Kampen's theorem (this uses the surjectivity).  Thus, $\kappa$ is topologically unknotted with $\kappa\cap S^3=K$, so $K$ is topologically doubly slice.
\end{proof}

\begin{proposition}\label{prop:topo_infection}
Let $K$ be a topologically doubly slice knot and let $K'=I_{\vec\eta}(\vec J)$ be the result of infecting $K$ with the knots $J_i$, each of which is topologically doubly slice.  Then $K'$ is topologically doubly slice.
\end{proposition}

\begin{proof}
We can isotope the link $K\cup\vec\eta$ so that the $\eta_i$ span small, disjoint disks $D_i$ for $i=1,\ldots, n$, which $K$ intersects transversely in $m_i$ points.  Because $K$ is doubly slice, there is an unknotted 2--sphere $\kappa\subset S^4$ such that $\kappa\cap (S^3\times[-1,1]) = K\times[-1,1]$.  Let $D_i\times I$ denote a a thickening of $D_i$ in $S^3$, so $(D_i\times I, K\times I)$ is a trivial $m_i$--strand tangle. From each $D_i\times I\times[-1,1]$, we will remove the interior of a small 4--ball $B_i$ such such that $B_i\cap (K\times[-1,1])$ is a disjoint collection of $m_i$ parallel disks and $B_i\cap (S^3,K)$ is a trivial tangle of $m_i$ strands. Let $m=\sum_{i=1}^nm_i$.  Let $\overline B$ be the result of this removal, i.e., to form $\overline B$ we have removed $n$ 4--balls from $S^4$ and and $m$  2--disks from $\kappa$ to form a punctured manifold pair.

Now, let $J_i$ be one of the topologically doubly slice knots that will be used in the infection.  Let $\mathcal J_i$ be an unknotted 2--sphere in $S^4$ such that $\mathcal J_i\cap(S^3\times[-1,1])=J_i\times[-1,1]$.  Let $\lambda_i$ denote the disjoint union of $m_i$ parallel copies of $\mathcal J_i$.  Then, $\lambda_i\cap S^3$ is the $(m_i,0)$--cable $C_i$ of $J_i$, and $\lambda_i\cap(S^3\times[-1,1])=C_i\times[-1,1]$.

We can assume that the parallel copies of $\mathcal J_i$ are close enough so that there is a small 4--ball $B'_i\subset S^3\times [-1,1]$ such that $B_i'\cap(C_i\times I)$ is a collection of $m_i$ parallel disks and $B_i'\cap(S^3,C_i)$ is a trivial tangle of $m_i$ strands.  Form $\overline B_i$ by removing the interior of $B_i'$.  Then $\overline B_i$ is a 4--ball that contains $m_i$ parallel, topologically unknotted disks that intersect the $B^3$ cross-section of $B^4$ in the tangle $(B^3,C_i)$, i.e., a 3--ball containing $m_i$ arcs that are tied in $C_i$.

Finally, we will re-form $S^4$ from $\overline B$ by gluing in $\overline B_i$ along  $\partial B_i\subset \overline B$.  This has the effect of replacing each parallel set of $m_i$ topological disks that we removed from $\kappa$ with a parallel set of $m_i$ topological disks.  Since $\kappa$ was originally topologically unknotted, this new 2--sphere $\kappa'$ is clearly topologically unknotted.  Furthermore, for each $i$, we removed from $(S^3, K)$ a trivial tangle of $m_i$ strands.  We have now replaced that tangle with the $(B^3, C_i)$ tangle described above.  The result of this is to tie the $m_i$ strands in the knot $C_i$.  This is precisely the effect of infection of $K$ with $J_i$ along $\eta_i$.  In other words, $\kappa'$ is a topologically unknotted 2--sphere with $\kappa'\cap S^3 = I_{\vec\eta}(\vec J)=K'$.  It follows that $K'$ is topologically doubly slice.
\end{proof}

We remark that the conclusion of Proposition \ref{prop:topo_infection} holds if $K$ is smoothly doubly slice and that an analogous proposition holds in the smooth category.  We can apply the previous proposition to the knots $\mathcal K_{p,k}$, proving that the knots referenced in Theorems \ref{thm:main} and \ref{thm:main2} are topologically doubly slice.

\begin{corollary}\label{coro:doubly_slice}
The knots $\mathcal K_{p,k}$ are topologically doubly slice and smoothly slice.
\end{corollary}

\begin{proof}
Let $K=T_{2,p}\#T_{2,-p}$, let $J=\#_kD$, and let $\eta$ be as shown in Figure \ref{fig:InfectedKnots}.   Then, $\mathcal K_{p,k} = K_\eta(J)$, with $K$  smoothly doubly slice (by Zeeman \cite{zeeman}) and $J$  topologically doubly slice (by Corollary \ref{coro:freedman}, since $\Delta_J=1$).  Thus, by Proposition \ref{prop:topo_infection}, $\mathcal K_{p,k}$ is topologically doubly slice.

To see that $\mathcal K_{p,k}$ is smoothly slice, consider it as the boundary of a punctured Klein bottle, as in Figure \ref{fig:DoubleBranchedCover}(a).  This punctured Klein bottle is formed by attaching two bands to a disk.  In this case, the right most band is unknotted and untwisted.  It follows that we can push the interior of the punctured Klein bottle into the 4--ball and surger it along the core of this band.  The result is a smooth, properly embedded disk in the 4--ball with boundary $\mathcal K_{p,k}$.
\end{proof}

\subsection{Relevant 3--manifolds and 4--dimensional cobordisms}\label{subsec:triad}\  

Let $I_{J,n}$ be the infected knot described above, and let $Z_{J,n}$ be the double-cover of $S^3$ branched along $I_{J,n}$.  In \cite{akbulut-kirby}, Akbulut and Kirby described how to get a surgery diagram for the double-cover of $B^4$ branched along a surface bounded by a knot.  Applying this technique, we see that $Z_{J,n} = S^3_{n,-n}((J\#J)_{(2,0)})$, i.e., surgery on the (2,0)--cable of $J\#J$ with surgery coefficients $n$ and $-n$ (see Figure \ref{fig:DoubleBranchedCover}).  Note that throughout this paper, $J$ will be a reversible knot, so $J^r=J$.

\begin{figure}
\centering
\includegraphics[scale = .75]{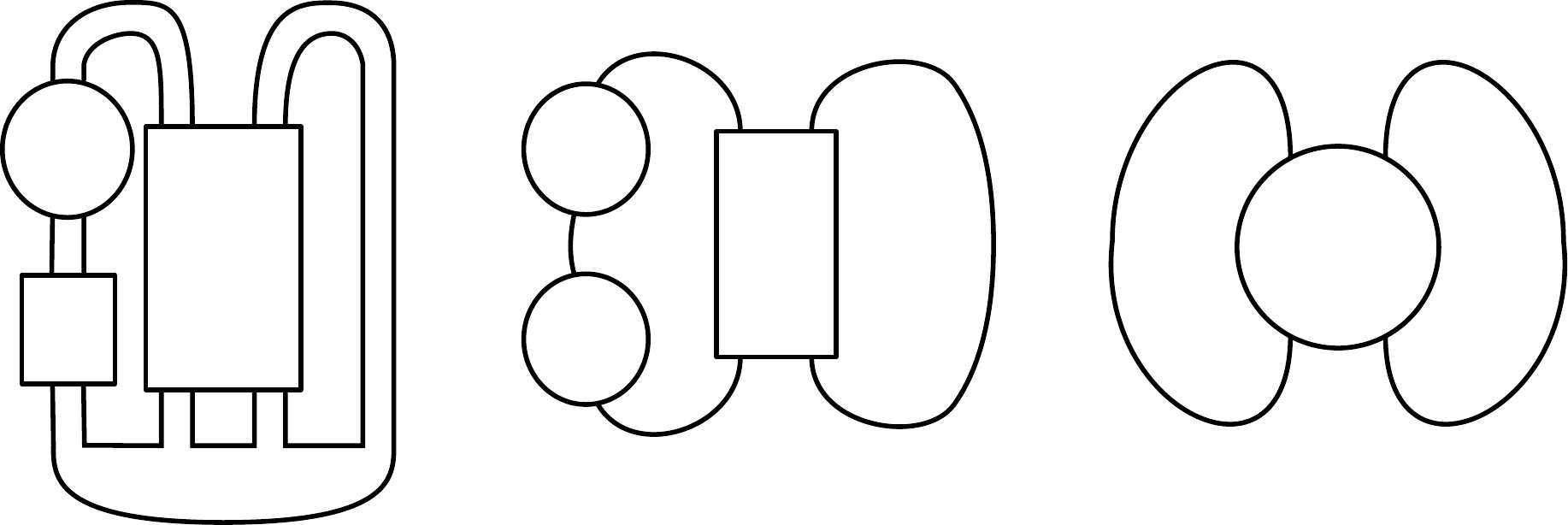}
\put(-350,-15){(a)}
\put(-140,-15){(b)}
\put(-140,65){\huge$\approx$}
\put(-387,90){\huge$J$}
\put(-385,45){\Large$n$}
\put(-225,13){\Large$n$}
\put(-160,15){\Large$0$}
\put(-90,15){\Large$n$}
\put(-35,15){\Large-$n$}
\put(-208,65){\Large$2n$}
\put(-347,63){\huge$n$}
\put(-257,90){\huge$J$}
\put(-257,40){\huge$J^r$}
\put(-78,65){\LARGE$J\#J^r$}
\caption{(a) The knot $I_{J,n}$, shown as the boundary of a punctured Klein bottle.  The boxes indicate $n$ positive half-twists.  (b) Two descriptions of the resulting branched double cover, $Z_{J,n}$, which are related by a handleslide.}
\label{fig:DoubleBranchedCover}
\end{figure}
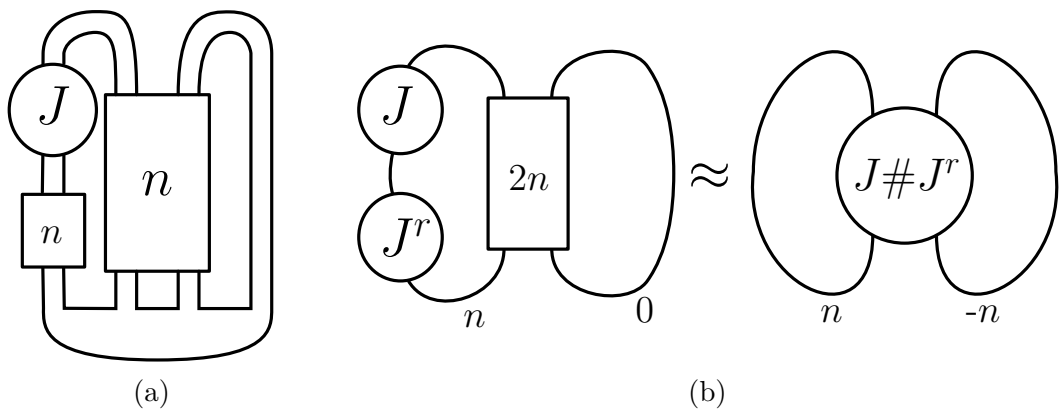

Let $X=S^3_n(J\#J)$, and let $K\subset X$ be the null-homologous knot shown in Figure \ref{fig:KnotInY}.  If we think of $X$ as $n$--surgery on one component of the (2,0)--cable of $J\#J$, then $K$ is the image (in the surgery manifold) of the second component of the (2,0)--cable.  Since $K$ is a longitudinal push-off of $J\#J$ in $S^3$, it bounds, in $S^3$, a Seifert surface $F$ with $g(F) = g(J\#J)$.  Since $F$ is disjoint from $J\#J$, we see that $F$ is a Seifert surface for $K$ in $X$, as well.  Thus, $K$ is null-homologous in $X$.  With respect to the Seifert framing of $K$ in $X$, we have $X_{-n}(K) = Z_{J,n}$.

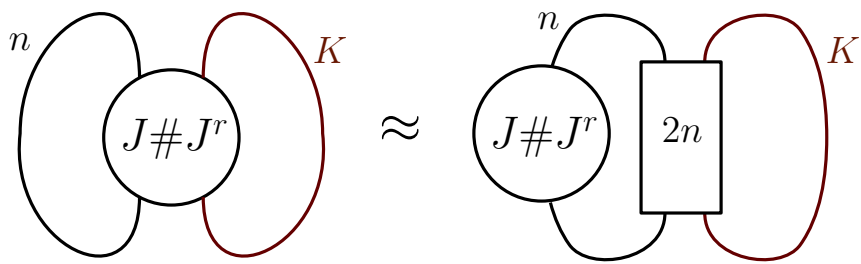
\begin{figure}
\centering
\includegraphics[scale = .75]{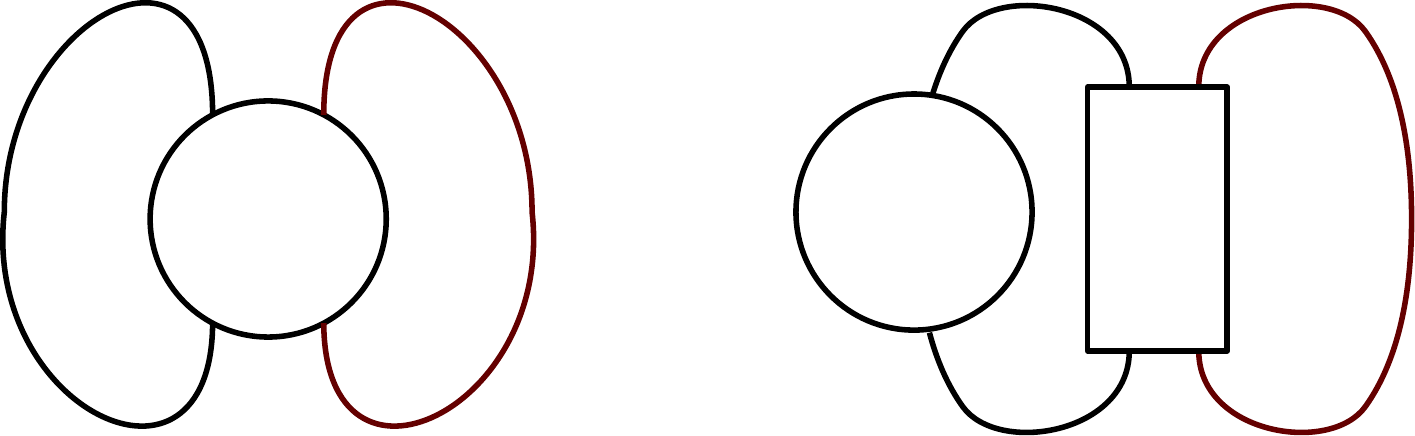}
\put(-170,45){\huge$\approx$}
\put(-310,80){\Large$n$}
\put(-110,88){\Large$n$}
\put(-63,45){\Large$2n$}
\put(-195,75){\textcolor{Sepia}{\Large$K$}}
\put(-1,75){\textcolor{Sepia}{\Large$K$}}
\put(-268,42){\LARGE$J\#J^r$}
\put(-128,43){\LARGE$J\#J^r$}
\caption{Two equivalent views of the null-homologous knot $K$ in $X=S^3_n(J\#J^r)$.  Note that the Seifert framing on $K$ is different in these two descriptions.  Compare with Figure \ref{fig:DoubleBranchedCover} to see that $Z$ is obtained by surgery on $K$.}
\label{fig:KnotInY}
\end{figure}

Now, let $Y = X_{-n-1}(K)$.  After performing a handle-slide and blowing down (see Figure \ref{fig:BlowdownY}), we see that $Y = S^3_{n^2+n}(J\#J\#T_{n,n+1})$.  These three manifolds, $X, Y,$ and $Z=Z_{J,n}$ form a triad:

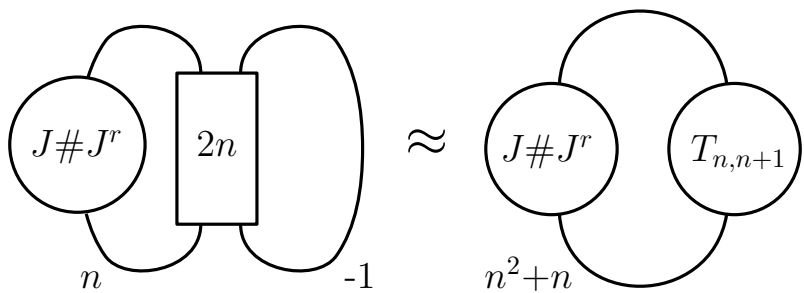
\begin{figure}
\centering
\includegraphics[scale = .75]{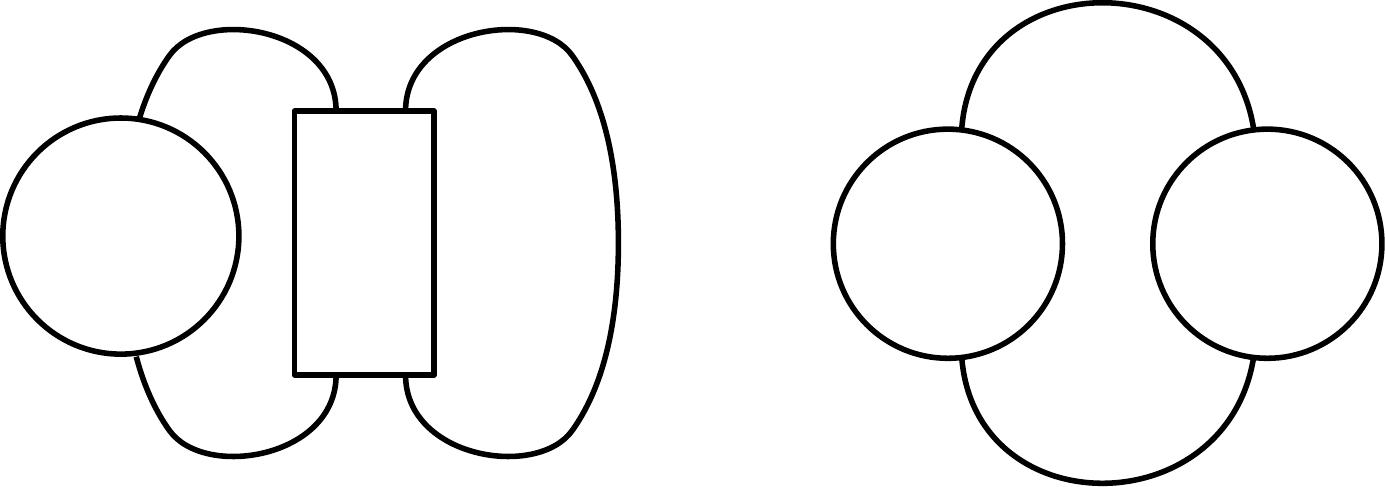}
\put(-150,50){\huge$\approx$}
\put(-273,0){\Large$n$}
\put(-173,0){\Large-$1$}
\put(-120,0){\Large$n^2$+$n$}
\put(-229,50){\Large$2n$}
\put(-292,50){\Large$J\#J^r$}
\put(-114,48){\Large$J\#J^r$}
\put(-42,48){\Large$T_{n,n+1}$}
\caption{The manifold $Y$ is obtained as $(-1)$--surgery on $K$ in $X$.  After a blowdown, $Y$ can be realized by $(n^2+n)$--surgery on $J\#J\#T_{n,n+1}$.}
\label{fig:BlowdownY}
\end{figure}

$$\begin{tikzpicture}
  \matrix (m) [matrix of math nodes,row sep=4em,column sep=4em,minimum width=2em] {
     X& \  & Z \\
     \  & Y & \  \\};
  \path[-stealth]
    (m-1-1) edge node [pos=.6,left] {$W_1\ $} (m-2-2)
    (m-2-2) edge node [pos=.4,right] {$\ W_2$} (m-1-3)
    (m-1-3) edge node [above] {$W_3$} (m-1-1);
\end{tikzpicture}$$

Now, since $-\overline{W_3}$ is the cobordism from $X$ to $Z$ corresponding to attaching a $(-n)$--framed 2--handle along $K$ in $X$, we have that $H_2(-\overline{W_3})\cong\Z$ is generated by the class $S_3 = F\cup D^2$ (i.e., the genus $g$ Seifert surface for $K$, capped off with the core disk of the 2--handle), and $[S_3]\cdot [S_3] = -n$ in $-\overline{W_3}$.  Therefore, $W_3$ is a positive definite cobordism whose second homology is generated by a surface of genus $g(J\#J)$ with self-intersection $n$.

Similarly, $W_1$ is formed by attaching a $(-n-1)$--framed 2--handle to $X$ along $K$.  The result is that $W_1$ is a negative definite cobordism whose second homology is generated by a class $[S_1]$, where $S_1$ is a surface of genus $g=g(K)$ with self-intersection $-n-1$.  Note also that $H^2(W_1)\cong \Z_n\oplus\Z$.  The map from $H^2(W_1)\to H^2(X)$ induced by restricting to $X$ is realized by projection onto the first component:  $\Z_n\oplus\Z\to\Z_n$, while the corresponding map from $H^2(W_1)\to H^2(Y)$ is reduction modulo $n+1$ of the second component and the identity on the first: $\Z_n\oplus\Z\to\Z_n\oplus\Z_{n+1}$.

Finally, $W_2$ is obtained by attaching a $(-1)$--framed 2-handle along the meridian $\mu$ shown in Figure \ref{fig:MeridinalBlowdown}.  In fact, $\mu$ is rationally null-homologous, and bounds a rational Seifert surface, $S_2$.  It turns out that this surface has self-intersection $-n^2-n$ and $[S_2]$ generates the second homology of $W_2$, so $W_2$ is negative definite.  Note also that $H^2(W_2)\cong \Z_n\oplus\Z$.  The map from $H^2(W_2)\to H^2(Y)$ induced by restricting to $Y$ is realized by reduction modulo $n+1$ of the second component and the identity on the first: $\Z_n\oplus\Z\to\Z_n\oplus\Z_{n+1}$, while the corresponding map from $H^2(W_2)\to H^2(Z)$ is reduction modulo $n$ of the second component and the identity on the first: $\Z_n\oplus\Z\to\Z_n\oplus\Z_n$.

\begin{figure}
\centering
\includegraphics[scale = .75]{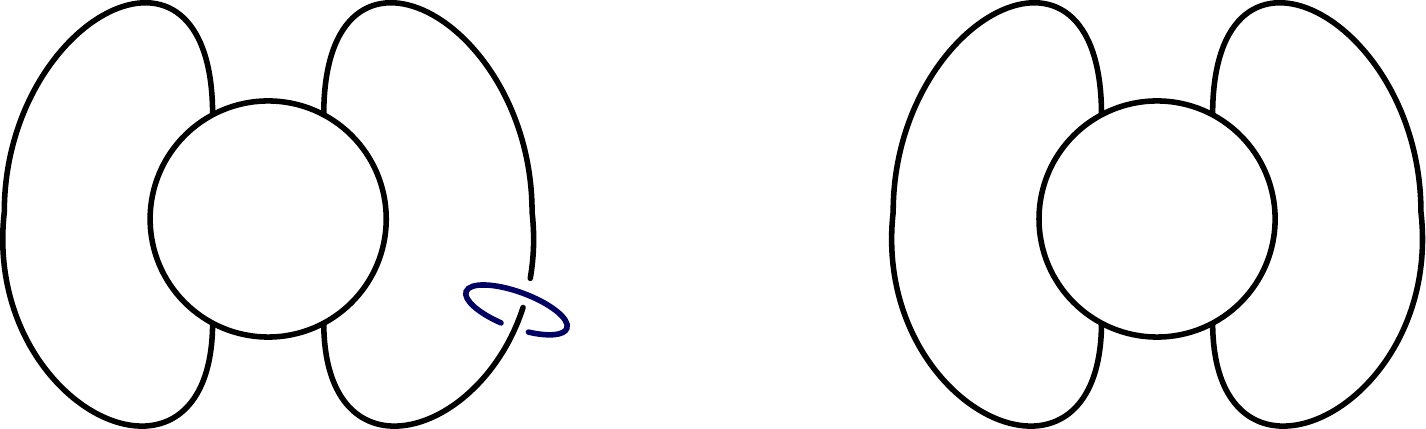}
\put(-67,-30){\Large(b)}
\put(-260,-30){\Large(a)}
\put(-267,-8){\large$n$}
\put(-215,-8){\large-$(n+1)$}
\put(-182,30){\textcolor{Blue}{\large$\mu$}}
\put(-75,-8){\large$n$}
\put(-25,-8){\large-$n$}
\put(-270,43){\Large$J\#J^r$}
\put(-74,43){\Large$J\#J^r$}
\caption{(a) The manifold $Y'$ shown with the rationally null-homologous meridian $\mu$.  (b) The manifold $Z$, obtained by $(-1)$--surgery on $\mu$.}
\label{fig:MeridinalBlowdown}
\end{figure}
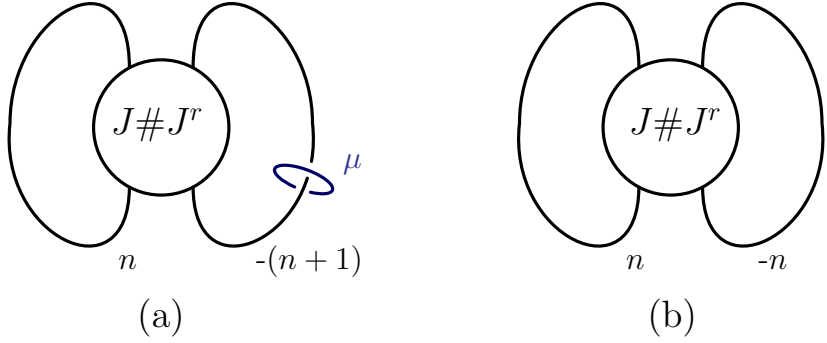

Let us see why the capped off rational Seifert surface has self-intersection $-n^2-n$.  We are performing $(-1)$--surgery on a meridian, $\mu$, to one component of the framed link giving $Y$.  The effect of this surgery is to attach a 0--framed disk to every $(-1,1)$--curve on $\partial N(\mu)$.  If we select $n+1$ of these curves, we get the torus link $T_{n+1,n+1}$.    Since this is an $(n+1)$--component link and each component is a meridian, it is homologous to $(n+1)\cdot\mu=0$.  So this $T_{n+1,n+1}$ bounds an orientable surface in $Y$.  If we attach 0--framed 2--handles to each component, it is easy to see that the intersection among these disks is simply given by the total linking of the components of $T_{n+1,n+1}$.  Let $S_2$ be the surface obtained by capping off the $n+1$ boundary components of this orientable surface with these 0--framed disks.  Then, $S_2\cdot S_2 = -n(n+1)$.

The following example will be pertinent to our calculations in Section \ref{section:calculations}.

\begin{example}\label{ex:triad_spaces}
If $J$ is the unknot, then 
\begin{eqnarray*}
X &=&L(n,1), \\
Y &=& S^3_{n^2+n}(T_{n,n+1}) = L(n,1)\#L(n+1,-1), \text{ and } \\
Z &=&L(n,1)\#L(n,-1).
\end{eqnarray*}



In general, 
\begin{eqnarray*}
X  & = &  S^3_n(J\#J), \\
Y & = & S^3_{n^2+n}(J\#J\#T_{n,n+1}), \text{ and }\hspace{.75in}\hspace{.4in} \\
Z  & = & S^3_{n,-n}((J\#J)_{(2,0)}).
\end{eqnarray*}

\end{example}

\subsection{Enumerating $\spinc$ structures}\label{subsec:enumerate}\ 

This nice homological  set-up gives us natural enumerations of the $\spinc$ structures on the manifolds in question.  Since $X$ is surgery on a knot in $S^3$, there is an enumeration of $\spinc(X)$ by $i\in\Z_n$.  Let $\frak s_i\in\spinc(X)$ for some $i\in\Z_n$.

Let $[\frak s_i,\frak s_j]\in\spinc(Y)$ denote the $\spinc$ structure on $Y$ that is cobordant to $\frak s_i$ via  a $\spinc$ structure $[\frak s_i, \frak t_m]$ with
$$\langle c_1([\frak s_i,\frak t_m]),[S_1]\rangle = 2m+n,$$
where $m\in\Z$ is any integer satisfying $m\equiv j\pmod{n+1}$.

Let $[\frak s_i, \frak s_k]\in\spinc(Z)$ denote the $\spinc$ structure that is cobordant to $[\frak s_i,\frak s_j]$ via $[\frak s_i,\frak r_m]$ with
$$\langle c_1([\frak s_i,\frak r_m]),[S_2]\rangle = 2m+n(n+1),$$
where $m\in\Z$ is any integer satisfying $m\equiv j \pmod{n+1}$ and $m\equiv k\pmod n$.

A key feature of this set-up is that we are given affine identifications:
\begin{eqnarray*}
\spinc(X) & \cong & \Z_n \\
\spinc(Y) & \cong & \Z_n\oplus\Z_{n+1} \\
\spinc(Z) & \cong & \Z_n\oplus\Z_n,
\end{eqnarray*}
the first and third of which take the unique spin structure to the identity element.

\subsection{Remarks about surgery coefficients}\label{subsec:coefficients}\ 

In what follows, we will use Heegaard Floer theory to study the manifolds described above.  In general, when studying the Heegaard Floer homology of surgeries on knots, calculations become much simpler when dealing with large surgery coefficients.  For example, Theorem \ref{thm:surgery_formula}, which we will use extensively, requires that the surgery coefficient be positive and at least $2g-1$, where $g$ is the genus of the knot that is being surgered.  The purpose of this subsection is to show that this criterion is  met in what follows and to examine the knots $\mathcal K_{p,k_p}$, which will be used in Section \ref{section:infinite_order} to prove Theorem \ref{thm:main2}.

Let $I_{J,p}$ be the knot formed by infecting $T_{2,p}\#T_{2,-p}$ with $J$ along $\eta$, as shown in Figure \ref{fig:InfectedKnots}.  Consider $J=\#_kD$, which is a knot of genus $k$.  In order to apply Theorem \ref{thm:surgery_formula} to the manifold $X=S^3_p(J\#J)$, we must have $p\geq 2g(J\#J)-1=4k-1$.  In order to apply Theorem \ref{thm:surgery_formula} to the manifold $Y=S^3_{p^2+p}(J\#J\#T_{p,p+1})$, we must have 
$$p^2+p\geq 2g(J\#J\#T_{p,p+1})-1=2\left(2k + \frac{p(p-1)}{2}\right)-1.$$
So, we must have $p\geq \frac{4k-1}{2}$, i.e., $k\leq \frac{2p+1}{4}$.  In Section \ref{section:infinite_order}, it will be necessary for us to consider knots where $k\geq\frac{p+5}{12}$.  Let $k_p = \lceil\frac{p+6}{12}\rceil$, and define $\mathcal K_{p,k_p} = I_{\#_{k_p}D,p}$.  Then, the manifolds associated to $\mathcal K_{p,k_p}$ are surgeries of appropriately large coefficient and $k_p$ is large enough to satisfy the conditions in Section \ref{section:infinite_order}:
$$ 4k_p-1 \leq 4\left[\frac{p+6}{12}+1\right]-1 = \frac{p+18}{3}-1 \leq p, $$
and
$$ \frac{4k_p-1}{2}\leq \frac{4\left[\frac{p+6}{12}+1\right]-1}{2} = \frac{p+15}{6}\leq p .$$
These inequalities will be satisfied for large $p$, and for small $p$ it is easy to see that the condition on $k_p$ can be relaxed.  It should be noted that there are, in general, many values of $k$ that will suffice for each value of $p$, we have simply chosen one that will work for all large values of $p$.

\subsection{Linking forms and Question \ref{question:cancelation}}\label{subsec:question}\ 

A knot $K\subset S^3$ is called \emph{stably doubly slice} if there exists a doubly slice knot $J$ such that $K\#J$ is doubly slice.  Question \ref{question:cancelation} can be rephrased to ask whether there exist stably doubly slice knots that are not doubly slice. In this subsection we show that the correction terms could possibly detect the difference between smoothly doubly slice knots and smoothly stably doubly slice knots.

Analogously, we say that a 3--manifold $M$ \emph{stably embeds} smoothly in $S^4$ if there is a 3--manifold $N$ that embeds smoothly in $S^4$ such that $M\#N$ embeds smoothly in $S^4$.  It is not known if such an $M$ must itself embed in $S^4$.

Give a finite abelian group $G$, a \emph{linking form} on $G$ is a non-degenerate, symmetric, bilinear form $\lambda:G\times G\to\Q/\Z$.  For every rational homology 3--sphere $M$ there is a linking form $\lambda:H_1(M)\times H_1(M)\to\Q/\Z$ defined by Poincar\'e duality.

Now we will consider \emph{linking triples} $(G,\lambda, f)$, where $G$ is a finite abelian group, $\lambda$ is a linking form on $G$, and $f:G\to\Q$ is a function (not necessarily a homomorphism).  Such a triple is called \emph{metabolic} if there is a subgroup $G_1<G$ with $|G_1|^2 =|G|$ such that $\lambda|_{G_1}\equiv 0$ and $f(G_1)=0$.  The triple is called \emph{hyperbolic} if $G=G_1\oplus G_2$ with $G_1\cong G_2$ such that $\lambda|_{G_i}\equiv 0$ and $f(G_i)=0$ for $i=1,2$.  Note that the set of linking triples has an additive structure given by orthogonal sum.





\begin{lemma}\label{lemma:hyp_cancel}
Let $(A,\mu,f)$ and $(B,\nu,g)$ be linking triples.  If $(A,\mu,f)$ and $(A\oplus B, \mu\oplus\nu,f\oplus g)$ are both hyperbolic, then $(B,\nu,g)$ is metabolic.
\end{lemma}

Though we will use the hypotheses that $(A,\mu,f)$ and $(A\oplus B, \mu\oplus\nu,f\oplus g)$ are hyperbolic, the result hold if these objects are merely metabolic.  The following proof is, in essence, due to Kervaire \cite{kervaire} (cf. \cite{gilmer:slice}).

\begin{proof}
Let $A=A_0\oplus A_1$ and $A\oplus B = L\oplus M$ be the hyperbolic splittings of $A$ and $A\oplus B$.  Let $L_i=L\cap(A_i\oplus B)$ and $M_i=M\cap(A_i\oplus B)$ for $i=0,1$.  Let $B_i^L$ and $B_i^M$ be the projections of $L_i$ and $M_i$ onto $B$, respectively.  From now on, we will restrict our attention to $B_0^L$.

Let $b,b'\in B_0^L$.  Then there exist $a,a'\in A_0$ such that $a\oplus b,a'\oplus b'\in L$.  Then,
$$\nu(b,b') = \mu(a,a')+\nu(b,b') = \mu\oplus\nu(a\oplus b,a'\oplus b') = 0,$$
and
$$ g(b) = f(a) + g(b) = f\oplus g(a\oplus b) = 0.$$
Thus, the restrictions of $\nu$ and $g$ to the $B_0^L$ vanish.  Next we show that $|B_0^L|^2=|B|$.  Consider the following two short exact sequences:
$$ 0 \longrightarrow L_0 \longrightarrow L \stackrel{\pi_{A_1}}{\longrightarrow} L_{A_1} \longrightarrow 0,$$
where $\pi_{A_1}:A\oplus B\to A_1$ is projection onto $A_1<A$, and
$$ 0 \longrightarrow L\cap(A_0\oplus0) \longrightarrow L_0 \stackrel{\pi_{B}}{\longrightarrow} B_0^L \longrightarrow 0,$$
where $\pi_B:A\oplus B$ is projection onto $B$.

Next, we claim that $|L\cap(A_0\oplus 0)|\cdot|A_0|\cdot|L_{A_1}|\leq |A|$.  Assuming this, we see that
$$|B_0^L| = \frac{|L_0|}{|L\cap(A_0\oplus 0)|}=\frac{|L|}{|L\cap(A_0\oplus 0)|\cdot|L_{A_1}|}\geq \frac{|L|\cdot|A_0|}{|A|} = |B|^{1/2}.$$
Because $B_0^L$ is isotropic and $\nu$ is non degenerate, we have that $|B_0^L|^2=|B|$, as desired.  To justify claim assumed above, we will prove that $L\cap(A_0\oplus0)$ is orthogonal to $A_0\oplus L_{A_1}$ under $\mu$.  Clearly, $L\cap(A_0\oplus 0)\perp A_0$.  Let $u\in L\cap(A_0\oplus0)$ and $w\in L_{A_1}$.  Then there exists $v\oplus x\in L_0$ such that $(v+w)\oplus x\in L$.  Then, 
$$\mu(u,w) = \mu(u,w) + \mu(u,v) = \mu(u,w+v) + \nu(0,x) = \mu\oplus\nu(u\oplus0,(w+v)\oplus x) = 0.$$

This shows that $B_0^L$ is a metabolizing summand of $B$.  The same is true for $B_1^L, B_0^M,$ and $B_1^M$.\end{proof}

Note that the four metabolizers produced in the proof above are all isomorphic.  This follows from the classification of linking forms, specifically the fact that a linking form splits over the homogeneous $p$--group components of the group \cite{wall:linking}.  Because of this, we could have performed the above analysis one homogeneous $p$--group component at a time, each of which would split via $L$ and $M$.

Next, we investigate how these metabolizers sit inside $A$ and $B$.  Suppose that $A$ and $B$ are homogeneous $p$--groups with a common exponent and have ranks $2r$ and $2s$, respectively. Without loss of generality, we can write
$$L = \langle (a_1, b_1), \ldots, (a_t,b_t), (0,b_{t+1}), \ldots, (0, b_{t+l}), (a_{t+l+1},0),\ldots, (a_{r+s}, 0),$$
where the $b_i$ are linearly independent, and the $a_j$ are linearly independent.  Let $t'=r-l$.  Without loss of generality, we can assume that $a_1, \ldots, a_{t'}\in A_0$ and $a_{t'+1},\ldots, a_{2t'}\in A_1$ (by consideration of the ranks of $B_0^L$ and $B_1^L$).  Since $\nu$ is non-degenerate, we can assume that, for $0\leq i\leq t'$ and $t'+1\leq j\leq 2t'$, $\nu(b_i, b_j)\not=0$ if and only if $j=t'+i$ (perform change of bases within these rank $t'$ summands).  Note that $B/\langle b\rangle^\perp$ has rank one for each $b\in B$.

Clearly, $t+l\leq 2s$, and, in fact, we have that $t$ is even with $t/2+l = r$, i.e., $t=2t'$.  This claim follows from the $\nu$ being non-degenerate; if $t/2<r-l$, there is an element $(a_{2t'+1}, b_{2t'+1})$ that an be assumed to have the property that $\nu(b_{2t'+1}, b_i)=0$ for all $0\leq i\leq t'$.  However, if this were the case, then $\langle b_1,\ldots, b_{t'},b_{2t'+1}, b_{t+1}, \ldots, b_{t+l}\rangle$ would have rank $r+1$ and be isotropic, a contradiction.

It follows that each $a_i$ for $0\leq i\leq t$ is in either $A_0$ or $A_1$.  Together with a similar argument for $M$, we get that $\pi_B(L_0+L_1+M_0+M_1) =B$.  In particular, $B = B_0^L+B_1^L+B_0^M+B_1^M$.  We can use this to prove a simple corollary.

\begin{corollary}\label{coro:rank4}
Let $(A,\mu,f)$ and $(B,\nu,g)$ be linking triples.  If $(A,\mu,f)$ and $(A\oplus B, \mu\oplus\nu,f\oplus g)$ are both hyperbolic,  and if each homogeneous $p$--group component of $B$ is at most rank 4, then $(B,\nu,g)$ is hyperbolic.
\end{corollary}

\begin{proof}
Since $B$ is rank 4 and spanned by four metabolizers of rank 2 (by the comments above), either some pair of the metabolizers are disjoint, or there is an element $b$ common to each of the four metabolizers.  However, the latter case implies that $(0,b)\in L\cap M$, a contradiction.  Thus, there is a pair giving a hyperbolic splitting of $(B,\nu,g)$.  If $B$ is rank 2, a similar argument works.
\end{proof}

Next, we give a counterexample that shows that Corollary \ref{coro:rank4} is as strong as possible, in some sense.

\begin{example}\label{ex:rank6}
Let $A\cong\Z_p^6=\langle z_1, w_1, z_2, w_2, z_3, w_4\rangle$ and let $B\cong\Z_p^6=\langle x_1, y_1, x_2, y_2, x_3, y_4\rangle$.  Let $A_0 = \langle z_1, z_2, z_3\rangle$ and $A_1 = \langle w_1, w_2, w_3\rangle$.  With respect to these bases, let $\mu$ and $\nu$ be linking forms given by 
$$\bigoplus_3\begin{pmatrix} 0 & -2/p \\ -2/p & 0 \end{pmatrix} \text{ and } \bigoplus_3\begin{pmatrix} 0 & 2/p \\ 2/p & 0 \end{pmatrix},$$
respectively.  Consider the splitting $A\oplus B=L\oplus M$, where 
$$L =\langle (z_1, x_1), (z_2, x_2), (w_1, y_1), (w_2, y_2), (0, x_3), (w_3, 0)\rangle,$$
and
$$M =\langle (z_1, y_2), (z_3,x_1), (w_1, x_2), (w_3, y_1), (0, y_3), (w_2, 0)\rangle.$$
It is straightforward to check that $L\cap M=0$ and that $L+M=A\oplus B$.  Furthermore, it is obvious that $\mu\oplus\nu$ vanishes on both $L$ and $M$.  Next, notice that
\begin{eqnarray*}
B_0^L & = & \langle x_1, x_2, x_3\rangle \\
B_1^L & = & \langle y_1, y_2, x_3\rangle \\
B_0^M& = & \langle x_1, y_2, y_3\rangle \\
B_1^M & = & \langle y_1, x_2, y_3\rangle. 
\end{eqnarray*}
No pair of these metabolizers is disjoint.  Define $g:B\to\Q$ by
$$g(b) = \begin{cases} 0 & \text{ if $b\in B_0^L\cup B_1^L\cup B_0^M\cup B_1^M$}, \\
1 & \text{ otherwise} \end{cases}.$$
Define $f:B\to\Q$ by
$$f(a) = \begin{cases} -g(b_a) & \text{ if $a\not\in A_0\cup A_1$}, \\
0 & \text{ if $a\in A_0\cup A_1$} \end{cases},$$
Where $a\mapsto b_a$ its the isomorphism from $A$ to $B$ that sends the $z_i$ to the $x_i$ and the $w_i$ to the $y_i$.

With this set up, it is clear that $(A,\mu, f)$ is hyperbolic and that $g:B\to \Q$ is not hyperbolic.  It remains to show that $f\oplus g:A\oplus B\to\Q$ vanishes on $L$ and $M$.  This will imply that $(A\oplus B, \mu\oplus\nu, f\oplus g)$ is hyperbolic, thus exemplifying the necessity of the rank restriction in Corollary \ref{coro:rank4}.

Let $l\in L\cup M$ with $l=(a,b)$.  It suffices to check that $f(a)=0$ if $b$ is in one of the metabolizers listed above and that $a\not\in A_0\cup A_1$ if $b$ is not in one of these metabolizers.  It is straightforward to check that these criteria are met.

\end{example}

Let $K\subset S^3$, and let $\mathcal A=(A, \mu,f)$ be the linking triple associated to $\Sigma_2(K)$, i.e., $A=H_1(\Sigma_2(K))$, $\mu$ is the linking from on $A$, and $f(a) = d(\Sigma_2(K),\frak s_a)$, where $\frak s_a$ is the $\spinc$ structure corresponding to $a\in H_1(\Sigma_2(K))$.  Let $\mathcal A_{p^k}$ denote the restriction of this triple to the homogeneous $p^k$--group component of $A$.  We have shown the following.

\begin{proposition}\label{prop:stable_terms}
Let $K\subset S^3$ and let $\mathcal A$ be the associated linking triple.  Suppose that $\det(K) = |A| = p_1^{k_1}\cdots p_n^{k_n}$.
\begin{enumerate}
\item If $K$ is smoothly doubly slice, then $\mathcal A$ is hyperbolic.
\item If $K$ is smoothly stably doubly slice, then $\mathcal A_{p_i^{k_i}}$ is hyperbolic whenever $k_i\leq 4$.
\end{enumerate}
\end{proposition}

Note that (1) is a restatement of Theorem \ref{thm:hyp_corr_terms}.  We will use this result in Sections \ref{section:calculations} and \ref{section:infinite_order} to help prove Theorems \ref{thm:main} and \ref{thm:main2}.




\section{Heegaard Floer homology}\label{section:HF} 

Below, we collect some basic facts about the suite of invariants known as Heegaard Floer homology.  For complete details, see (for example) \cite{oz-sz:absolute,oz-sz:knots,oz-sz:3-manifolds_1}.  Throughout, let $\F$ denote the field with two elements.

\subsection{3--manifold invariants}\ 

Let $M$ be a closed 3--manifold, and let $\frak s\in\spinc(M)$ be a torsion $\spinc$ structure on $M$.  Heegaard Floer homology theory associates to $(M,\frak s)$ a $\Z$--filtered, $\Q$--graded chain complex $CF^\infty$, which is well-defined up to filtered chain homotopy equivalence.  This complex is a free, finitely generated $\F[U,U^{-1}]$--module.  The action of $U$ lowers the filtration level by one, and lowers the grading by two.  Henceforth, if $C$ is any filtered, graded chain complex, then $C_{\{i\leq n\}}$ denotes the subcomplex consisting of elements of filtration level at most $n$.

Denote the associated homology group by $HF^\infty(M,\frak s)$.  If $M$ is a rational homology 3--sphere, it turns out that these  groups are uninteresting.  Let $\mathcal T^\infty = \F[U,U^{-1}]$.  Then, for any rational homology 3--sphere $M$ and any $\frak s\in\spinc(M)$, we get $HF^\infty(M,\frak s)\cong\mathcal T^\infty$.  This means that any interesting information about $(M,\frak s)$ must be stored at the chain complex level.

Indeed, there are associated sub- and quotient-complexes:
$$CF^-(M,\frak s) = CF^\infty(M,\frak s)_{\{i<0\}},$$
$$CF^+(M,\frak s) = CF^\infty(M,\frak s)/CF^-(M,\frak s),$$
and
$$\CFhat(M,\frak s) = CF^\infty(M,\frak s)_{\{i\leq0\}}/CF^-(M,\frak s).$$

The corresponding homology groups, $HF^-(M,\frak s)$, $HF^+(M,\frak s)$ , and $\HFhat(M,\frak s)$ turn out to be very powerful 3--manifold invariants.  These groups are related by two important long exact sequences:
$$ \cdots \longrightarrow HF^-(M,\frak s) \stackrel{\iota}{\longrightarrow} HF^\infty(M,\frak s) \stackrel{\pi}{\longrightarrow} HF^+(M,\frak s) \longrightarrow\cdots$$
and
$$ \cdots \longrightarrow \HFhat(M,\frak s) \stackrel{\hat\iota}{\longrightarrow} HF^+(M,\frak s) \stackrel{U}{\longrightarrow} HF^+(M,\frak s) \longrightarrow\cdots.$$

Note that $\HFhat(M,\frak s)$ is a finitely generated $\F$--vector space.  Define $$HF_{red}(M,\frak s) = HF^+(M,\frak s)/\im(\pi).$$  Let $\mathcal T^+ = \F[U, U^{-1}]/U\cdot\F[U]$.  If $M$ is a rational homology 3--sphere, we have the following decomposition:
$$HF^+(M,\frak s) = \mathcal T^+\oplus HF_{red}(M,\frak s).$$
It turns out that the grading of the element of lowest grading living in $\mathcal T^+$, which we call the tower part of $HF^+(M,\frak s)$, is an interesting invariant called the \emph{correction term}.

\begin{definition}
The \emph{correction term} (or \emph{$d$--invariant}) of $(M,\frak s)$ is denoted $d(M,\frak s)$ and is given by
$$\min\{gr(\pi(\alpha)):\alpha\in HF^\infty(M,\frak s)\}.$$
\end{definition}

The correction term enjoys a number of nice properties, including the fact that $d$ is a $\spinc$ rational homology cobordism invariant  (see  \cite{oz-sz:absolute}):
\begin{enumerate}
\item $d(M_1\#M_2, \frak s_1\#\frak s_2) = d(M_1,\frak s_1)+d(M_2,\frak s_2)$,
\item $d(-M,\frak s) = -d(M,\frak s)$, where $-M$ denotes $M$ with the opposite orientation, and
\item $d(M,\frak s) = 0$ whenever  $(M,\frak s) = \partial (W,\frak t)$, where $W$ is a rational-homology 4--ball, and $\frak t|_{\partial W} = \frak s$.  
\end{enumerate}
This last property is key in the proof of Theorem \ref{thm:hyp_corr_terms}.

As mentioned above, there are affine identifications $\spinc(M)\cong H^2(M;\Z)$, so a rational homology 3--sphere $M$ will have $|H^2(M)|$ correction terms.  We will denote the collection of correction terms associated to $M$ by $\mathcal D(M)$.  When possible, the group structure of $H^2(M)$ will be implicit in our presentation of $\mathcal D(M)$.  For example, in \cite{oz-sz:absolute} a formula for the correction terms of lens spaces is given.  In particular,

\begin{equation}\label{eqn:lens_corr_terms}
d(L(p,1),i) = \frac{p-(2i-p)^2}{4p}.
\end{equation}

\begin{example}\label{ex:lens_corr_terms}
Consider the case from Section \ref{section:geometry} when $J$ is unknotted and $n=5$.  Then, $Y=L(5,1)$, and Equation \ref{eqn:lens_corr_terms} tells us that,  
$$\mathcal D(L(5,1)) = \{1,1/5,-1/5,-1/5,1/5\}.$$
By the additivity of the correction terms, we have the following:
$$\mathcal D(L(5,1)\#L(5,-1)) = \left\{
\begin{array}{ccccc}
0 & 4/5 & 6/5 & 6/5 & 4/5 \\
-4/5 & 0 & 2/5 & 2/5 & 0 \\
-6/5 & -2/5 & 0 & 0 & -2/5 \\
-6/5 & -2/5 & 0 & 0 & -2/5 \\
-4/5 & 0 & 2/5 & 2/5 & 0 
\end{array}
\right\}.$$

Note that implicit in the presentation matrix is the affine identification $\spinc(L(n,1)\#L(n,-1))\cong\Z_5\oplus\Z_5$ given by $[\frak s_i,\frak s_j]\sim (i,j)$.  For example, the correction terms vanish on all elements of the subgroups generated by $(1,1)$ and $(1,4)$ in $\Z_5\oplus\Z_5$.

It will sometimes be helpful to write such collections as follows:
$$\mathcal D(L(5,1)) = \{-1/5, 1/5,1,1/5,-1/5\}.$$
and
$$\mathcal D(L(5,1)\#L(5,-1)) = \left\{
\begin{array}{ccccc}
0 & -2/5 & -6/5 & -2/5 & 0 \\
2/5 & 0 & -4/5 & 0 & 2/5 \\
6/5 & 4/5 & 0 & 4/5 & 6/5 \\
 2/5 & 0 & -4/5 & 0 & 2/5 \\
0 & -2/5 & -6/5 & -2/5 & 0 \\
\end{array}
\right\}.$$
The only difference here, is that we have centered our indexing set about zero, using 
$$\{-(p-1)/2, -(p-3)/2,\ldots,-1,0,1,\ldots (p-3)/2,(p-1)/2\}$$
to index $\Z_p$ instead of \{$0,1,2,\ldots, p-1\}$.  
\end{example}

\subsection{The surgery exact triangle and 4--dimensional cobordisms}\label{subsec:cobordism}\ 
A $\spinc$--cobordism between two $\spinc$ 3--manifolds induces certain maps between the Heegaard Floer homology groups associate to the two manifolds.  We now turn our attention to some aspects of these induced maps.

Let $M$ be a rational homology 3--sphere, and let $K$ be a null-homologous knot in $M$.  Let $M_0$ be the result of $N$--surgery on $K$, and let $M_1$ be the result of $(N+1)$--surgery on $K$.  This is a special case of a broader context in which the triple $(M, M_0, M_1)$ is called a \emph{triad}.  For a discussion relevant to this subsection, see \cite{oz-sz:lectures}.  Implicit in this set up is a triple of cobordisms obtained by 2--handle addition (cf. Subsection \ref{subsec:triad}).

$$\begin{tikzpicture}
  \matrix (m) [matrix of math nodes,row sep=4em,column sep=4em,minimum width=2em] {
     M& \  & M_1 \\
     \  & M_0 & \  \\};
  \path[-stealth]
    (m-1-1) edge node [pos=.6,left] {$W_1\ $} (m-2-2)
    (m-2-2) edge node [pos=.4,right] {$\ W_2$} (m-1-3)
    (m-1-3) edge node [above] {$W_3$} (m-1-1);
\end{tikzpicture}$$

\begin{theorem}
Let $(M, M_0, M_1)$ be a triad, then there exist exact triangles relating their Heegaard Floer homologies:
$$\begin{array}{cc}
\begin{tikzpicture}
  \matrix (m) [matrix of math nodes,row sep=2.8em,column sep=1.25em,minimum width=1em] {
     \HFhat(M)& \  & \HFhat(M_1) \\
     \  & \HFhat(M_0) & \  \\};
  \path[-stealth]
    (m-1-1) edge node [pos=.6,left] {$\widehat F_1\ $} (m-2-2)
    (m-2-2) edge node [pos=.4,right] {$\ \widehat F_2$} (m-1-3)
    (m-1-3) edge node [above] {$\widehat F_3$} (m-1-1);
\end{tikzpicture}
&
\begin{tikzpicture}
  \matrix (m) [matrix of math nodes,row sep=3em,column sep=1.25em,minimum width=1em] {
     HF^+(M)& \  & HF^+(M_1) \\
     \  & HF^+(M_0) & \  \\};
  \path[-stealth]
    (m-1-1) edge node [pos=.6,left] {$F^+_1\ $} (m-2-2)
    (m-2-2) edge node [pos=.4,right] {$\ F^+_2$} (m-1-3)
    (m-1-3) edge node [above] {$F^+_3$} (m-1-1);
\end{tikzpicture}
\end{array}$$
These maps are induced by the 2--handle cobordisms relating the triad.  
\end{theorem}

Moreover, the grading shifts associated to these induced maps are given by the following formula:
$$gr(F^\circ_i) = gr(F^\circ_i(x))-gr(x) = \frac{(c_1(\frak t))^2 -2\chi(W_i)-3\sigma(W_i)}{4}.$$

This set-up can be applied to the 3--manifolds and 4--dimensional cobordisms introduced in Section \ref{section:geometry}.  Below, we will use these exact triangles to understand the Heegaard Floer homology of $Z$ (i.e., $M_1$) via the Heegaard Floer homology of $X$ and $Y$ (i.e., $M$ and $M_0$), which are more tractable, since they are each realized by surgery on knots in $S^3$.

In addition to this nice set-up, we have two important theorems about the behavior of these maps on certain types of cobordisms.

\begin{theorem}[\cite{oz-sz:absolute}]\label{thm:infty_isom}
Let $W$ be a cobordism between rational homology 3--manifolds obtained by surgery on a knot such that $b_2^+(W)=0$.  Then $F^\infty_{W,\frak t}$ is an isomorphism for all $\frak t\in\spinc(W)$.
\end{theorem}

The following theorem is implicit in the work of Ozsv\'ath and Szab\'o \cite{oz-sz:absolute}, and can also be found in \cite{lisca-stipsicz}.

\begin{theorem}\label{thm:hat_vanish}
Let $W$ be a cobordism induced by attaching a 2--handle to a rational homology 3--sphere, and let $\frak t\in\spinc(W)$.  Suppose that $W$ contains a smoothly embedded, closed, orientable surface $\Sigma$ with $g(\Sigma)>0$ such that
$$\Sigma\cdot\Sigma\geq 0 \text{ and } |\langle c_1(\frak t),[\Sigma]\rangle| + \Sigma\cdot\Sigma > 2g(\Sigma) - 2.$$
Then $\widehat F_{W,\frak t}$ is zero.
\end{theorem}

For example, if the 2--handle attachment occurs along a knot with large positive framing relative to its genus, the induced map $\widehat F_{W,\frak t}$ will vanish for all $\frak t\in\spinc(W)$.




\subsection{Knot complexes}\ 

A  rationally null-homologous knot $K$ in $M$ induces a second filtration on $CF^\infty(M,\frak s)$, which thus becomes a $\Z\oplus\Z$--filtered, $\Q$--graded complex, and is denoted $CFK^\infty(M,K,\frak s)$.  The action of $U$ lowers both filtrations by one, and lowers the grading by two.  For our purposes, the most important aspect of this complex is that it can be used to determine the Heegaard Floer homology of surgeries on $K$.

For a positive integer $p$, let $\frak s_m$ denote the element of $\spinc(S^3_p(K))$ which is $\spinc$ cobordant to the unique $\spinc$ structure on $S^3$ via an element $\frak t_m\in\spinc(W)$ (where $W$ is the 2--handle cobordism induced by $p$--surgery) satisfying
$$\langle c_i(\frak t_m, [S]\rangle +p = 2m,$$
where $S$ denotes a Seifert surface for $K$, capped off with the core of the 2--handle.  Then the following theorem is stated as in \cite{HLR}, but is originally proved in \cite{oz-sz:knots}.

\begin{theorem}\label{thm:surgery_formula}
Let $K$ be a knot in $S^3$, and suppose that $g(K)=g$.  Let $p\geq 2g-1$.  Then for all $m$ satisfying $|m|\leq \frac{1}{2}(p-1)$, there is a chain homotopy equivalence of graded complexes over $\F[U]$:
$$CF_k^+(S^3_p(K),\frak s_m)\simeq CFK^\infty_l(M,K,\frak s)_{\{\max(i,j-m)\geq 0\}},$$
where
$$k=l+\frac{p-(2m-p)^2}{2p}.$$
\end{theorem}

Equation \ref{eqn:lens_corr_terms} can be viewed a special case of this (i.e., when $K$ is the unknot).  We make extensive use of this theorem in the calculations required by the proof in Section \ref{section:calculations}, which can be found in Appendix \ref{appendix}.

One corollary of this set-up is that the correction terms of manifolds obtained by surgery on knots can be compared to those of lens spaces.  We refer the reader to \cite{ni-wu:cosmetic,ni-wu:rational} for a nice development.  In short,  by considering $CFK^\infty(S^3,K)$, one can define two sequences of nonnegative integers $V_k, H_k$ for $k\in\Z$ satisfying
$$V_k = H_{-k},\hspace{.25in} V_k\geq V_{k+1}\geq V_k-1, \hspace{.25in} V_{g(K)}=0.$$
It turns out that the correction terms of surgeries on $K$ are determined by these integers.

\begin{theorem}\label{thm:ni-wu}
Let $K$ be a knot in $S^3$.  Then,
$$d(S^3_p(K),i) = d(L(p,1),i) - 2\max\{V_i,H_{i-p}\}.$$
\end{theorem}


\section{Proof of Theorem \ref{thm:main}}\label{section:calculations}

In this section, we prove the following proposition, whose corollary, together with Corollary \ref{coro:doubly_slice}, implies Theorem \ref{thm:main}.  Recall the geometric set-up from Section \ref{section:geometry}. In particular, let $\mathcal Z_{p,k_p}$ be the double branched cover of the knot $\mathcal K_{p,k_p}$.

\begin{proposition}\label{prop:X_corr_terms}
The difference $\mathcal D(L(p,1)\#L(p,-1)) - \mathcal D(\mathcal Z_{p,k_p})$ is given by the following matrix $\mathcal M$.
$$
{
\left[
\setlength{\arraycolsep}{3.5pt}
\begin{array}{cccccccccccccccccccccc}
0 & \cdots & 0 & 0 & 0 & 0 & 0 & 0 & 0 & 0 & 0 & 0 & 0 & 0 & \cdots & 0 & 0&  \cdots & 0 \\
\vdots & & \vdots &\vdots & \vdots & \vdots & \vdots & \vdots & \vdots & \vdots & \vdots & \vdots & \vdots & \vdots && \vdots& \vdots & & \vdots \\
0 & \cdots & 0 & 0 & 0 & 0 & 0 & 0 & 0 & 0 & 0 & 0 & 0 & 0 & \cdots & 0 & 0&  \cdots & 0 \\
2 & \cdots  & \textcolor{Black}{2} &  0 & \textcolor{Black}{0} & 0 & 0 & 0 &0 & 0 & 2 & 2 & 2 & 2 &  \cdots & 2 & 2 &  \cdots & 2 \\
2 & \cdots  & 2 & \textcolor{Black}{2} & 0 & \textcolor{Black}{0} & 0 & 0 & 0 & 0 & 2  & 2 & 2 & 2 & \cdots & 2 & 2 &  \cdots & 2 \\
4 & \cdots  & 4 & 4 & \textcolor{Black}{2} & 0 & \textcolor{Black}{0} & 0 & 0 & 0 & 4 & 4 & 4 & 4 & \cdots & 4 & 4 &  \cdots & 4 \\
4 & \cdots  & 4 & 4 & 4 & \textcolor{Black}{2} & 0 & \textcolor{Black}{0} & 0 & 0 & 4 & 4 & 4 & 4 & \cdots & 4 & 4 &  \cdots & 4\\
\vdots &  & \vdots & \vdots &  &  & \textcolor{Black}{\ddots} & \ddots & \textcolor{Black}{\ddots} & & \vdots & \vdots & \vdots &\vdots & & \vdots & \vdots& & \vdots \\
2k_p & \cdots  & 2k_p & 2k_p & \cdots & 2k_p & 2k_p & \textcolor{Black}{2} & 0 & \textcolor{Black}{0} & 2k_p & 2k_p & 2k_p & 2k_p & \cdots & 2k_p & 2k_p  & \cdots & 2k_p \\
2k_p & \cdots  & 2k_p & 2k_p & \cdots & 2k_p & 2k_p & 2k_p & \textcolor{Black}{2} & 0 & \textcolor{Black}{2} & 2k_p & 2k_p &2k_p &  \cdots & 2k_p & 2k_p & \cdots & 2k_p \\
2k_p & \cdots  & 2k_p & 2k_p & \cdots & 2k_p & 2k_p & 2k_p & 2k_p & \textcolor{Black}{0} & 0 & \textcolor{Black}{2} & 2k_p & 2k_p & \cdots & 2k_p & 2k_p & \cdots & 2k_p \\
\vdots &  & \vdots & \vdots & & \vdots & \vdots  & \vdots  & \vdots & & \textcolor{Black}{\ddots} & \ddots & \textcolor{Black}{\ddots} & & & \vdots   & \vdots&  & \vdots \\
4 & \cdots  & 4 & 4 & \cdots & 4 & 4 & 4 & 4 & 0 & 0 & \textcolor{Black}{0} & 0 & \textcolor{Black}{2} & 4 & 4 & 4 &  \cdots & 4 \\
4 & \cdots  & 4 & 4 & \cdots& 4 & 4 & 4 &4 & 0 & 0 & 0 & \textcolor{Black}{0} & 0 & \textcolor{Black}{2} & 4 & 4 &  \cdots & 4 \\
2 & \cdots  & 2 & 2 &  \cdots& 2 & 2 & 2 & 2 & 0 & 0 & 0 & 0 & \textcolor{Black}{0} & 0 & \textcolor{Black}{2} & 2 &  \cdots & 2 \\
2 & \cdots & 2 & 2 & \cdots & 2 & 2 & 2 & 2& 0 & 0& 0 & 0 & 0 & \textcolor{Black}{0} & 0 & \textcolor{Black}{2} &  \cdots & 2 \\
0 & \cdots  & 0 & 0 & \cdots & 0 & 0 & 0 &0 & 0 & 0 & 0 & 0 &0 & 0 & 0 & 0 &  \cdots & 0 \\
\vdots &  & \vdots & \vdots & & \vdots & \vdots & \vdots  & \vdots & \vdots & \vdots & \vdots & \vdots & \vdots& \vdots & \vdots & \vdots & & \vdots \\
0 & \cdots  & 0 & 0 & \cdots & 0 & 0 & 0 & 0 & 0 & 0 & 0 & 0 & 0 & 0 & 0 & 0 &  \cdots & 0 \\
\end{array}
\right].
}
$$
\end{proposition}

This matrix presentation makes use of the affine identification $\spinc(\mathcal Z_{p,k_p})\cong H^2(\mathcal Z_{p,k_p})\cong\Z_p\oplus\Z_p$, where $(i,j)\in\Z_p\oplus\Z_p$ is such that $|i|,|j|\leq \frac{p-1}{2}$.  There is an indeterminacy present that must be discussed.  In Appendix \ref{appendix}, the calculation of the correction terms for $Y =S^3_{p^2+p}(J\#J\#T_{p,p+1})$ (with $J=\#_{k_p}D$) is done in a way that forgets the explicit identification of $\spinc(Y)\cong H^2(Y)\cong\Z_p\oplus\Z_{p+1}$.  Thus, we lose track of the difference between $j$ and $-j$ in $\Z_{n+1}$ and between $i$ and $-i$ in $\Z_n$.  As a consequence, when regarding the matrix above, we must consider it only up to horizontal reflection about the central column and vertical reflection about the central row.  This indeterminacy is inconsequential in what follows.  In particular, the following corollary holds.

\begin{corollary}\label{coro:terms}
The manifold $\mathcal Z_{p,k_p}$ has precisely $2p-2k_p-1$ vanishing correction terms.  Therefore, $\mathcal K_{p,k_p}$ is not smoothly doubly slice.  Moreover, the $\mathcal K_{p,k_p}$ are nontrivial in $\mathcal C_\mathcal D$.
\end{corollary}

\begin{proof}
The $(p\times p)$--matrix for $\mathcal D(L(p,1)\#L(p,-1))$ has zeros along the two (orthogonal) diagonals and non-integer rational numbers elsewhere. The corresponding entries in the matrix for $\mathcal D(Z_{p,k_p})$ are lowered by even integers, corresponding to the matrix $\mathcal M$ by Proposition \ref{prop:X_corr_terms}.  It is easy to see that precisely $2k_p$ of the $2p-1$ vanishing entries in $\mathcal D(L(p,1)\#L(p,-1))$ will be lowered by a nonzero amount.  These changes correspond to the non-zero entries of $\mathcal M$ on the cross-diagonal.  Since all entries are changed by an even integer, no new zeros will be created.  Therefore, $\mathcal Z_{p,k_p}$ has precisely $2p-2k_p-1$ vanishing correction terms.  By Theorem \ref{thm:hyp_corr_terms}, this implies that the $\mathcal K_{p,k_p}$ are not smoothly doubly slice.  In fact, by Proposition \ref{prop:stable_terms}, this implies that each $\mathcal K_{p,k_p}$ is not even smoothly stably doubly slice, since $\det(\mathcal K_{p,k_p}) = p^2$.  Therefore, each $\mathcal K_{p,k_p}$ represents a nontrivial element in $\mathcal C_\mathcal D$.
\end{proof}

\subsection{Notation and set-up}\ 

Let $X=S^3_n(K)$, and let $[\frak s_i]\in\spinc(X)$ be the enumeration of $\spinc(X)$ introduced in Subsection \ref{subsec:enumerate}.  Then we have the following decomposition:
$$HF^\infty(X) = \bigoplus_{i=0}^{n-1} HF^\infty(X,\frak s_i)=\bigoplus_{i=0}^{n-1}\mathcal T^\infty_i(X).$$
Note that here and throughout, subscripts will correspond to the labelings of $\spinc$ structures on the manifolds.  Theorem \ref{thm:surgery_formula} implies that, for any $x\in\mathcal T^\infty_i(X)$,
$$gr(x)\equiv d(L(n,1),i) \pmod 2$$
for all $i\in\Z_n$.  Let $\bar x^\infty_i$ denote the element in $\mathcal T^\infty_i(X)$ such that
$$gr(\bar x^\infty_i)=d(L(n,1),i).$$

Let $Y = X_{-n-1}(K)$ for a null homologous knot $K$ in $X$, and let $[\frak s_i,\frak s_j]\in\spinc(Y)$ be the enumeration of $\spinc(Y)$, as in Subsection \ref{subsec:enumerate}.  This gives the following decomposition:
$$HF^\infty(Y) = \bigoplus_{i=0}^{n-1}\bigoplus_{j=0}^nHF^\infty(Y,[\frak s_i, \frak s_j]) = \bigoplus_{i=0}^{n-1}\bigoplus_{j=0}^n\mathcal T^\infty_{i,j}(Y).$$

Let $F^\infty_{W_1,[\frak s_i,\frak t_m]}:HF^\infty(X,\frak s_i)\to HF^\infty(Y, [\frak s_i,\frak s_j])$ be the map induced by $(W_1,[\frak s_i,\frak t_m])$, as in Subsection \ref{subsec:cobordism}.  Since $W_1$ is negative definite, we can conclude (see \cite{oz-sz:absolute}) that $F_{W_1,\frak t}^\infty$ is an isomorphism for all $\frak t\in\spinc(W_1)$.  Furthermore, 
$$gr\left(F^\infty_{W_1,[\frak s_i,\frak t_m]}\right)= \frac{(n+1)-(2m+(n+1))^2}{4(n+1)}$$
for each $i\in\Z_n$.  In general, if $F$ is any graded map between graded abelian groups, we denote the grading shift of $F$ by $gr(F)$.

\begin{lemma}\label{lemma:mod2}
For all $i\in\Z_n$ and $j\in\Z_{n+1}$, let $y$ be any element in $\mathcal T_{i,j}^\infty(Y)$, then $$gr(y)\equiv gr(L(n,1),i)-gr(L(n+1,1),j) \pmod 2.$$
\end{lemma}

\begin{proof}
The fact that $F^\infty_{W_1,[\frak s_i,\frak t_{m}]}$ is an isomorphism, and the labeling of $\spinc$ structures, implies that $F^\infty_{W_1,[\frak s_i,\frak t_m]}(x_i^\infty)\subset\mathcal T^\infty_{i,j}$ if and only if $m\equiv j\pmod{n+1}$.  Let $m=-j$, then, since all elements in $\mathcal T^\infty_{i,j}$ can be obtained from each other by translation by $U$,
\begin{eqnarray*}
gr(y) & \equiv & gr\left(F^\infty_{W_1,[\frak s_i,\frak t_{-j}]}(\bar x_i^\infty)\right) \pmod 2 \\
&  \equiv  &  gr(\bar x_i^\infty) + gr(F^\infty_{W_1,[\frak s_i,\frak t_{-j}]})  \\
& \equiv &  d(L(n,1),i)-d(L(n+1,1),j) 
\end{eqnarray*}
\end{proof}

Let $\bar y^\infty_{i,j}$ denote the element in $\mathcal T^\infty_{i,j}(Y)$ satisfying
$$gr(\bar y_{i,j}^\infty) = d(L(n,1),i)-d(L(n+1,1),j).$$  Using this notation, we gain a precise understanding of the map $$ F^\infty_1 = \sum_{\frak t\in\spinc(W_1)}F^\infty_{W_1,\frak t},$$
given by the following lemma.  Note that $F_1^\infty$ is not a well-defined map to $HF^\infty(Y)$, since its image will generally consist of infinite sums of elements in $HF^\infty(Y)$.  The important fact for us is that all but finitely many of the terms will have coefficients that are large powers of $U$.

\begin{lemma}\label{lemma:F1_infty}
Let the $\bar x_i^\infty$ and $\bar y_{i,j}^\infty$ be defined as above.  Then, for all $i\in\Z_n$,
$$F^\infty_1(\bar x_i^\infty) = (\bar y^\infty_{i,1} + \bar y_{i,2}^\infty + \cdots + \bar y_{i,n}^\infty) + U(\bar y_{i,1}^\infty+\bar y_{i,n}^\infty) + U^2(\bar y_{i,2}^\infty + \bar y_{i,n-1}^\infty)+\cdots,$$
where the expression continues indefinitely with increasing positive powers of $U$ as coefficients.
\end{lemma}

\begin{proof}
The proof of this lemma is a simple examination of $gr(F^\infty_{W_1,[\frak s_i,\frak t_m]})$ as $m$ varies over the integers.  The powers of $U$ in the tail follow a growth pattern that depends quadratically on $n$ in a simple way, but will not be relevant in what follows.
\end{proof}

Continuing, let $Z$ be obtained from $Y$ by blowing down a meridian, as in Subsection \ref{subsec:triad}, let $W_2$ be the induced cobordism, and let $[\frak s_i,\frak s_k]\in\spinc(Z)$ be the enumeration of $\spinc(Z)$, as in Subsection \ref{subsec:enumerate}.  This gives the following decomposition:
$$HF^\infty(Z) = \bigoplus_{i=0}^{n-1}\bigoplus_{k=0}^{n-1}HF^\infty(Z,[\frak s_i, \frak s_k]) = \bigoplus_{i=0}^{n-1}\bigoplus_{k=0}^{n-1}\mathcal T^\infty_{i,k}(Z).$$

Let $F^\infty_{W_2,[\frak s_i,\frak r_m]}:HF^\infty(Y,[\frak s_i,\frak s_j])\to HF^\infty(Z, [\frak s_i,\frak s_k])$ be the map induced by $(W_2,[\frak s_i, \frak r_m])$.  Since $W_2$ is negative definite, we can conclude that $F_{W_2,\frak r}^\infty$ is an isomorphism for all $\frak r\in\spinc(W_2)$.  Furthermore, 
$$gr\left(F^\infty_{W_2,[\frak s_i,\frak r_m]}\right) = \frac{n(n+1)-(2m+n(n+1))^2}{4n(n+1)}$$
for each $i\in\Z_n$ and $j\in\Z_{n+1}$.

\begin{lemma}
Let $z$ be any element in $\mathcal T_{i,k}^\infty(Z)$, then $$gr(z)\equiv gr(L(n,1),i)-gr(L(n,1),k) \pmod 2$$
for all $i\in\Z_n$ and $k\in\Z_{n}$.
\end{lemma}

\begin{proof}
This proof is identical to that of Lemma \ref{lemma:mod2}.
\end{proof}

Let $\bar z^\infty_{i,k}$ denote the element in $\mathcal T^\infty_{i,k}(Z)$ satisfying
$$gr(\bar z_{i,k}^\infty) = d(L(n,1),i)-d(L(n,1),j).$$  Using this notation, we gain a precise understanding of the map $$ F^\infty_2 = \sum_{\frak r\in\spinc(W_2)}F^\infty_{W_2,\frak r},$$
in an analogous way to Lemma \ref{lemma:F1_infty}.  From this point on, we will index $H_1(X)\cong\Z_n,H_1(Y)\cong\Z_n\oplus\Z_{n+1}$, and $H_1(Z)\cong\Z_n\oplus\Z_n$ by $i,(i,j)$, and $(i,k)$ (respectively), such that $-\frac{n-1}{2}\leq i,k\leq \frac{n-1}{2}$ and $-\frac{n+1}{2}\leq j\leq \frac{n-1}{2}$.

\begin{lemma}\label{lemma:F2_infty}
Let the $\bar y_{i,j}^\infty$ and $\bar z_{i,k}^\infty$ be defined as above.  Then, for all $i\in\Z_n$,
\begin{eqnarray*}
F^\infty_2(\bar y_{i,0}^\infty) & = & \bar z^\infty_{i,0} + U(\bar z_{i,1}^\infty+\bar z_{i,n-1}^\infty) + U^5(\bar z_{i,2}^\infty+ \bar z_{i,n-2}^\infty)+\cdots, \\
F^\infty_2(\bar y_{i,j}^\infty) & = & \bar z^\infty_{i,j-1}+ \bar z^\infty_{i,j} + U(\bar z_{i,j+1}^\infty) + U^3(\bar z_{i,j-2}^\infty)+  \cdots,
\end{eqnarray*}
if $|j|=1$, and
\begin{eqnarray*}
\hspace{.7in} F^\infty_2(\bar y_{i,j}^\infty) & = & \bar z^\infty_{i,j-1}+ \bar z^\infty_{i,j} + U(\bar z_{i,j-2}^\infty+ \bar z_{i,j+1}^\infty) + U^3(\bar z_{i,j-3}^\infty+ \bar z_{i,j+2}^\infty)+\cdots,
\end{eqnarray*}
for $|j|>1$. The expressions continue indefinitely with increasing positive powers of $U$ as coefficients.
\end{lemma}

\begin{proof}
This proof is the same as that of Lemma \ref{lemma:F1_infty}.
\end{proof}

Let $\pi:HF^\infty(M,\frak s)\to HF^+(M,\frak s)$, be the natural projection map.  Let $\bar x_i^+ = \pi(\bar x_i^\infty)$, and define  $\bar y_{i,j}^+$ and $\bar z_{i,k}^+$ similarly.  Analogous to the discussion above, we have the following decomposition:
$$HF^+(X)/HF_{red}(X) = \bigoplus_{i=-\frac{n-1}{2}}^{\frac{n-1}{2}}\mathcal T_i^+(X),$$
as well as similar decompositions corresponding to $Y$ and $Z$.  Note that we are not claiming that $\bar x_i^+$ is nonzero in $\mathcal T_i^+(X)$.  Similarly, it may be that $\bar y_{i,j}^+$ and the $\bar z_{i,k}^+$ vanish.  Define
$$F_1^+  = \sum_{\frak t\in\spinc(W_1)}F^+_{W_1,\frak t},$$
and
$$F_2^+  = \sum_{\frak r\in\spinc(W_2)}F^+_{W_2,\frak r}.$$

\subsection{Proof of Proposition \ref{prop:X_corr_terms}}\ 

With this notational set-up, we recall that the triad $(X,Y,Z)$ introduced in Section \ref{section:geometry} induces certain long exact sequence (discussed in Section \ref{section:HF}), which will be used below in the proof of Proposition \ref{prop:X_corr_terms}.

Let $J=\#_{k_p}D$, so $X=S^3_p(J\#J)$, $Y = S^3_{p^2+p}(J\#J\#T_{p,p+1})$, and $Z=\mathcal Z_{p,k_p}=\Sigma_2(I_{\#_{k_p}D,p})$.  The calculations made in Appendix \ref{appendix} give us the correction terms for $X$ and $Y$.  In particular, Lemma \ref{lemma:Y_corr_terms} tells us that $\mathcal D(L(p,1)) - \mathcal D(X)$ is given by
$$\vec w = \{0,\ldots, 0,2,2,4,4,\ldots,2k_p-2,2k_p-2,2k_p,2k_p,2k_p,2k_p-2,2k_p-2,\ldots,4,4,2,2,0,\ldots, 0\},$$
where  $2w_i$ is the value of the $i^\text{th}$ coordinate of $\vec w$ for $i\in\Z$ with our symmetric labeling.

Let $x_i^\infty = U^{w_i}\bar x_i^\infty$, and let $\pi(x_i^\infty) = x_i^+$.  It follows that $x_i^+$ is the element of lowest grading in $\mathcal T_i^+(X)$, i.e., $gr(x_i^+) = d(X,\frak s_i)$.  Similarly, by Corollary \ref{coro:Y'_corr_terms}, $\mathcal D(L(n,1)\#L(n+1,-1)) - \mathcal D(Y)$ is given by the matrix $\mathcal M = (2m_{i,j})$, which has the following form.
$$
{
\left[
\setlength{\arraycolsep}{2.3pt}
\begin{array}{cccccccccccccccccccccc}
0 & \cdots & 0& 0 & 0 & 0 & 0 & 0 & 0 & 0 & 0 & 0 & 0 & 0 & 0 & \cdots & 0 & 0& 0& \cdots & 0 \\
\vdots & & \vdots & \vdots &\vdots & \vdots & \vdots & \vdots & \vdots & \vdots & \vdots & \vdots & \vdots & \vdots & \vdots && \vdots& \vdots & \vdots& & \vdots \\
0 & \cdots & 0 & 0 & 0 & 0 & 0 & 0 & 0 & 0 & 0 & 0 & 0 & 0 & 0 & \cdots & 0 & 0& 0& \cdots & 0 \\
2 & \cdots & 2 & 2 & 0 & 0 & 0 & 0 & 0 &0 & 0 & 2 & 2 & 2 & 2 &  \cdots & 2 & 2 & 2& \cdots & 2 \\
2 & \cdots & 2 & 2 & 2 & 0 & 0 & 0 & 0 & 0 & 0 & 2  & 2 & 2 & 2 & \cdots & 2 & 2 & 2& \cdots & 2 \\
4 & \cdots & 4 & 4 & 4 & 2 & 0 & 0 & 0 & 0 & 0 & 4 & 4 & 4 & 4 & \cdots & 4 & 4 & 4& \cdots & 4 \\
4 & \cdots & 4 & 4 & 4 & 4 & 2 & 0 & 0 & 0 & 0 & 4 & 4 & 4 & 4 & \cdots & 4 & 4 & 4& \cdots & 4\\
\vdots &  & \vdots & \vdots & \vdots &  &  & \ddots & \ddots & \ddots & & \vdots & \vdots & \vdots &\vdots & & \vdots & \vdots& \vdots& & \vdots \\
2k_p & \cdots & 2k_p & 2k_p & 2k_p & \cdots & 2k_p & 2k_p & 2 & 0 & 0 & 2k_p & 2k_p & 2k_p & 2k_p & \cdots & 2k_p & 2k_p & 2k_p & \cdots & 2k_p \\
2k_p & \cdots & 2k_p & 2k_p & 2k_p & \cdots & 2k_p & 2k_p & 2k_p & 2 & 0 & 2 & 2k_p & 2k_p &2k_p &  \cdots & 2k_p & 2k_p & 2k_p& \cdots & 2k_p \\
2k_p & \cdots & 2k_p & 2k_p & 2k_p & \cdots & 2k_p & 2k_p & 2k_p & 2k_p & 0 & 0 & 2 & 2k_p & 2k_p & \cdots & 2k_p & 2k_p & 2k_p& \cdots & 2k_p \\
\vdots & & \vdots & \vdots & \vdots & & \vdots & \vdots  & \vdots  & \vdots & & \ddots & \ddots & \ddots & & & \vdots   & \vdots& \vdots&  & \vdots \\
4 & \cdots & 4 & 4 & 4 & \cdots & 4 & 4 & 4 & 4 & 0 & 0 & 0 & 0 & 2 & 4 & 4 & 4 & 4& \cdots & 4 \\
4 & \cdots & 4 & 4 & 4 & \cdots& 4 & 4 & 4 &4 & 0 & 0 & 0 & 0 & 0 & 2 & 4 & 4 & 4& \cdots & 4 \\
2 & \cdots & 2 & 2 & 2 &  \cdots& 2 & 2 & 2 & 2 & 0 & 0 & 0 & 0 & 0 & 0 & 2 & 2 & 2& \cdots & 2 \\
2 & \cdots & 2 & 2 & 2 & \cdots & 2 & 2 & 2 & 2& 0 & 0& 0 & 0 & 0 & 0 & 0 & 2 & 2& \cdots & 2 \\
0 & \cdots & 0 & 0 & 0 & \cdots & 0 & 0 & 0 &0 & 0 & 0 & 0 & 0 &0 & 0 & 0 & 0 & 0& \cdots & 0 \\
\vdots & & \vdots & \vdots & \vdots & & \vdots & \vdots & \vdots  & \vdots & \vdots & \vdots & \vdots & \vdots & \vdots& \vdots & \vdots & \vdots & \vdots& & \vdots \\
0 & \cdots & 0 & 0 & 0 & \cdots & 0 & 0 & 0 & 0 & 0 & 0 & 0 & 0 & 0 & 0 & 0 & 0 & 0& \cdots & 0 \\
\end{array}
\right]
}
$$
Note that the values in the $i^\text{th}$ row of $\mathcal M$ are bounded above by $2w_i$.  (Remember that the rows are labeled by $\Z_n$ symmetrically about zero, and the columns are labeled by $\Z_{n+1}$ by $j\in[-\frac{n+1}{2},\frac{n-1}{2}]$).  We remark again that the calculation given in the proof of Corollary \ref{coro:Y'_corr_terms} introduces an indeterminacy regarding our presentation of the correction terms.  Namely, we cannot distinguish between $i$ and $-i$ and $j$ and $-j$ in the present labeling.  This indeterminacy is merely notational and will not affect the results.

Let $y_{i,j}^\infty = U^{m_{i,j}}\bar y_{i,j}^\infty$, and let $y_{i,j}^+ = \pi(y_{i,j}^\infty)$.  It follows that $y_{i,j}^+$ is the element of lowest grading in $\mathcal T_{i,j}^+(Y)$, i.e., $gr(y_{i,j}^+) = d(Y,[\frak s_i, \frak s_j])$.  With this notational set-up, we can prove the following lemma about the map $F_1^+:HF^+(X)\to HF^+(Y)$.

\begin{lemma}
Let $x_i^+\in\mathcal T_i(X)$ and $y_{i,j}^+\in\mathcal T_{i,j}^+(Y)$ be elements of lowest grading in their respective towers.  Then,
$$F_1^+(x_i^+) = \sum_{j\in\mathcal I_i}y_{i,j}^+,$$
where $\mathcal I_i = \{j\not=0\ :\ m_{i,j}=w_i\}$.
\end{lemma}

\begin{proof}
By Lemma \ref{lemma:F1_infty} we have that
$$F_1^\infty(\bar x_i^\infty) = \sum_{j\not=0}\bar y_{i,j}^\infty + \mathcal U(\bar y^\infty_{i,j}),$$
where $\mathcal U(\bar y_{i,j}^\infty)$ represents the terms that are positive $U$--translates of the $\bar y_{i,j}^\infty$.  By $U$--equivariance, we have
$$F_1^\infty(x_i^\infty) = U^{w_i}F_1^\infty(\bar x_i^\infty) = \sum_{j\not=0}U^{w_i}\bar y_{i,j}^\infty + U^{w_i}\mathcal U(\bar y_{i,j}^\infty).$$
Since $F_1$ commutes with the natural projection $\pi$ (which is $U$--equivariant), we see that
$$F_1^+(x_i^+) = \pi(F_1^\infty(x_i^\infty)) = \sum_{j\not=0}U^{w_i}\pi(\bar y^\infty_{i,j}) =  \sum_{j\not=0}U^{w_i}\bar y^+_{i,j},$$
where the tail has vanished, by $U$--equivariance.  By definition, we have $U^{w_i}\bar y_{i,j}^+ = U^{w_i-m_{i,j}}y_{ij}^+$, and this term will be nonzero if and only if $w_i\leq m_{i,j}$.  This can only happen if $w_i = m_{i,j}$, since, as we noticed above, $m_{i,j}\leq w_i$.
\end{proof}

Note that $|\mathcal I_i|\geq\frac{p+1}{2}$ for each $i$; so, in particular, $F_1^+(x_i^+)$ is a linear combination of at least $\frac{p+1}{2}$ terms for each $i$.  One consequence of this is that $y_{i,j}^+$ is not in the image of $F_1^+$ for any $i,j$.

Let $z_{i,j}^+$ denote the element of lowest grading in $\mathcal T_{i,j}^+(Z)$.  We know by $U$--equivariance that $$F_{W_2,[\frak s_i,\frak s_k]}(y_{i,j}^+) = U^{c_{i,k}}z_{i,k}^+$$ for some nonnegative integer $c_{i,k}$.  If we can show that $c_{i,j}=0$ for all $i,k$, we will have proved Proposition \ref{prop:X_corr_terms}, because we will have shown that $z_{i,k}^\infty = U^{m_{i,k}}\bar z_{i,k}^\infty$.  This is accomplished by the following lemma.  Recall the natural inclusion map $\hat\iota:\HFhat(Z)\to HF^+(Z)$.

\begin{lemma}
Let $z_{i,k}^+$ be the element of lowest grading in $\mathcal T_{i,k}^+(Z)$, and let $y_{i,k}^+$ be the element of lowest grading in $\mathcal T_{i,k}^+(Y)$.  Then,
$$gr(z_{i,k}^+) = gr\left(F_{W_2,[\frak s_i,\frak t_k]}^+(y_{i,k}^+)\right).$$
\end{lemma}

\begin{proof}
Let $\hat z\in\HFhat(Z)$ such that $\hat\iota(\hat z) = z_{i,k}^+$.  By Theorem \ref{thm:hat_vanish}, we know that $\widehat F_3(\hat z)=0$.  (Recall that $-\overline{W_3}$ is induced by $(-p)$--surgery on a knot of genus $2k_p$ with $p>4p_k-1$, see Subsection \ref{subsec:coefficients}.)  By the exactness at $\HFhat(X)$, there exists some $\hat y\in\HFhat(Y')$ such that $\widehat F_2(\hat y) = \hat z$.  Now, $\hat y$ may not be homogeneous, so write $\hat y = \sum_a\hat y_a$, where each $\hat y_a$ is homogeneous and in $\HFhat(Y',[\frak s_i,\frak s_{j_a}])$.  By Lemma \ref{lemma:reduced_bound}, we know that for each $a$, $gr(\hat y_a) \leq gr(y_{i,j_a}^+)$.  So, we have

\begin{eqnarray*}
gr\left(\widehat F_{W_2,[\frak s_i,\frak t_{m_a}]}(\hat y_a)\right) & = & gr(\hat y_a) + gr\left(\widehat F_{W_2,[\frak s_i,\frak t_{m_a}]}\right) \\
& \leq & gr(y_{i,j_a}^+) + gr\left(F^+_{W_2,[\frak s_i,\frak t_{m_a}]}\right) \\
& = & gr\left( F_{W^+_2,[\frak s_i,\frak t_{m_a}]}(y_{i,j_a}^+)\right) \\
& \leq & gr\left( F_{W^+_2,[\frak s_i,\frak t_k]}(y_{i,k}^+)\right), \\
\end{eqnarray*}
where the last inequality follows from the fact that $gr\left(F_{W^+_2,[\frak s_i,\frak t_m]}(y_{i,j}^+)\right)$ is maximized when  by $j$ with $|j|= k$.  (Note that $j_a\equiv k\pmod{p}$.)  Since
$$gr(\hat z) \leq \max_a gr\left(\widehat F_{W_2,[\frak s_i,\frak t_{m_a}]}(\hat y_a)\right),$$
we have
$$gr(\hat z) \leq gr\left( F_{W^+_2,[\frak s_i,\frak t_k]}(y_{i,k}^+)\right).$$
This implies the desired equality once we recall that
$$gr(\hat z) = gr(z_{i,k}^+) \geq gr\left( F_{W^+_2,[\frak s_i,\frak t_k]}(y_{i,k}^+)\right),$$
by $U$--equivariance.
\end{proof}


\section{Proof of Theorem \ref{thm:main2}}\label{section:infinite_order}

In this section, we give a reformulation of one of the invariants introduced in \cite{GRS} for the study of double concordance of knots and use it to find an infinitely generated subgroup in $\ker\psi_\mathcal D$.

Let $A$ be a finite abelian group, so $A$ can be written as the product of cyclic groups.  Let $r_{p,k}(A)$ denote the number of copies of $\Z_{p^k}$ in the decomposition of $A$.  Let $r_p(A) = \sum_{k=1}^\infty r_{p,k}(A)$.  In other words, any generating set for $A$ must contain at least $r_p(A)$ elements of order $p^k$ for some $k\in\N$.

The following definition differs from \cite{GRS} only in the use of $r_p(A)$.


\begin{definition}
Let $K$ be a  knot in $S^3$ and let $p\in\N$ be a positive prime.  Let $M=\Sigma_2(K)$.  Fix an affine identification between $\spinc(M)$ and $A=H^2(M;\Z)$ such that the unique spin structure $\frak s_0$ gets identified with zero in $A$.  Let $\mathcal G_p$ denote the collection of all subgroups of $A$ of order $p$.  Define
$$\frak D_p(K) = \min\left\{ \left|\sum_{H\in\mathcal G_p}n_HS_H(d(M))\right|\ :\ \begin{array}{c} n_H \geq 0 \text{ for all $H$}, \\ \text{ at least $r_p(A)$ of the $n_H$ are nonzero}\end{array}\right\}$$
if $p$ divides $\det(K)$ and $$\frak D_p(K)  =0$$ otherwise, where $S_H(d(M)) = \sum_{h\in H}d(M,h)$.
\end{definition}

The proof of the following theorem is essentially given in \cite{GRS}, but is formulated here for double concordance.

\begin{theorem}\label{thm:GRS}
Let $K\subset S^3$ be a knot and $p\in\N$ a positive prime.  If there is a positive $n\in\N$ such that $\#_nK$ is smoothly doubly slice, then $\frak D_p(K)=0$.
\end{theorem}

\begin{proof}
Suppose that $J=\#_nK$ is smoothly doubly slice.  Let $N =\Sigma_2(J) = \#_n\Sigma_2(K)$.  The identification of $\spinc(\Sigma_2(K))$ with $A$ gives an identification of $\spinc(N)$ with $A^n$.  By Theorem \ref{thm:hyp_corr_terms}, there exists subgroups $G$ and $H$ in $A^n$ such that $G\oplus H=A^n$ and $G\cong H$.

Assume that $p$ divides $\det(K)$, and let $r=r_p(A)$.  Projection onto the first coordinate $\pi:A^n\to A$ is onto, so $\pi(G)+\pi(H) = A$.  Let $a_1, \ldots, a_r$ be linearly independent generators of $A$ of $p$--power order  such that $\pi^{-1}(a_i)\cap(G\cup H)$ is nonempty.  Let $g_i'\in \pi^{-1}(a_i)\cap(G\cup H)$, then $|g_i'| = p^{k_i}q$ for some positive $q\in\Z$ relatively prime to $p$.  Let $g_i = qp^{k_i-1}g_i'$.  Then $\{g_1, \ldots, g_r\}$ is a collection elements of order $p$ in $G\cup H$.  Furthermore, the elements of $\{\pi(g_1), \ldots, \pi(g_r)\}$ are linearly independent in $A$, so, as elements of $\mathcal G_p$, $\langle g_i\rangle = \langle g_j\rangle$ if and only if $i=j$.  Write $g_i = (g_i^1, \ldots, g_i^n)$ for $i=1, \ldots, r$.


By Theorem \ref{thm:hyp_corr_terms}, $d(N, x) = 0$ for all $x\in G\cup H$.  Let $f:A\to\Q$ be given by $f(x) = d(N,x)$, and let $f^{(n)}:A^n\to\Q$ be given by $f(x_1,\ldots,x_n) = f(x_1)+\cdots+f(x_n)$. Since $\langle g_i\rangle<G\cup H$, we have
\begin{eqnarray*} \sum_{m=0}^{p-1}f^{(n)}(mg_i)= 0 & \implies &  \sum_{m=0}^{p-1}\sum_{j=1}^nf(mg_i^j)= 0 \\
& \implies &  \sum_{j=1}^n\sum_{m=0}^{p-1}f(mg_i^j)= 0 \\
& \implies &  \sum_{j=1}^nS_{\langle g_i^j\rangle}(f)= 0 \\
& \implies &  \sum_{j=1}^nS_{\langle g_i^j\rangle}(d(N))= 0
\end{eqnarray*}
Since $\langle g_i^j\rangle\in \mathcal G_p$ for each $j$,
$$\sum_{j=1}^nS_{\langle g_i^j\rangle}(d(N)) = \sum_{H\in\mathcal G_p}n_HS_H(d(N)),$$
with at least one $n_H$ nonzero (since at least $g_i^1$ is nontrivial).  For each $j=1, \ldots, r$, we get a similar linear combination, and, since the $g_j^1$ are independent, each linear combination is nontrivial on a distinct element of $\mathcal G_p$.  Summing, we get
$$\sum_{i=1}^r\sum_{j=1}^nS_{\langle g_i^j\rangle}(d(N)) = \sum_{H\in\mathcal G_p}n_HS_H(d(N)),$$
where at least $r$ of the $n_H$ are nonzero.  It follows that $\mathcal D_p(K)=0$, as desired.

\end{proof}

To prove Theorem \ref{thm:main2}, we will need to understand $S_G(f)$ for each subgroup $G$ of $\Z_p\oplus\Z_p$.  Let $G_\star =\langle(1,1)\rangle$ and let $G_a = \langle (a,a+1)\rangle$ for $a\in\Z_p$.  Then, together, $G_\star$ and the $G_a$ represent the $p+1$ distinct order $p$ subgroups of $\Z_p\oplus\Z_p$. 

First let us consider $Z=L(p,1)\#L(p,-1)$ for a positive prime $p$.  We saw in Subsection \ref{subsec:enumerate}  that we have an affine identification $[\frak s_i,\frak s_j]\sim (i,j)$ between $\spinc(Z)$ and $Z_p\oplus\Z_p$. 

Let $f:H_1(Z)\to\Q$ be given by $f(x)=d(Z,[\frak s_i,\frak s_j])$, where $[\frak s_i,\frak s_j]\sim x$ is the given affine identification.  It is possible to check using Equation \ref{eqn:lens_corr_terms}  that
$$S_a^{\text{lens}}=S_{G_a}(f) = \begin{cases}
\frac{(p-1)(p+1)}{6} & \text{ if $a=0$,} \\
-\frac{(p-1)(p+1)}{6} & \text{ if $a=p-1$,} \\
0 & \text{ if $a=\star$,}\\
0 & \text{ otherwise.}
\end{cases}
$$

Furthermore, by Proposition \ref{prop:X_corr_terms}, we know that 
$$\mathcal D(L(n,1)\#L(n,-1)) - \mathcal D(\mathcal Z_{p,k_p})$$
is given by $\mathcal M$.
Let $S'_G = \sum_{g\in G}\mathcal M_g$, where $\mathcal M_g=\mathcal M_{i,j}$, if  $g=(i,j)\in\Z_p$.  Then, we see that

$$S'_{G_a} = \begin{cases}
2k(p-3)+4 & \text{ if $a=0$,} \\
0 & \text{ if $a=p-1$,} \\
0 & \text{ if $a=\star$,}\\
\text{(large positive number)} & \text{ otherwise.}
\end{cases}
$$

It follows that the pertinent sums for $\mathcal Z_{p,k_p}$ are given by $S^{\mathcal Z_{p,k_p}}_{G_a}(f) = S_a^{\text{lens}} - S'_{G_a}$.  So,
$$S^{\mathcal Z_{p,k_p}}_{G_a} = \begin{cases}
\frac{(p-1)(p+1)}{6}-(2k(p-3)+4) & \text{ if $a=0$,} \\
-\frac{(p-1)(p+1)}{6} & \text{ if $a=p-1$,} \\
0 & \text{ if $a=\star$,}\\
\text{(large negative number)} & \text{ otherwise.}
\end{cases}
$$

The upshot is that $S_{G_a}^{\mathcal Z_{p,k_p}}$ will be strictly negative for all $a\not=\star$ if and only if
$$\frac{(p-1)(p+1)}{6}-(2k(p-3)+4)<0.$$
The left side will be negative if $k\geq \frac{p+5}{12}$.  As we saw above in Subsection \ref{subsec:coefficients}, we will let $k_p=\left\lceil{\frac{p+6}{12}}\right\rceil$, which will satisfy this condition. Now we can prove the following, recalling our set-up from Section \ref{section:geometry}.

\begin{proposition}
Let $\mathcal K_{p,k_p} = I_{\#_{2k_p}J,p}$, where $J$ is $T_{2,3}$ or $D$, and where $k_p=\left\lceil{\frac{p+6}{12}}\right\rceil$. Then,
\begin{enumerate}
\item No knot in the span (under connected sum) of the $\mathcal K_{p,k_p}$ is smoothly doubly slice.
\item Each of the $\mathcal K_{p,k_p}$ has order greater than two in $\mathcal C_{\mathcal D}$.
\item The collection $\{\mathcal K_{p,k_p}\}$ forms a basis for an infinitely generated subgroup of $\mathcal C_\mathcal D$. 
\end{enumerate}
\end{proposition}

Note that this is independent of the indeterminacies $i\leftrightarrow -i$ and $j\leftrightarrow -j$ discussed earlier.  Notice also that Example \ref{ex:rank6} illustrates why we cannot claim that the $\mathcal K_{p,k_p}$ have infinite order in $\mathcal C_{\mathcal D}$.

\begin{proof}
By Corollary \ref{coro:terms}, we know that each of these knots is nontrivial in $\mathcal C_\mathcal D$.   The discussion preceding this proposition shows that the Grigsby-Ruberman-Strle invariant $\mathcal D_p$ is nonzero for $\mathcal K_{p,k_p}$.  This follows because, for these knots, $S^{\mathcal Z_{p,k_p}}_{G_a}$ is nonnegative for only one subgroup of $\Z_p\oplus\Z_p$.  Since the condition on $\mathcal D_p$ states that $n_G$ must be nonzero for at least two distinct subgroups $G$, the sum $\sum_{G\in\mathcal G_p}n_GS_G(M)$ will always be nonzero.  By Theorem \ref{thm:GRS}, this shows that $\#_a\mathcal K_{p,k_p}$ is not doubly slice for all $a\in \N$.  By Proposition \ref{prop:stable_terms}, $\mathcal K_{p,k_p}\#\mathcal K_{p,k_p}$ is nontrivial in $\mathcal C_\mathcal D$, since $\mathcal A_p$ is rank 4 and not hyperbolic for these knots.

Suppose that 
$$\mathcal K = \left(\#_{n_{p_1}}\mathcal K_{p_1,k_{p_1}}\right)\#\left(\#_{n_{p_2}}\mathcal K_{p_2,k_{p_2}}\right)\#\cdots\#\left(\#_{n_{p_m}}\mathcal K_{p_m,k_{p_m}}\right).$$
Since the $p_i$ are all distinct primes, we get that $\mathcal D_{p_i}(K) = \mathcal D_{p_i}(\#_{n_i}\mathcal K_{p_i,k_{p_i}})$.  It is easy to see that, for the knots in question, $\mathcal D_{p_i}(\#_{n_i}\mathcal K_{p_i,k_{p_i}}) \not=0$, since $S^{\mathcal Z_{p_i,k_{p_i}}}_{G_a}$ is always nonpositive and strictly negative away from a single metabolizer.  It follows that $\mathcal D_{p_i}(\mathcal K)\not=0$.  This proves that $K$ is not doubly slice, and if any of the $n_{p_i}$ are less than 3, then $\mathcal K$ is nontrivial in $\mathcal C_\mathcal D$.
\end{proof}

By Corollary \ref{coro:doubly_slice}, each member of $\{\mathcal K_{p,k_p}\}$ is topologically doubly slice.  It follows that these knots generate an infinitely generated subgroup of $\ker\psi_\mathcal D$ that consists of knots that are not smoothly doubly slice.  This proves Theorem \ref{thm:main2}.


\appendix

\section{Assorted knot Floer complex calculations}\label{appendix}

The goal of this appendix is to perform the correction term calculations required by the proof in Section \ref{section:calculations}.  Throughout, we will let $J = \#_mK$ be the connected sum of $m$ copies of $K$, where $K$ will always be one of three knots: the unknot; the right-handed trefoil $T_{2,3}$; or the positive, untwisted Whitehead double of the right-handed trefoil $D$.  Let $X = S^3_n(J\#J)$ and $Y=S^3_{n^2+n}(J\#J\#T_{n,n+1})$; throughout, $n$ will be a positive odd number.  The following facts are collected from two theorems of Hedden, Kim, and Livingston \cite[Proposition 6.1, Theorem B.1]{HKL}, and are the basis what follows.  We will work with coefficients in $\F_2$ throughout.

\begin{theorem}[\cite{HKL}]\label{thm:HKL}\ 
\begin{enumerate}
\item The chain complex $CFK^\infty(S^3, D)$ is filtered chain homotopy equivalent to $CFK^\infty(S^3, T_{2,3})\oplus \mathcal A$, where $\mathcal A$ is an acyclic subcomplex.  
\item The chain complex $CFK^\infty(S^3,\#_mT_{2,3})\simeq CFK^\infty(S^3, T_{2,3})^{\otimes m}$ is filtered chain homotopy equivalent to $CFK^\infty(S^3, T_{2,2m+1})\oplus\mathcal A'$, where $\mathcal A'$ is an acyclic subcomplex.
\end{enumerate}
\end{theorem}

First, we calculate the correction terms for $X$ when $m=2k$, the case relevant to our discussion.  Recall that the affine identification $\spinc(X)\cong \Z_n$ gives rise to a natural indexing of $\frak s_i\in\spinc(X)$, where $|i|\leq (n-1)/2$.  This symmetry of this indexing is advantageous, and will be used here.  Let $\mathcal D(X)$ denote the collection of correction terms associated to $X$, i.e.,
$$\mathcal D(X) = \{d(X,\frak s_{\frac{-n+1}{2}}), d(X,\frak s_{\frac{-n+3}{2}}),\ldots, d(X,\frak s_{-1}), d(X,\frak s_0), d(X,\frak s_1),\ldots, d(X,\frak s_{\frac{n-3}{2}}), d(X,\frak s_{\frac{n-1}{2}})\}.$$

\begin{lemma}\label{lemma:Y_corr_terms}
Let $X=S^3_n(\#_{2k}K)$, where $K$ is $T_{2,3}$ or $Wh^+(T_{2,3},0)$.  Then, $ \mathcal D(L(n,1)) - \mathcal D(X)$ is given by 
$$ \{0,\ldots, 0,2,2,4,4,\ldots,2k-2,2k-2,2k,2k,2k,2k-2,2k-2,\ldots,4,4,2,2,0,\ldots, 0\}.$$
\end{lemma}

Of course, if $K$ is the unknot, then $\mathcal D(X) = \mathcal D(L(n,1))$.

\begin{proof}
By combining parts (1) and (2) of Theorem \ref{thm:HKL}, we get that
$$CFK^\infty(S^3, \#_mD)\simeq CFK^\infty(S^3, \#_mT_{2,3})\oplus\mathcal A'' \simeq CFK^\infty(S^3, T_{2,2m+1})\oplus\mathcal A'''.$$
The acyclic pieces can contribute to the homology of $HF^+(X)$, but these contributions are confined to $HF_{red}(X)$ and will not affect the correction term calculations.  It follows that $d(X,\frak s_i) = d(S^3_n(T_{2,2m+1}),\frak s_i)$ for all $|i|\leq (n-1)/2$.

The complex $C=CFK^\infty(S^3, T_{2,2m+1})$ can be easily obtained from the Alexander polynomial $\Delta_{T_{2,2m+1}}(t)$, since $T_{2,2m+1}$ is an alternating $L$-space knot \cite{oz-sz:alternating}, and is shown in Figure \ref{fig:AltTorusComplexes}.  Its basic building block (which we call a \emph{germ}) can be seen in Figure \ref{fig:AltTorusComplexes}.  One way to characterize which piece of the total complex is the germ $G$ is to say that $G$ is contained in the first $(i,j)$--quadrant, but $UG$ is not.  The total complex is obtained by taking $\Z$ copies of $G$, which are related by $U$--translation, i.e., $C = \sqcup_{z\in\Z}U^zG$.

There is a simple way to calculate $V_l=V_l(T_{2,2m+1})$ in this case \cite{ni-wu:cosmetic, ni-wu:rational}.  Consider the subcomplex $C_{\{\max(i,j-l)\geq 0\}}$.  Then,
$$V_l = \max\{z\ :\ U^zG\cap C_{\{\max(i,j-l)\geq 0\}}=\emptyset\}.$$
See, for example, Figure \ref{fig:VsAltTorus}.  With this in mind, it is now easy to see that
$$\{V_l(T_{2,2m+1})\}_{l\geq 0} = \begin{cases} 
\{k,k,k-1,k-1,\ldots, 2,2,1,1,0,\ldots\} & \text{ if $m=2k$}, \\
\{k,k-1,k-1,\ldots, 2,2,1,1,0,\ldots\} & \text{ if $m=2k+1$}, \\
\end{cases}
$$
where each value less than $k$ appears twice in each list, and the infinite tails each consists of zeros. Simply start with the shaded box in the third quadrant, as in Figure \ref{fig:VsAltTorus}(b), and notice how the homology changes as the box is moved vertically upward. If we recall that $H_{l} = V_{-l}$, then Theorem \ref{thm:ni-wu} completes the proof.
\end{proof}

\begin{figure}
\centering
\includegraphics[scale = .33]{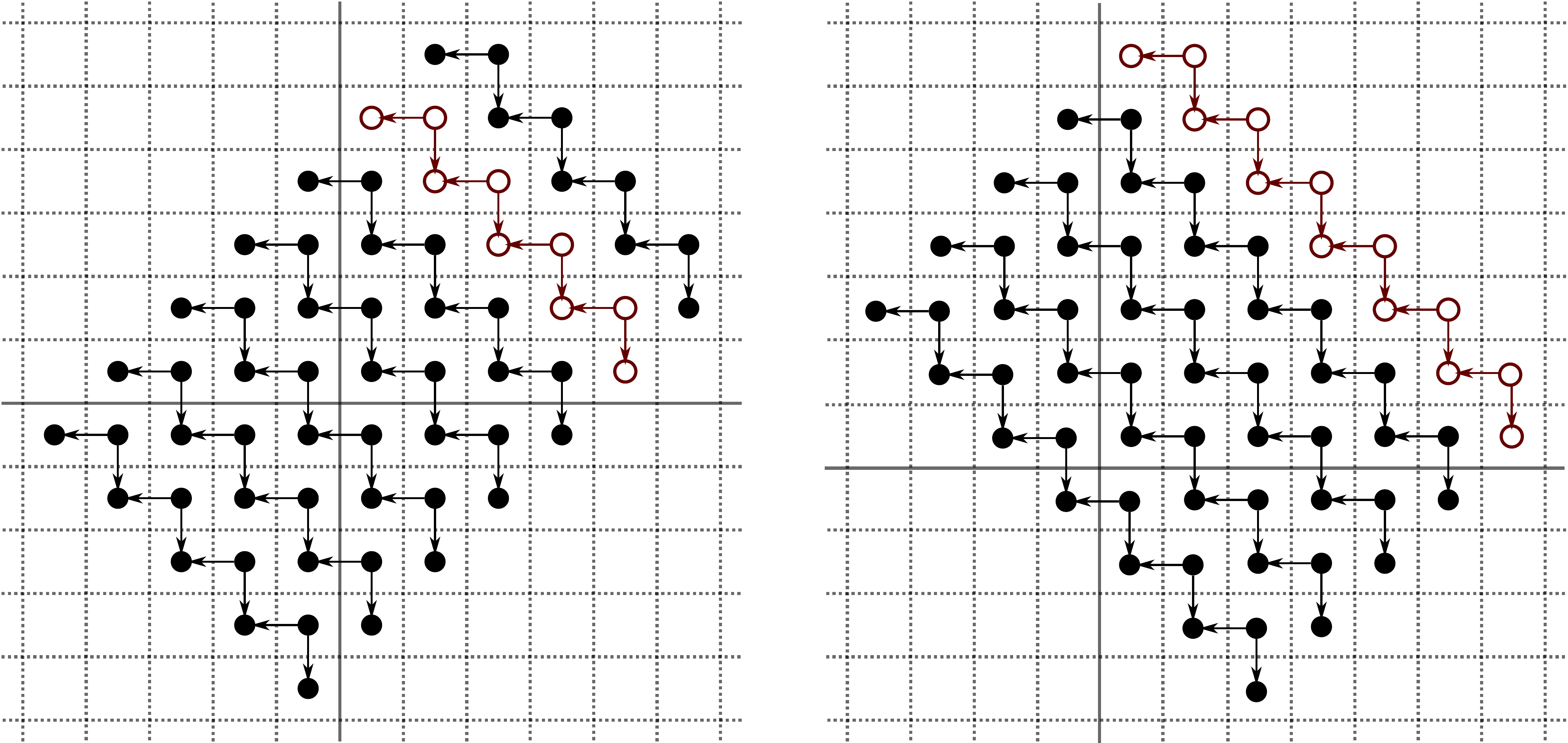}
\put(-307,-20){(a)}
\put(-101,-20){(b)}
\caption{Portions of the complex $CFK^\infty(S^3, T_{2,2m+1})$ is shown above for (a) $m=4$ and (b) $m=6$.  The germ of each complex is shown hollow in red.}
\label{fig:AltTorusComplexes}
\end{figure}

\begin{figure}
\centering
\includegraphics[scale = .33]{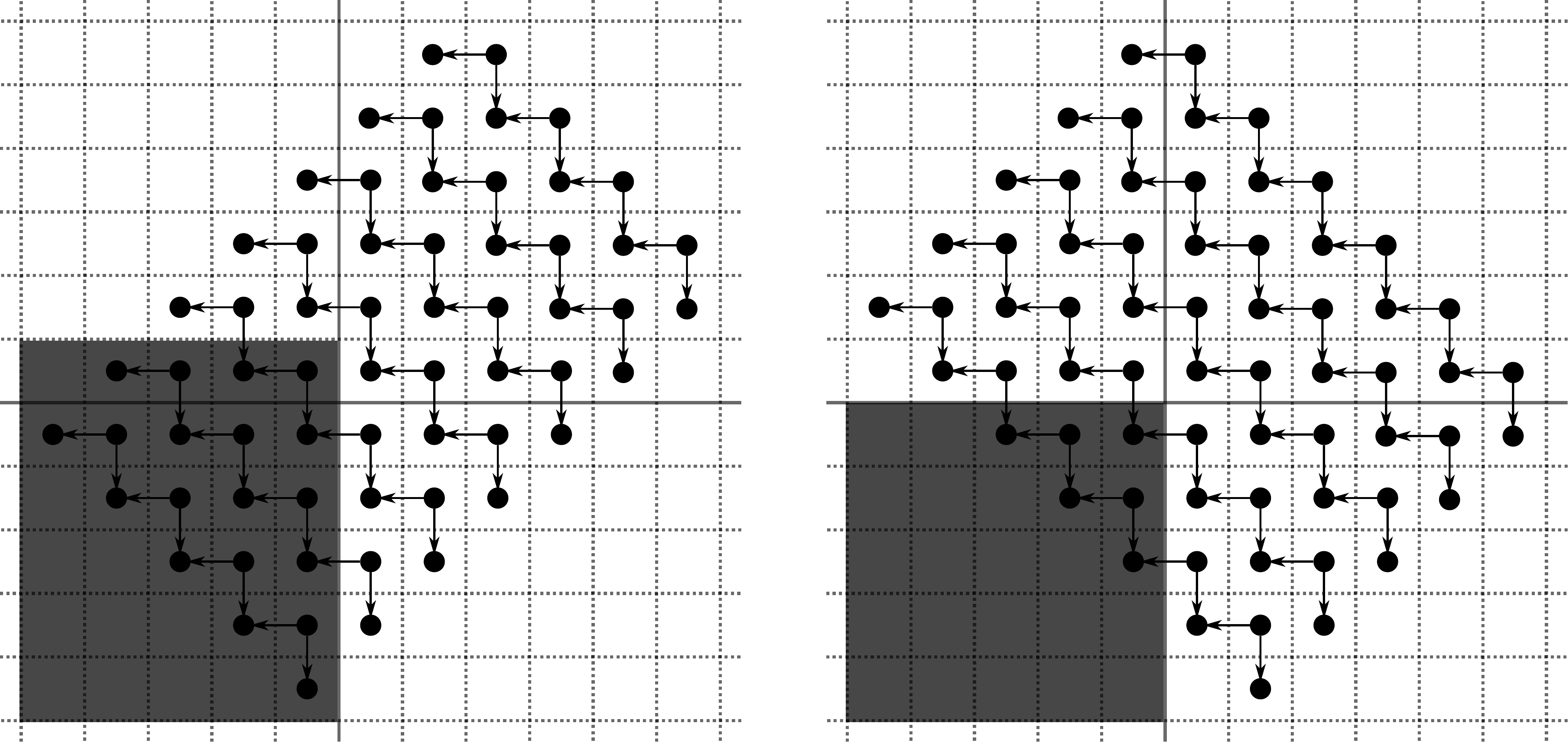}
\put(-307,-20){(a)}
\put(-101,-20){(b)}
\caption{(a) The calculation showing that $V_1(T_{2,9})=2$.  (b) The calculation showing that $V_0(T_{2,13})=3$.}
\label{fig:VsAltTorus}
\end{figure}

Let $L_k$ denote the list of even integers given in Lemma \ref{lemma:Y_corr_terms}, but with each value halved, and consider the bijection between $L_k$ and $\Z$ where the central $k$ corresponds to zero and the values to the left and right correspond to the negative and positive integers, respectively. Let $L_k^t$ denote a truncated version of $L_k$ where the value of any term in $L_k^t$ corresponding to an integer less than $-t$ is set to zero.  Let $L_k^t(x)$ represent the element of $L_k^t$ corresponding to $x\in\Z$.  For example,
$$
\begin{array}{ccc}
L_3 & = &  \{\ldots,0,0,1,1,2,2,3,3,3,2,2,1,1,0,0\ldots\}, \\
L_3^1 & = &  \{\ldots,0,0,0,0,0,0,3,3,3,2,2,1,1,0,0\ldots\}, \\
L_3^3 & = &  \{\ldots,0,0,0,0,2,2,3,3,3,2,2,1,1,0,0\ldots\}, \\
\end{array}
$$
and $L_3^1(-1)=3$.  We will make use of these truncated lists later.  




\begin{figure}
\centering
\includegraphics[scale = .75]{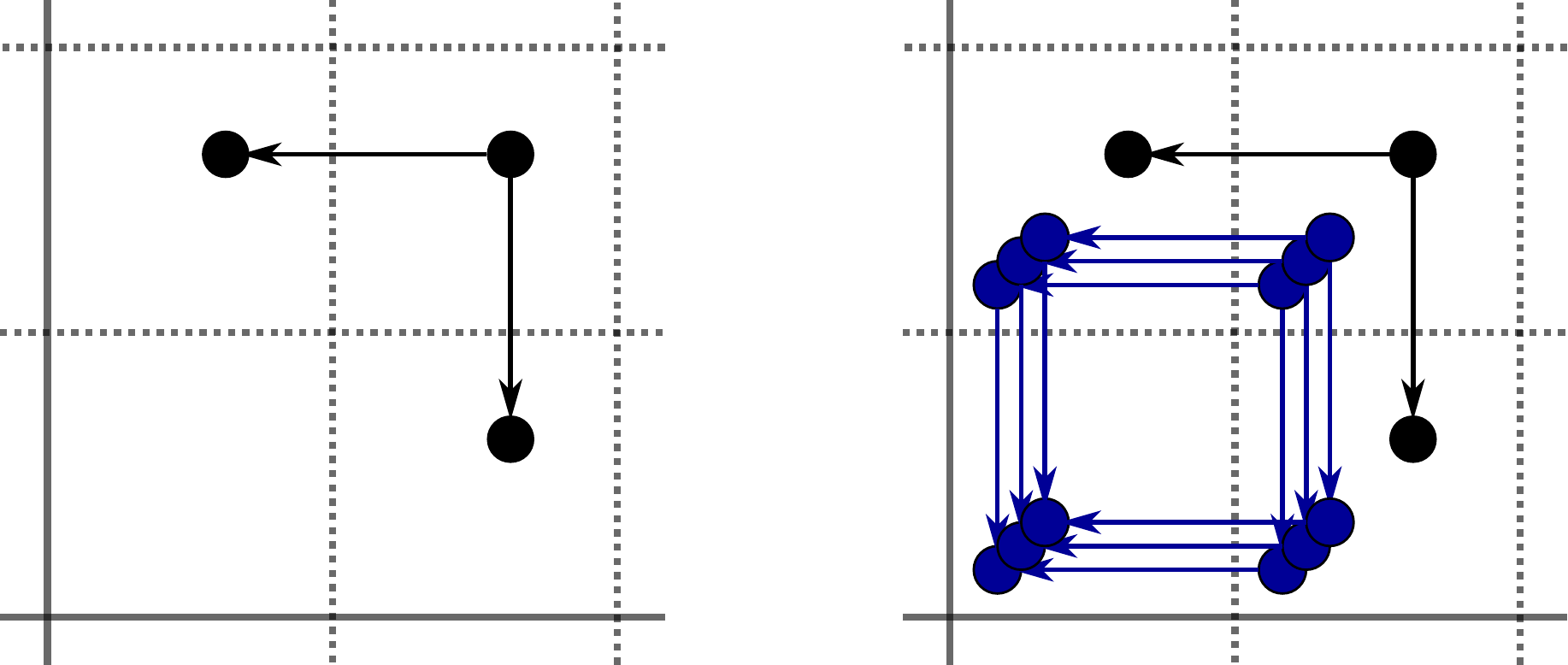}
\put(-321,-20){(a)}
\put(-91,-20){(b)}
\put(-354,134){0}
\put(-125,135){0}
\put(-58,19){\textcolor{Blue}{\LARGE$\star$}}
\caption{The complexes (a) $CFK^\infty(S^3,T_{2,3})$ and (b) $CFK^\infty(S^3,D)$ are shown with gradings; the three chains adjacent to the star have gradings $-2,-2$, and $-1$.}
\label{fig:TrefoilComplexes}
\end{figure}


Our next task is to give a calculation for the correction terms of $Y$.  To start, consider the case when $K$ is the unknot, so $Y = S^3_{n^2+n}(T_{n,n+1})$.

\begin{lemma}\label{lemma:Vs_Torus}
Let $n=2d+1$ for $d\in\N$.  Then, $\{V_l(T_{n,n+1})\}_{l\geq 0}$ is given by
$$
{\SMALL
\begin{array}{cccccccccc}
\{ & Tr(d), & Tr(d), & \ldots & Tr(d), & Tr(d)-1, & Tr(d) - 2, & \ldots, & Tr(d-1)+2, & Tr(d-1) + 1, \\
 & Tr(d-1), & Tr(d-1), & \ldots & Tr(d-1), & Tr(d-1)-1, & Tr(d-1) - 2, & \ldots, & Tr(d-2)+2, & Tr(d-2) + 1, \\
 & Tr(d-2), & Tr(d-2), & \ldots & Tr(d-2), & Tr(d-2)-1, & Tr(d-2) - 2, & \ldots, & Tr(d-3)+2, & Tr(d-3) + 1, \\
 & & & & & \vdots & & & & \\
 & 3, & 3, & 3, & \ldots, & & \ldots, & 3, & 3, & 2, \\
 & 1, & 1, & 1, & \ldots, & & \ldots, & 1, & 1, & 1, \\
 & 0, & 0, & 0, & \ldots, & \}, && & &  \\
\end{array}
}
$$
where $Tr(k)$ denotes the $k^\text{th}$ triangular number.
\end{lemma}

To clarify, the above list has been displayed so as to make the pattern of its elements more clear.  On the $i^\text{th}$ line, the value $Tr(d-i+1)$ appears $d+i+1$ times, followed by sequential decreases by 1, until the next triangular number is hit, which begins a new line.  The tail of the list is all zeros.  We will refer to the first appearance of each triangular number (i.e., the first element of each line) as a \emph{pivot}.  These pivots occur when $l$ is a multiple of $n$ and correspond to the cycles in the germ $G$ of $CFK^\infty(S^3, T_{n,n+1})$ (see Figure \ref{fig:TorusKnotComplex}).

\begin{figure}
\centering
\includegraphics[scale = .28]{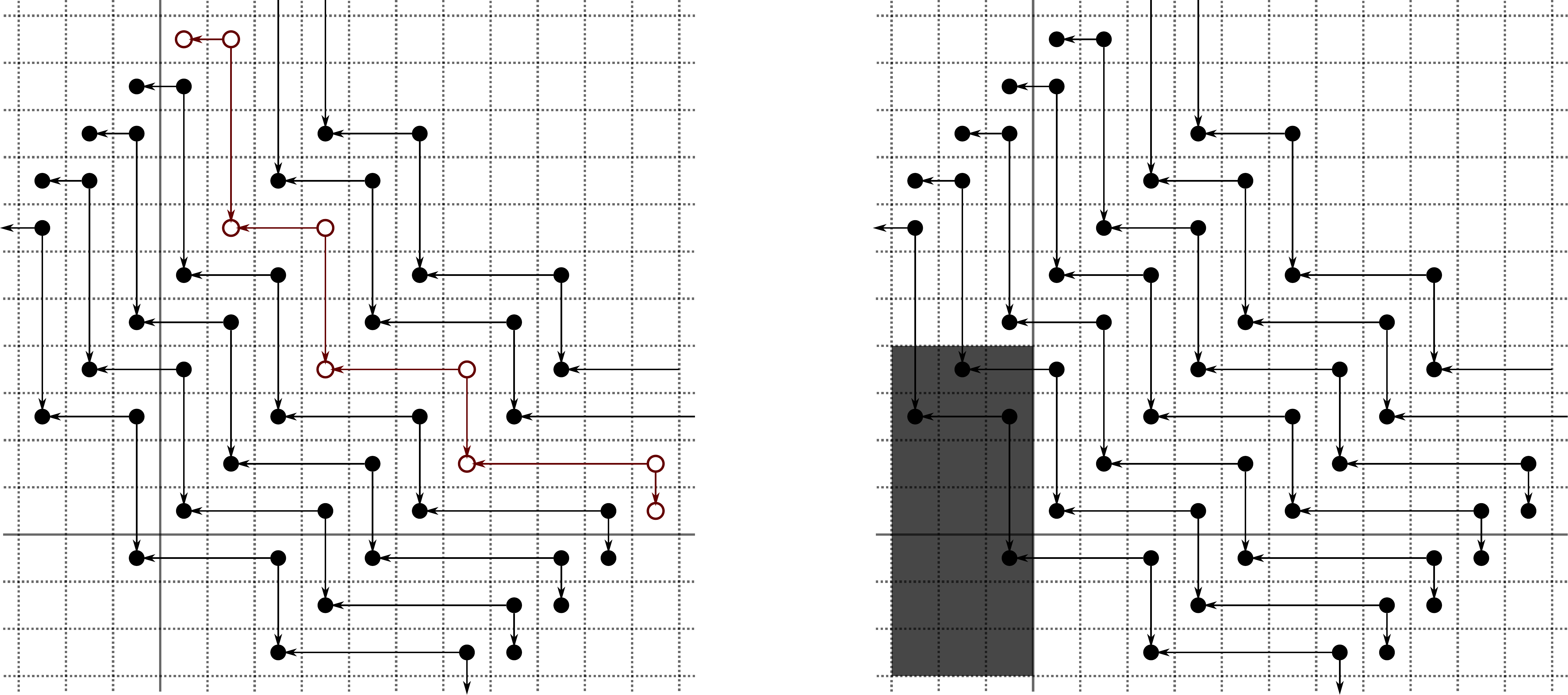}
\put(-356,-20){(a)}
\put(-106,-20){(b)}
\caption{(a) The complex $CFK^\infty(S^3,T_{n,n+1})$; here $n=5$. The calculation showing that $V_4(T_{5,6}) = 2$.}
\label{fig:TorusKnotComplex}
\end{figure}

\begin{proof}
As noted in \cite{HKL}, $C=CFK^\infty(S^3, T_{n,n+1})$ has germ $G$ as shown in red in Figure \ref{fig:TorusKnotComplex}(a).  The total complex is obtained by taking $\Z$ copies of this germ, which are related by $U$--translation, i.e., $C = \cup_{z\in\Z}U^zG$.  As in the proof of Lemma \ref{lemma:Y_corr_terms}, the $V_l = V_l(T_{n,n+1})$ are given by $$V_l = \max\{z\ :\ U^zG\cap C_{\{\max(i,j-l)\geq 0\}} = \emptyset\}.$$
See for example, Figure \ref{fig:TorusKnotComplex}(b).  Putting all this together, it is easy to see that $\{V_l\}_{l\geq 0}$ is as claimed.
\end{proof}

Lemma \ref{lemma:Vs_Torus} gives us a basis to understand the correction terms for surgeries on $J\#J\#T_{n,n+1}$.  To continue, we need to understand how the knot chain complex for $T_{n,n+1}$ changes under connected sum with $K$.

\begin{lemma}\label{lemma:Complex_Sum}
After a filtration-preserving change of basis, $CFK^\infty(S^3,J\#J\#T_{n,n+1}) = C_{sum}\oplus\mathcal A$, where a germ for $C_{sum}$ is made up of the the characteristic pieces shown in Figure \ref{fig:SumComplex}, and $\mathcal A$ is an acyclic subcomplex.
\end{lemma}

\begin{figure}
\centering
\includegraphics[scale = .35]{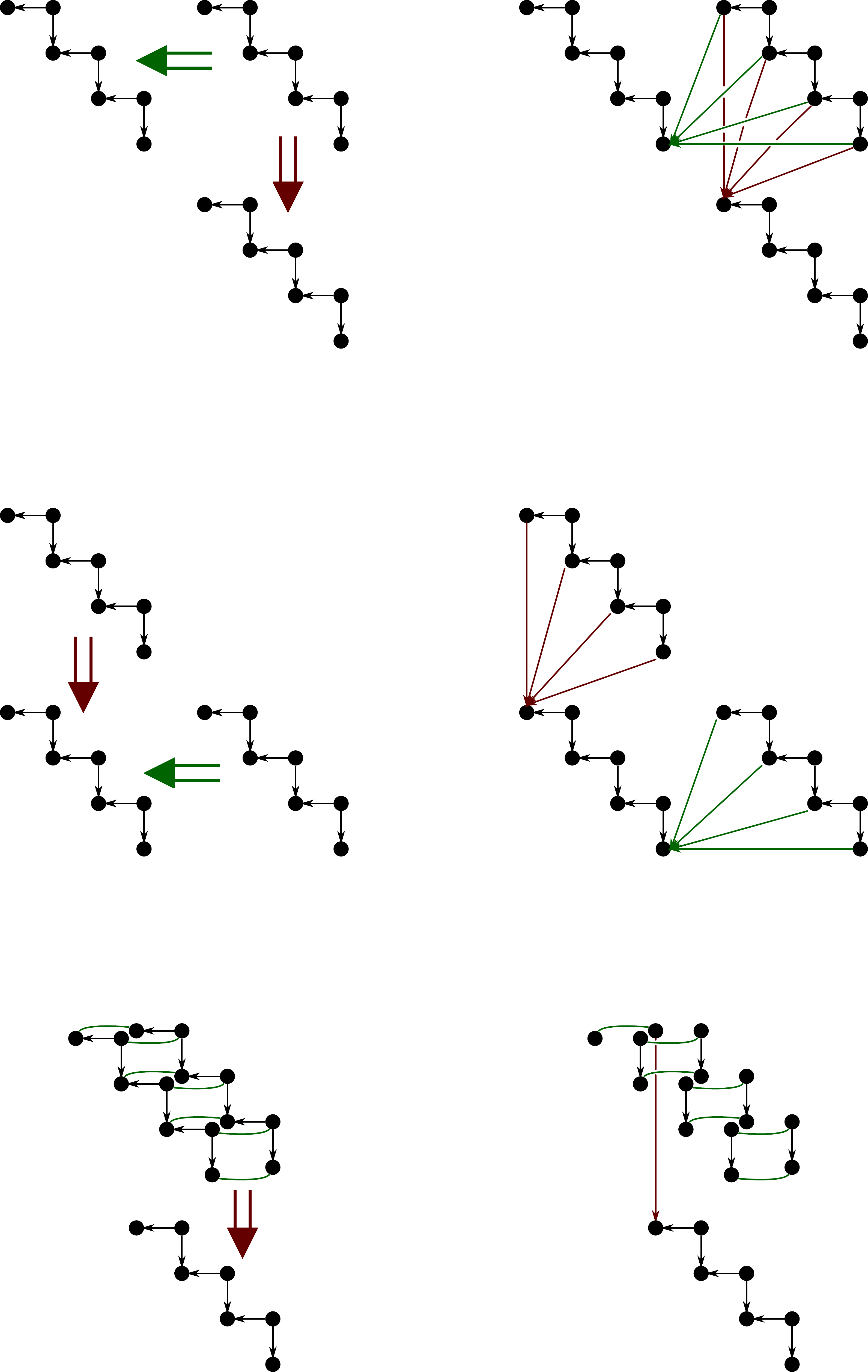}
\put(-150,440){\LARGE{$\simeq$}}
\put(-180,240){\LARGE{$\simeq$}}
\put(-150,60){\LARGE{$\simeq$}}
\put(-160,360){(a)}
\put(-160,160){(b)}
\put(-160,-20){(c)}
\caption{The three filtered chain homotopy equivalences used in the proof of Lemma \ref{lemma:Complex_Sum}.  Each is obtained by a filtration-preserving change of basis.  Note that in (a) and (b), if the off-diagonal group overlaps with either of the other two groups, then the result has a slightly different form, but is qualitatively the same.  Also, compare with Figures \ref{fig:SumComplex}(a) and \ref{fig:SumComplex}(b).  Here, each group of ``stairs'' represents a copy of $CFK^\infty(S^3, T_{2,2m+1})$; here $m=3$.  The groups of stairs are referred to as the $a$--, $b$--, and $c$--groups.}
\label{fig:PictureProofs}
\end{figure}

\begin{proof}
Recall that $CFK^\infty(S^3, J\#J)\simeq CFK^\infty(S^3, T_{2,2m+1})\oplus \mathcal A$, by Theorem \ref{thm:HKL}.  It follows that $CFK^\infty(S^3, J\#J\#T_{n,n+1})\simeq CFK^\infty(S^3, T_{n,n+1})\otimes CFK^\infty(S^3, T_{2,2m+1})\oplus\mathcal A$.  (Here, we use $\mathcal A$ to represent potentially different acyclic subcomplexes.)  To see that this tensor product has the desired form, we need three pictorial lemmas, shown in Figure \ref{fig:PictureProofs}.  All three parts show a chain homotopy equivalence achieved via a filtration preserving change of basis.  Consider Figure \ref{fig:PictureProofs}(a), and denote the chains by $a_1, \ldots, a_{2m+1}$, $b_1,\ldots, b_{2m+1}$, and $c_1, \ldots, c_{2m+1}$ (i.e., $b_i\mapsto a_i+c_i$). The double arrows mean that there should be an arrow between each pair of vertically or horizontally aligned chains.  The pertinent change of basis is:
$$b_k\mapsto b_k+\sum_{\substack{i<k, \\ i \text{ even}, \\ n_j(b_k)>n_j(c_i)}}c_i + \sum_{\substack{j>k, \\ j \text{ even} \\ n_i(b_k)>n_i(a_i)}}a_j,$$
for odd $k$.  (Note that the indexing variables $i,j$, and $k$ used here are not related to the uses of $i,j,$ and $k$ used elsewhere; in particular, $i$ and $n_i$ are not related here.)  The third condition on each summation guarantees that this change of basis is filtration preserving.  Note that any vertical arrows hitting the $a$--group or horizontal arrows hitting the $c$--group are unaffected, since the chains in these groups are unchanged.  The result is that the only arrows from the $b$--group to either the $a$--group or the $c$--group will go from $b_k$ with odd $k$ to $a_i$ and $c_j$ with odd $i$ and odd $j$.  Parts (a) and (b) of Figure \ref{fig:SumComplex} show the possible results of this local change of basis on the germ of $CFK^\infty(S^3, J\#J\#_{n,n+1})$.  In (a), the $a$ and $b$ pieces overlap; in (b) they do not.

Next, Figure \ref{fig:PictureProofs}(b) corresponds to the filtration preserving change of basis given by
$$c_k\mapsto c_k+\sum_{\substack{j>k, \\ i \text{ even}, \\ n_j(b_j)>n_j(c_k)}}b_j,$$
for odd $k$, and 
$$a_k\mapsto a_k + \sum_{\substack{i<k, \\ i \text{ even} \\ n_i(b_i)>n_i(a_k)}}b_i,$$
for odd $k$.

Finally, Figure \ref{fig:PictureProofs}(c) corresponds to the filtration preserving change of basis given by
$$a_k\mapsto a_k+c_1,$$
for odd $k\geq 3$, 
$$a_k\mapsto a_k + b_{k-1} + \sum_{\substack{i<k, \\ i \text{ even}}}c_i,$$
for even $k$, and
$$b_k\mapsto b_k + \sum_{\substack{i<k, \\ i \text{ even}}}c_i,$$
for odd $k>2$.  See also Figure \ref{fig:SumComplex}(c).

By applying these three types of change of basis, we see that $CFK^\infty(S^3, J\#J\#T_{n,n+1}) = C_{sum}\oplus\mathcal A$, where the characteristic pieces of $C_{sum}$ are shown in Figure \ref{fig:SumComplex}.  In other words, connected summing $T_{n,n+1}$ with $J\#J$ introduces a stepping pattern at every joint (i.e., the cycles) of the germ of $CFK^\infty(S^3, T_{n,n+1})$, except the first and last, where the germ simply extends by $m$.
\end{proof}

\begin{figure}
\centering
\includegraphics[scale = .25]{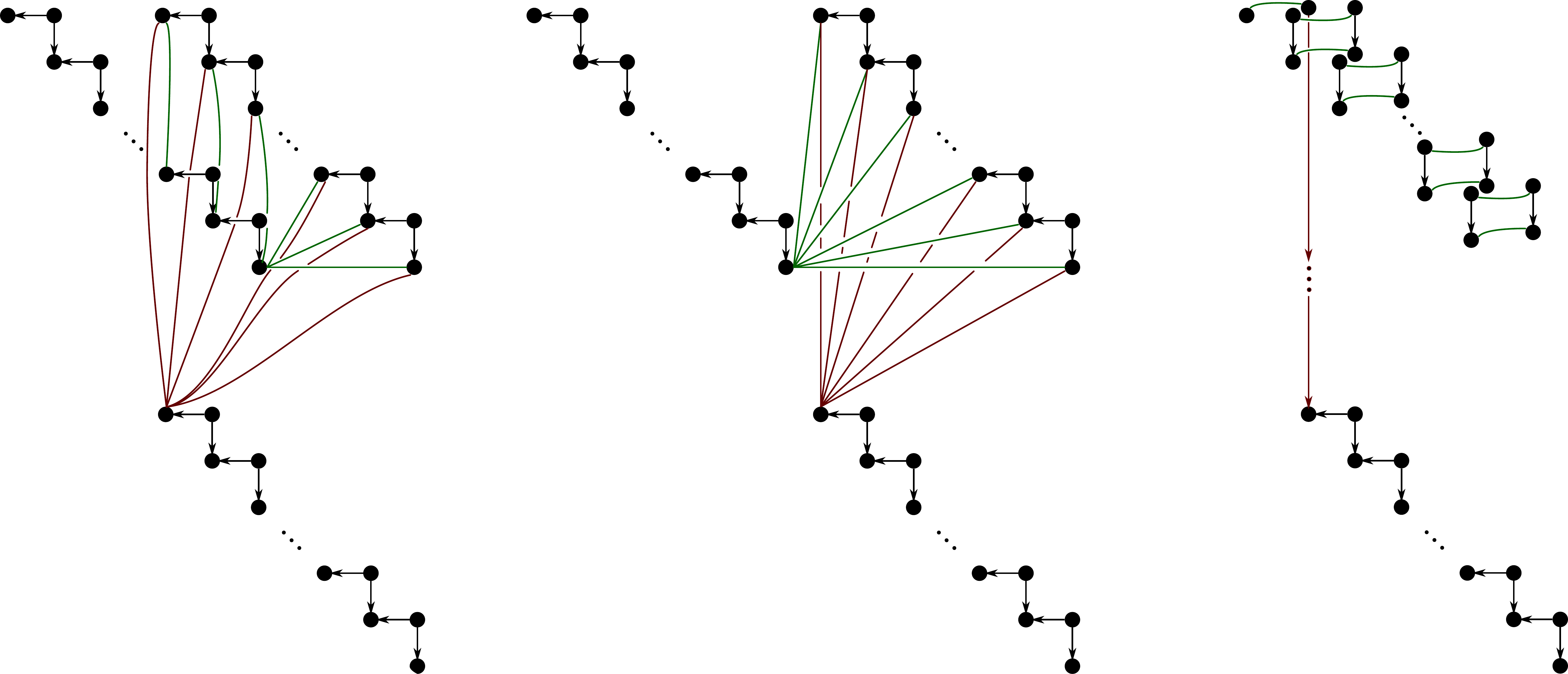}
\put(-360,-20){(a)}
\put(-210,-20){(b)}
\put(-50,-20){(c)}
\caption{Three characteristic pieces of the complex $C_{T_{n,n+1}}$.}
\label{fig:SumComplex}
\end{figure}

\begin{lemma}\label{lemma:Vs_Sum}
Let $K$ be $T_{2,3}$ or $D$ and let $J=\#_{k}K$.   Then, for $|l|\leq \frac{n^2+n}{2}$,
$$V_l(J\#J\#T_{n,n+1}) - V_l(T_{n,n+1}) =
\begin{cases}
L_k^{t(l)}(\overline l) & \text{ if \ $0\leq l\leq \frac{n(n-1)}{2}$}, \\
1 & \text{ if \ $\frac{n(n-1)}{2}\leq l \leq \frac{n(n-1)}{2}+k$}, \\
0 & \text{ if \ $\frac{n(n-1)}{2}+k<l$},
\end{cases}
$$
where $\overline l\in[\frac{-n+1}{2},\frac{n-1}{2}]$ is the mod $n$ reduction of $l$, and $t(l) = \frac{n-3}{2} - a$ if $|l|\in[an-\frac{n-1}{2},an+\frac{n+1}{2})$.
\end{lemma}

\begin{proof}

Lemma \ref{lemma:Complex_Sum} tells us that the germ of $CFK^\infty(S^3, J\#J\#T_{n,n+1})$ is given locally as in Figure \ref{fig:SumComplex}.  Forming the connected sum changes the complex for $T_{n,n+1}$ by introducing a stepping pattern at each joint.  It is straightforward to see the effect of this on the $V_l(T_{n,n+1})$. For each joint, one simply superimposes a copy of $L_k^t$ over $\{V_l(T_{n,n+1})\}_{\geq 0}$, with $L_k(0)$ centered over the $l$ corresponding to the joint.  If $C_{sum}$ looks locally like Figure \ref{fig:SumComplex}(a) at the joint (i.e., there is some overlap), then we use $L^t_k$ where $t-1$ is the amount of overlap (in Figure \ref{fig:SumComplex}(a), the overlap shown is 3).  If $C_{sum}$ looks locally like Figure \ref{fig:SumComplex}(b) at the joint (i.e., there is no overlap), then $L^t_k = L_k$ (i.e., there is no truncation).

To clarify, $\{V_l(T_{n,n+1})\}_{\geq 0}$ is shown below, with the pivots highlighted in red.  These pivots correspond to the joints of $CFK^\infty(S^3, T_{n,n+1})$.  In $C_{sum}$, each such joint has been tensored with the germ for $CFK^\infty(S^3, T_{2,2m+1})$ and looks locally as in Figure \ref{fig:SumComplex}.  By considering these local pictures, we can see that $V_l=V_l(J\#J\#T_{n,n_1})$ will have the value claimed, because the introduction of the stepping pattern corresponds precisely to adding  $L_k^{t(l)}(\overline l)$ to $V_l$.

$$
{\SMALL
\begin{array}{cccccccccc}
\{ & \textcolor{Red}{Tr(d)}, & Tr(d), & \ldots & Tr(d), & Tr(d)-1, & Tr(d) - 2, & \ldots, & Tr(d-1)+2, & Tr(d-1) + 1, \\
 & \textcolor{Red}{Tr(d-1)}, & Tr(d-1), & \ldots & Tr(d-1), & Tr(d-1)-1, & Tr(d-1) - 2, & \ldots, & Tr(d-2)+2, & Tr(d-2) + 1, \\
 & \textcolor{Red}{Tr(d-2)}, & Tr(d-2), & \ldots & Tr(d-2), & Tr(d-2)-1, & Tr(d-2) - 2, & \ldots, & Tr(d-3)+2, & Tr(d-3) + 1, \\
 & & & & & \vdots & & & & \\
 & \textcolor{Red}{3}, & 3, & 3, & \ldots, & & \ldots, & 3, & 3, & 2, \\
 & \textcolor{Red}{1}, & 1, & 1, & \ldots, & & \ldots, & 1, & 1, & 1, \\
 & \textcolor{Red}{0}, & 0, & 0, & \ldots, & \}, && & &  \\
\end{array}
}
$$
\end{proof}

Now that we have calculated $\{V_l(J\#J\#T_{n,n+1})\}_{l\geq 0}$, it is straightforward to calculate the correction terms for $Y$.

\begin{corollary}\label{coro:Y'_corr_terms}
Let $K$ be $T_{2,3}$ or $D$, let $J=\#_kK$, and let $Y = S^3_{n^2+n}(J\#J\#T_{n,n+1})$.  Then,
$$\mathcal D(L(n,1)\#L(n+1,-1))-\mathcal D(Y) $$
is given by
$$
{\small
\left[
\setlength{\arraycolsep}{4pt}
\begin{array}{cccccccccccccccccccccc}
0 & \cdots & 0& 0 & 0 & 0 & 0 & 0 & 0 & 0 & 0 & 0 & 0 & 0 & 0 & \cdots & 0 & 0& 0& \cdots & 0 \\
\vdots & & \vdots & \vdots &\vdots & \vdots & \vdots & \vdots & \vdots & \vdots & \vdots & \vdots & \vdots & \vdots & \vdots && \vdots& \vdots & \vdots& & \vdots \\
0 & \cdots & 0 & 0 & 0 & 0 & 0 & 0 & 0 & 0 & 0 & 0 & 0 & 0 & 0 & \cdots & 0 & 0& 0& \cdots & 0 \\
2 & \cdots & 2 & \textcolor{Red}{2} & 0 & \textcolor{Red}{0} & 0 & 0 & 0 &0 & 0 & 2 & 2 & 2 & 2 &  \cdots & 2 & 2 & 2& \cdots & 2 \\
2 & \cdots & 2 & 2 & \textcolor{Red}{2} & 0 & \textcolor{Red}{0} & 0 & 0 & 0 & 0 & 2  & 2 & 2 & 2 & \cdots & 2 & 2 & 2& \cdots & 2 \\
4 & \cdots & 4 & 4 & 4 & \textcolor{Red}{2} & 0 & \textcolor{Red}{0} & 0 & 0 & 0 & 4 & 4 & 4 & 4 & \cdots & 4 & 4 & 4& \cdots & 4 \\
4 & \cdots & 4 & 4 & 4 & 4 & \textcolor{Red}{2} & 0 & \textcolor{Red}{0} & 0 & 0 & 4 & 4 & 4 & 4 & \cdots & 4 & 4 & 4& \cdots & 4\\
\vdots &  & \vdots & \vdots & \vdots &  &  & \textcolor{Red}{\ddots} & \ddots & \textcolor{Red}{\ddots} & & \vdots & \vdots & \vdots &\vdots & & \vdots & \vdots& \vdots& & \vdots \\
2k & \cdots & 2k & 2k & 2k & \cdots & 2k & 2k & \textcolor{Red}{2} & 0 & \textcolor{Red}{0} & 2k & 2k & 2k & 2k & \cdots & 2k & 2k & 2k & \cdots & 2k \\
2k & \cdots & 2k & 2k & 2k & \cdots & 2k & 2k & 2k & \textcolor{Red}{2} & 0 & \textcolor{Red}{2} & 2k & 2k &2k &  \cdots & 2k & 2k & 2k& \cdots & 2k \\
2k & \cdots & 2k & 2k & 2k & \cdots & 2k & 2k & 2k & 2k & \textcolor{Red}{0} & 0 & \textcolor{Red}{2} & 2k & 2k & \cdots & 2k & 2k & 2k& \cdots & 2k \\
\vdots & & \vdots & \vdots & \vdots & & \vdots & \vdots  & \vdots  & \vdots & & \textcolor{Red}{\ddots} & \ddots & \textcolor{Red}{\ddots} & & & \vdots   & \vdots& \vdots&  & \vdots \\
4 & \cdots & 4 & 4 & 4 & \cdots & 4 & 4 & 4 & 4 & 0 & 0 & \textcolor{Red}{0} & 0 & \textcolor{Red}{2} & 4 & 4 & 4 & 4& \cdots & 4 \\
4 & \cdots & 4 & 4 & 4 & \cdots& 4 & 4 & 4 &4 & 0 & 0 & 0 & \textcolor{Red}{0} & 0 & \textcolor{Red}{2} & 4 & 4 & 4& \cdots & 4 \\
2 & \cdots & 2 & 2 & 2 &  \cdots& 2 & 2 & 2 & 2 & 0 & 0 & 0 & 0 & \textcolor{Red}{0} & 0 & \textcolor{Red}{2} & 2 & 2& \cdots & 2 \\
2 & \cdots & 2 & 2 & 2 & \cdots & 2 & 2 & 2 & 2& 0 & 0& 0 & 0 & 0 & \textcolor{Red}{0} & 0 & \textcolor{Red}{2} & 2& \cdots & 2 \\
0 & \cdots & 0 & 0 & 0 & \cdots & 0 & 0 & 0 &0 & 0 & 0 & 0 & 0 &0 & 0 & 0 & 0 & 0& \cdots & 0 \\
\vdots & & \vdots & \vdots & \vdots & & \vdots & \vdots & \vdots  & \vdots & \vdots & \vdots & \vdots & \vdots & \vdots& \vdots & \vdots & \vdots & \vdots& & \vdots \\
0 & \cdots & 0 & 0 & 0 & \cdots & 0 & 0 & 0 & 0 & 0 & 0 & 0 & 0 & 0 & 0 & 0 & 0 & 0& \cdots & 0 \\
\end{array}
\right],
}
$$
where the rows are indexed by $i\in[\frac{-n+1}{2},\frac{n-1}{2}]$ and the columns are indexed by $j\in[\frac{-n}{2},\frac{n-2}{2}]$.
\end{corollary}
This matrix, $\mathcal M$ is not as complicated as it looks.  It is a $(n\times(n+1))$--matrix, with zeros along the right off-diagonal.  If the first column is removed, it is rotationally symmetric.  Think of $\mathcal M$ as the union of its diagonals. Every diagonal (other than the middle three) is simply a copy of (twice) $L_k^t$, where $t$ is the displacement (left or right) from the center three diagonals.  For example, if we consider the diagonal directly to the left of the middle three diagonals, we see a copy of (twice) $L_k^1$ that emanates towards the northwest.

\begin{proof}
The matrix presentation for $ \mathcal D(L(n,1)\#L(n+1,-1)) - \mathcal D(Y) $ follows from Lemma \ref{lemma:Vs_Sum} in light of the following remark.  Knowing $V_l(J\#J\#T_{n,n+1})$ allows us to calculate $d(Y,\frak s_l)$, but in the matrix above, we have given $d(Y, [\frak s_i, \frak s_j])$, which uses the enumeration of $\spinc$ structures introduced in Subsection \ref{subsec:enumerate}.  This correspondence is given by $l = \frac{n(n+1)}{2} +(n+1)i - nj$, where $l,i$, and $j$ are all taken to be centered about zero.  This identification maps the subgroup generated by $n+1\in\Z_{n^2+n}$ to the subgroup generated by $(1,1)\in\Z_n\oplus\Z_{n+1}$. 
\end{proof}

There is a slight issue related to our choice of identification.  In fact, there are four different identifications we could have chosen, each of which is related to the others by negating $i$, $j$, or both.  The following lemma proves that these are the only four identifications that we should concern ourselves with, since any other identification will not preserve the equivalence class of the correction terms modulo 2.

Even if we are content with only these four identifications, we need to note that a different choice of identification changes our labeling of the correction terms.  We have introduced an indeterminacy in which we cannot distinguish $i$ from $-i$ or $j$ from $-j$ in our labelings.  Fortunately, this does not affect the proofs of Theorem \ref{thm:main} and \ref{thm:main2}.

\begin{lemma}
Let $l = \frac{n(n+1)}{2} + (n+1)i - nj$.  Then,
$$d(S^3_{n^2+n}(J\#J\#T_{n,n+1}),l) \equiv d(L(n,1),i) - d(L(n+1,1),j) \pmod 2.$$
\end{lemma}

\begin{proof}
By the integer surgery formula, we see that $d(S^3_{n^2+n}(T_{n,n+1}),l) \equiv d(L(n^2+n,1),l) \pmod 2$.  Let $k=j-i$, then it is straightforward to show that
$$\frac{(2l-n(n+1))^2-n(n+1)}{4n(n+1)} - \frac{(2i-n)^2-n}{4n}-\frac{(2j-(n+1))^2-(n+1)}{4(n+1)} = k^2-k.$$
From this, it follows that $d(L(n^2+n,1),l)\equiv d(L(n,1),i)-d(L(n+1,1),j) \pmod 2$ and that 
$$d(S^3_{n^2+n}(T_{n,n+1}),l) \equiv d(L(n,1),i)-d(L(n+1,1),j) \pmod 2.$$

Furthermore, it is easy to see that $d(L(p,1),a)\equiv d(L(p,1),b) \pmod 2$ if and only if $b=p-a$, assuming $0\leq a<b<p$.    Thus, for any $Y$ obtained by surgery on a knot in $S^3$, we know that $d(Y,a) = d(Y,b)$ if and only if $b=a$ or $p-a$.  In other words, the equality of two correction terms is determined by their mod 2 equivalence class for such 3--manifolds.  This implies that
$$d(S^3_{n^2+n}(T_{n,n+1}),l) = d(L(n,1),i) - d(L(n+1,1),j).$$
To complete the proof, we simply note that $V_l(J\#J\#T_{n,n+1}) \equiv V_l(T_{n,n+1}) \pmod 2$, by Lemma \ref{lemma:Vs_Sum}.
\end{proof}

\begin{figure}
\centering
\includegraphics[scale = .5]{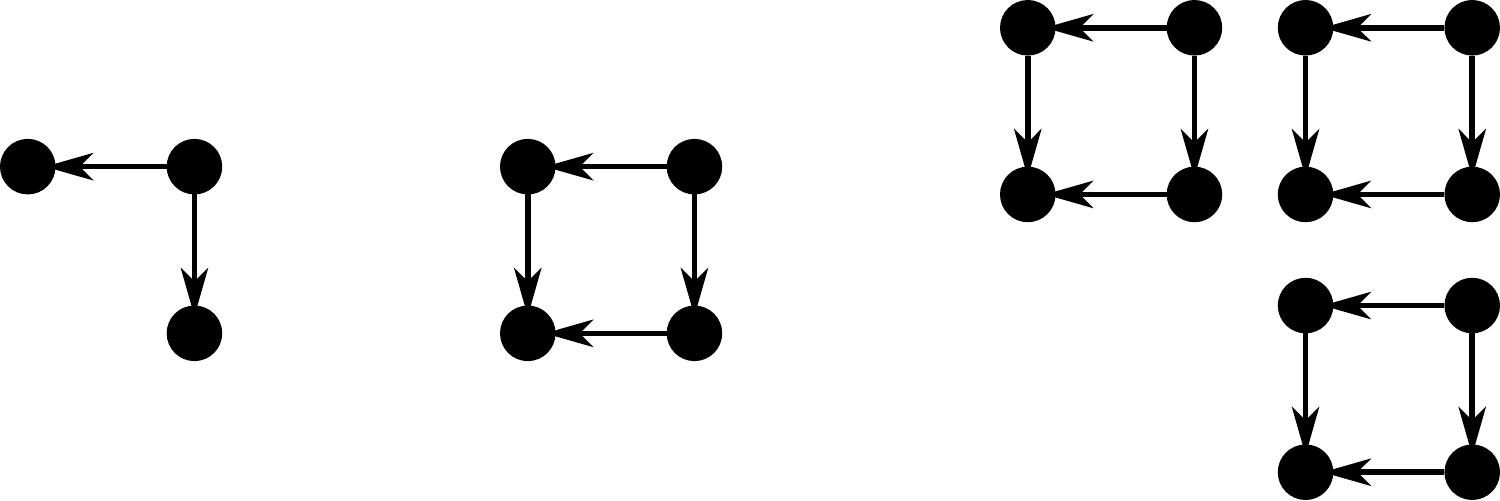}
\put(-98,33){\LARGE{$\simeq$}}
\put(-170,32){\LARGE{$\otimes$}}
\caption{The filtered chain homotopy equivalence shown above is a straightforward exercise.}
\label{fig:TensorLemma}
\end{figure}



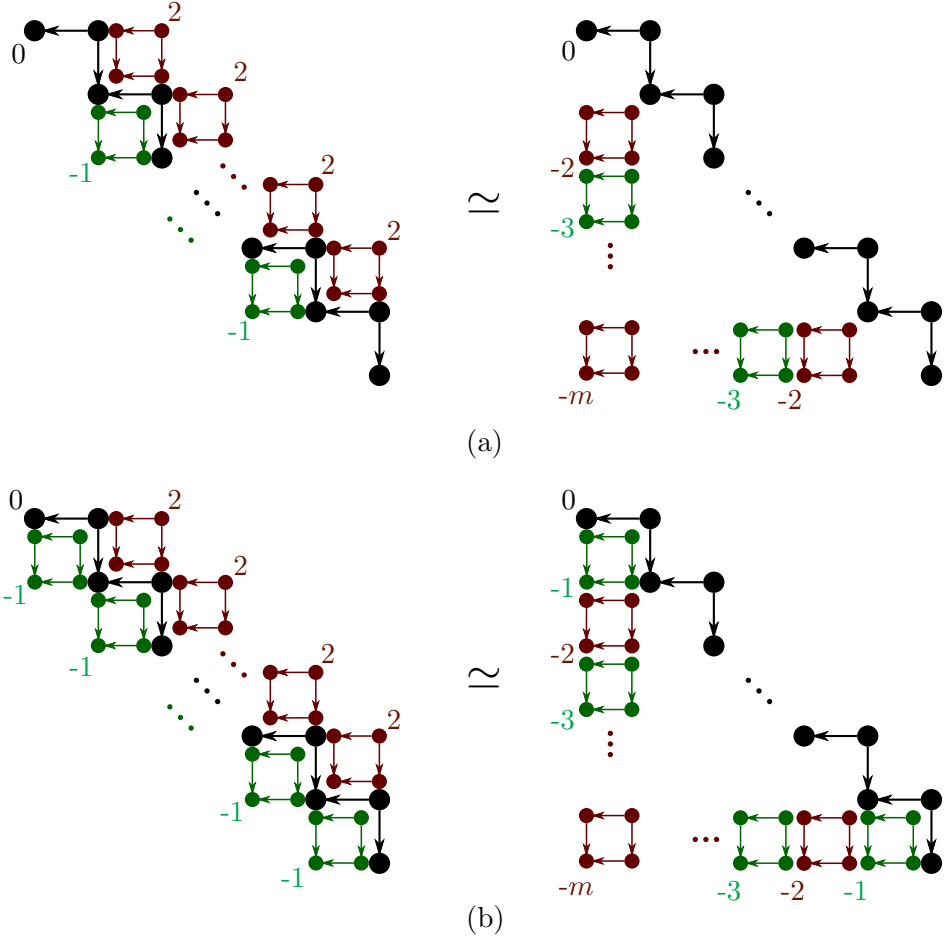
\begin{figure}
\centering
\includegraphics[scale = .5]{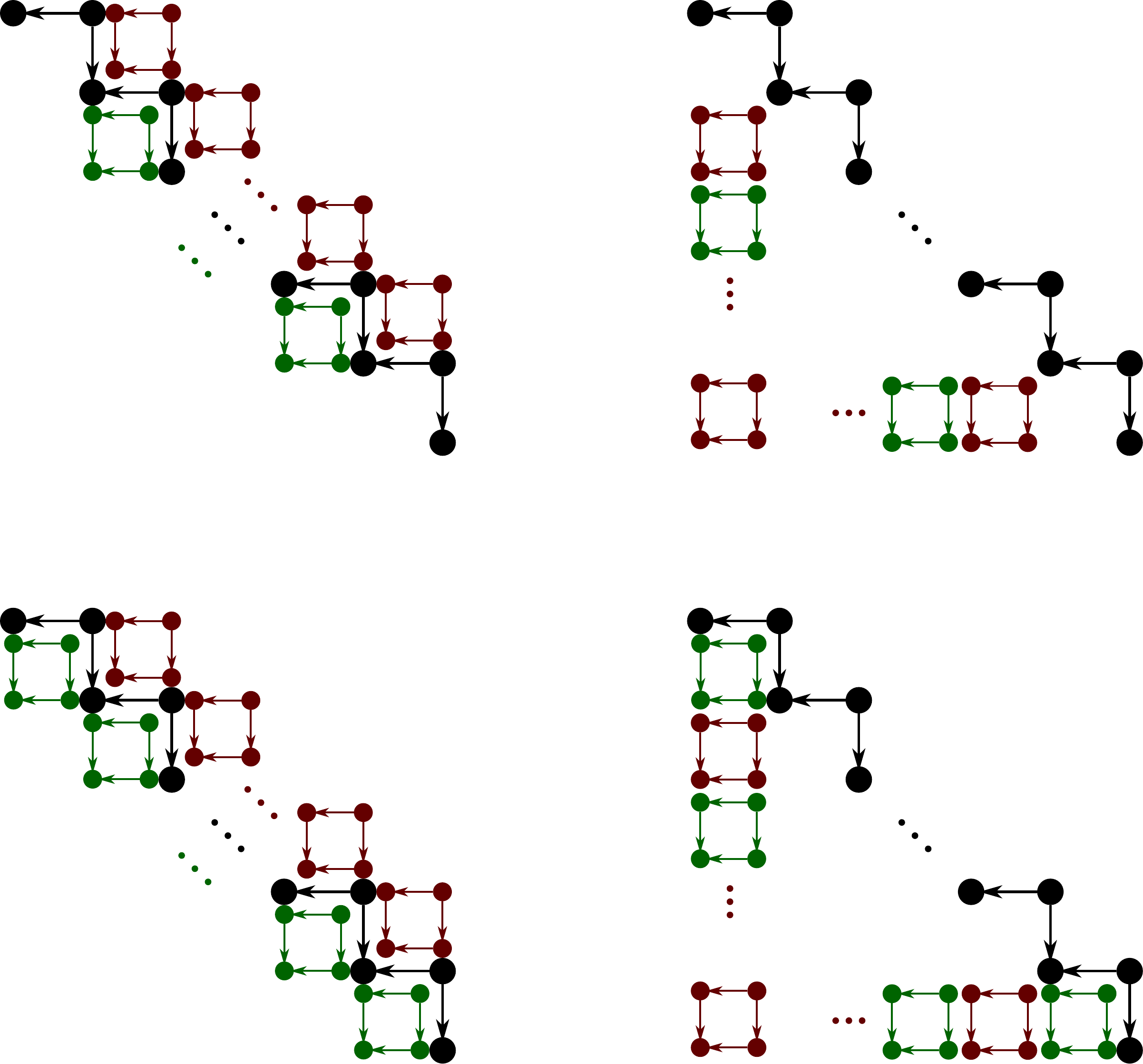}
\put(-180,250){\LARGE{$\simeq$}}
\put(-180,70){\LARGE{$\simeq$}}
\put(-180,160){(a)}
\put(-180,-20){(b)}
\put(-352,307){0}
\put(-353,138){0}
\put(-144,308){0}
\put(-144,138){0}
\put(-330,262){\textcolor{Green}{-1}}
\put(-270,202){\textcolor{Green}{-1}}
\put(-355,102){\textcolor{Green}{-1}}
\put(-273,20){\textcolor{Green}{-1}}
\put(-330,75){\textcolor{Green}{-1}}
\put(-250,-5){\textcolor{Green}{-1}}
\put(-210,55){\textcolor{Sepia}{2}}
\put(-235,80){\textcolor{Sepia}{2}}
\put(-268,113){\textcolor{Sepia}{2}}
\put(-293,138){\textcolor{Sepia}{2}}
\put(-210,240){\textcolor{Sepia}{2}}
\put(-235,265){\textcolor{Sepia}{2}}
\put(-268,300){\textcolor{Sepia}{2}}
\put(-293,323){\textcolor{Sepia}{2}}
\put(-148,264){\textcolor{Sepia}{-2}}
\put(-148,241){\textcolor{Green}{-3}}
\put(-145,178){\textcolor{Sepia}{-$m$}}
\put(-62,176){\textcolor{Sepia}{-2}}
\put(-85,176){\textcolor{Green}{-3}}
\put(-148,105){\textcolor{Green}{-1}}
\put(-148,80){\textcolor{Sepia}{-2}}
\put(-148,55){\textcolor{Green}{-3}}
\put(-145,-8){\textcolor{Sepia}{-$m$}}
\put(-61,-10){\textcolor{Sepia}{-2}}
\put(-85,-10){\textcolor{Green}{-3}}
\put(-37,-10){\textcolor{Green}{-1}}
\caption{Germs for the total complexes of (a) $CFK^\infty(S^3, \#_mT_{2,3})$ and (b) $CFK^\infty(S^3, \#_mD)$.  Note that each square above is meant to represent a multitude of overlaid squares.  In (a), gradings of overlaid squares match up and are as shown.  In (b), gradings in overlaid squares may be lower than shown.  (See text.) }
\label{fig:TorsionComplexes}
\end{figure}

Descriptions of the germs for the total complexes $CFK^\infty(S^3, \#_mT_{2,3})$ and $CFK^\infty(S^3, \#_mD)$ are given in Figure \ref{fig:TorsionComplexes}.  These presentation follow from Theorem \ref{thm:HKL}, the pictorial lemma shown in Figure \ref{fig:TensorLemma}, and induction.  The proof of this pictorial lemma is straightforward.  Note that in Figure \ref{fig:TorsionComplexes} each acyclic square shown is meant to represent a multitude of overlying acyclic squares, but the gradings are controlled.

Here is how the gradings behave in Figure \ref{fig:TorsionComplexes}.  In (a), the gradings are as shown, and overlaid squares have the same gradings as the representative shown.  In (b), for each collection of overlaid squares, the maximally graded representative is shown, and there are $m+1$ different possible gradings.  For example, consider the bottom-left square in the right part of (b).  This square represents many squares, each of which has as its bottom-left corner a chain with grading in $\{-m,-(m+1),-(m+2),\ldots, -2m\}$.

We are now prepared to prove one final property about  $HF^+(Y')$, which we will need in order to complete the proof in Section \ref{section:calculations}.

\begin{lemma}\label{lemma:reduced_bound}
Let $Y = S^3_{n^2+n}(J\#J\#T_{n,n+1})$, and let $\xi\in HF_{red}(Y,[\frak s_i,\frak s_j])$. Then
$$gr(\xi) \leq gr\left(\mathcal T_{i,j}^+(Y)\right).$$
\end{lemma}



\begin{proof}
Let $C = CFK^\infty(S^3, J\#J\#T_{n,n+1})$, let $C^1 = CFK^\infty(S^3,J\#J)$, and let $C^2 = CFK^\infty(S^3, T_{n,n+1})$, so $C = C^1\otimes C^2$.  Let $G, G^1$, and $G^2$ be the germs for $C, C^1$, and $C^2$, respectively.  Let $\xi\in HF_{red}(Y,[\frak s_i,\frak s_j])$, and let $c\in C$ be any chain such that $[c]=\xi$.

Let $c'\in C$ be any chain such that $[c']$ is the element of lowest grading in $\mathcal T_{i,j}(Y)$, so $gr(c') = gr(\mathcal T_{i,j}(Y))$.  Let $G' = U^{z'}G$ be the germ containing $c'$, where $z'\in\Z$.  Any chain in $\cup_{e> 0}U^eG'$ that is not homologous to a $U$--translate of $c'$ is not a cycle.  To see this, simply observe that $H_*(\cup_{e> 0}U^eG') \cong \mathcal T^+$, and is generated by $U$--translates of $[c']$.

Suppose that $c\in U^zG$ for some $z\in\Z$.  Since $c$ is a cycle and not homologous to a $U$--translate of $c'$, we see that $z\geq z'$.  Let $c'' = U^{z'-z}c$, so $c''\in G'$, and let $c'' = c^1\otimes c^2$ with $c^1\in G^1$ and $c^2\in U^{z'}G^2$.  Note that
$$0 = \partial c'' = \partial(c^1\otimes c^2) = \partial c^1\otimes c^2 + c^1\otimes\partial c^2.$$
It follows that $c^1$ and $c^2$ are both cycles

By considering Figure \ref{fig:TorsionComplexes}, we see that any cycle in $G^1$ has nonpositive grading.  Furthermore, any cycle in $U^{z'}G^2$ has grading $-2z'$.  Let $c' = c^3\otimes c^4$, where $c^3\in G^1$ and $c^4\in U^{z'}G^2$.  Since $[c']$ is the element of lowest grading in $\mathcal T_{i,j}(Y)$ it follows that $gr(c^3)=0$ and $gr(c^4)=-2z'$.

It follows that $gr(c)\leq gr(c'')\leq gr(c')$, as desired.

\end{proof}

\bibliographystyle{amsalpha}
\bibliography{MasterBibliography.bib}

\providecommand{\bysame}{\leavevmode\hbox to3em{\hrulefill}\thinspace}
\providecommand{\MR}{\relax\ifhmode\unskip\space\fi MR }
\providecommand{\MRhref}[2]{%
  \href{http://www.ams.org/mathscinet-getitem?mr=#1}{#2}
}
\providecommand{\href}[2]{#2}
\begin{thebibliography}{COT03}

\bibitem[AK80]{akbulut-kirby}
Selman Akbulut and Robion Kirby, \emph{Branched covers of surfaces in
  {$4$}-manifolds}, Math. Ann. \textbf{252} (1979/80), no.~2, 111--131.

\bibitem[BB08]{budney-burton:embeddings}
Ryan Budney and Benjamin~A. Burton, \emph{Embeddings of 3--manifolds in ${S}^4$
  from the point of view of the 11--tetrahedron census}, arXiv:0810.2346v5,
  2008.

\bibitem[COT03]{COT}
Tim~D. Cochran, Kent~E. Orr, and Peter Teichner, \emph{Knot concordance,
  {W}hitney towers and {$L^2$}-signatures}, Ann. of Math. (2) \textbf{157}
  (2003), no.~2, 433--519.

\bibitem[Don12]{donald:embedding}
Andrew Donald, \emph{Embedding {S}eifert manifolds in ${S}^4$},
  arXiv:1203.6008, 2012.

\bibitem[Fox62]{fox:problems}
R.~H. Fox, \emph{Some problems in knot theory}, Topology of 3-manifolds and
  related topics ({P}roc. {T}he {U}niv. of {G}eorgia {I}nstitute, 1961),
  Prentice-Hall, Englewood Cliffs, N.J., 1962, pp.~168--176.

\bibitem[FQ90]{freedman-quinn}
Michael~H. Freedman and Frank Quinn, \emph{Topology of 4-manifolds}, Princeton
  Mathematical Series, vol.~39, Princeton University Press, Princeton, NJ,
  1990.

\bibitem[Fre82]{freedman:4manifolds}
Michael~Hartley Freedman, \emph{The topology of four-dimensional manifolds}, J.
  Differential Geom. \textbf{17} (1982), no.~3, 357--453.

\bibitem[Fri04]{friedl:eta}
Stefan Friedl, \emph{Eta invariants as sliceness obstructions and their
  relation to {C}asson-{G}ordon invariants}, Algebr. Geom. Topol. \textbf{4}
  (2004), 893--934.

\bibitem[Gil83]{gilmer:slice}
Patrick~M. Gilmer, \emph{Slice knots in {$S^{3}$}}, Quart. J. Math. Oxford Ser.
  (2) \textbf{34} (1983), no.~135, 305--322.

\bibitem[GL83]{gilmer-livingston:embedding}
Patrick~M. Gilmer and Charles Livingston, \emph{On embedding {$3$}-manifolds in
  {$4$}-space}, Topology \textbf{22} (1983), no.~3, 241--252.

\bibitem[GRS08]{GRS}
J.~Elisenda Grigsby, Daniel Ruberman, and Sa{\v{s}}o Strle, \emph{Knot
  concordance and {H}eegaard {F}loer homology invariants in branched covers},
  Geom. Topol. \textbf{12} (2008), no.~4, 2249--2275.

\bibitem[GS75]{gordon-sumners}
C.~McA. Gordon and D.~W. Sumners, \emph{Knotted ball pairs whose product with
  an interval is unknotted}, Math. Ann. \textbf{217} (1975), no.~1, 47--52.

\bibitem[Han38]{hantzche}
W.~Hantzsche, \emph{Einlagerung von {M}annigfaltigkeiten in euklidische
  {R}\"aume}, Math. Z. \textbf{43} (1938), no.~1, 38--58.

\bibitem[HKL12]{HKL}
Matthew Hedden, Se-Goo Kim, and Charles Livingston, \emph{Topologically slice
  knots of smooth concordance order two}, arXiv:1212.6628, 2012.

\bibitem[HLR12]{HLR}
Matthew Hedden, Charles Livingston, and Daniel Ruberman, \emph{Topologically
  slice knots with nontrivial {A}lexander polynomial}, Adv. Math. \textbf{231}
  (2012), no.~2, 913--939.

\bibitem[JN07]{jabuka-naik:order}
Stanislav Jabuka and Swatee Naik, \emph{Order in the concordance group and
  {H}eegaard {F}loer homology}, Geom. Topol. \textbf{11} (2007), 979--994.

\bibitem[Ker71]{kervaire}
Michel~A. Kervaire, \emph{Knot cobordism in codimension two},
  Manifolds--{A}msterdam 1970 ({P}roc. {N}uffic {S}ummer {S}chool), Lecture
  Notes in Mathematics, Vol. 197, Springer, Berlin, 1971, pp.~83--105.

\bibitem[Kim06]{kim:new}
Taehee Kim, \emph{New obstructions to doubly slicing knots}, Topology
  \textbf{45} (2006), no.~3, 543--566.

\bibitem[KM78]{kirby-melvin}
Robion Kirby and Paul Melvin, \emph{Slice knots and property {${\rm R}$}},
  Invent. Math. \textbf{45} (1978), no.~1, 57--59.

\bibitem[Lev69a]{levine:invariants}
J.~Levine, \emph{Invariants of knot cobordism}, Invent. Math. 8 (1969),
  98--110; addendum, ibid. \textbf{8} (1969), 355.

\bibitem[Lev69b]{levine:groups}
\bysame, \emph{Knot cobordism groups in codimension two}, Comment. Math. Helv.
  \textbf{44} (1969), 229--244.

\bibitem[Lev89]{levine:hyperbolic}
J.~P. Levine, \emph{Metabolic and hyperbolic forms from knot theory}, J. Pure
  Appl. Algebra \textbf{58} (1989), no.~3, 251--260.

\bibitem[Liv05]{livingston:survey}
Charles Livingston, \emph{A survey of classical knot concordance}, Handbook of
  knot theory, Elsevier B. V., Amsterdam, 2005, pp.~319--347.

\bibitem[LS04]{lisca-stipsicz}
Paolo Lisca and Andr{\'a}s~I. Stipsicz, \emph{Ozsv\'ath-{S}zab\'o invariants
  and tight contact three-manifolds. {I}}, Geom. Topol. \textbf{8} (2004),
  925--945 (electronic).

\bibitem[NW10]{ni-wu:cosmetic}
Yi~Ni and Zhongtao Wu, \emph{Cosmetic surgeries on knots in ${S}^3$},
  arXiv:1009.4720v2, 2010.

\bibitem[NW14]{ni-wu:rational}
\bysame, \emph{Heegaard {F}loer correction terms and rational genus bounds},
  Adv. Math. \textbf{267} (2014), 360--380.

\bibitem[OS03a]{oz-sz:absolute}
Peter Ozsv{\'a}th and Zolt{\'a}n Szab{\'o}, \emph{Absolutely graded {F}loer
  homologies and intersection forms for four-manifolds with boundary}, Adv.
  Math. \textbf{173} (2003), no.~2, 179--261.

\bibitem[OS03b]{oz-sz:alternating}
\bysame, \emph{Heegaard {F}loer homology and alternating knots}, Geom. Topol.
  \textbf{7} (2003), 225--254 (electronic).

\bibitem[OS04a]{oz-sz:knots}
\bysame, \emph{Holomorphic disks and knot invariants}, Adv. Math. \textbf{186}
  (2004), no.~1, 58--116.

\bibitem[OS04b]{oz-sz:3-manifolds_1}
\bysame, \emph{Holomorphic disks and topological invariants for closed
  three-manifolds}, Ann. of Math. (2) \textbf{159} (2004), no.~3, 1027--1158.

\bibitem[OS06]{oz-sz:lectures}
\bysame, \emph{Lectures on {H}eegaard {F}loer homology}, Floer homology, gauge
  theory, and low-dimensional topology, Clay Math. Proc., vol.~5, Amer. Math.
  Soc., Providence, RI, 2006, pp.~29--70.

\bibitem[Sto78]{stoltzfus:double}
Neal~W. Stoltzfus, \emph{Algebraic computations of the integral concordance and
  double null concordance group of knots}, Knot theory ({P}roc. {S}em.,
  {P}lans-sur-{B}ex, 1977), Lecture Notes in Math., vol. 685, Springer, Berlin,
  1978, pp.~274--290.

\bibitem[Sum71]{sumners:invertible}
D.~W. Sumners, \emph{Invertible knot cobordisms}, Comment. Math. Helv.
  \textbf{46} (1971), 240--256.

\bibitem[Wal63]{wall:linking}
C.~T.~C. Wall, \emph{Quadratic forms on finite groups, and related topics},
  Topology \textbf{2} (1963), 281--298.

\bibitem[Zee65]{zeeman}
E.~C. Zeeman, \emph{Twisting spun knots}, Trans. Amer. Math. Soc. \textbf{115}
  (1965), 471--495.

\end{thebibliography}

\end{document}